\documentclass[reqno,12pt]{amsart}
\usepackage{amssymb, amsmath,latexsym,amsfonts,amsbsy, amsthm}
\usepackage{color}
\usepackage{setspace}
\usepackage{mathrsfs}
\usepackage{hyperref}
\allowdisplaybreaks[4]
\let\e=\varepsilon

\let\G= \Gamma

\let\ve=\varepsilon

\newcommand{\p}{\partial}
\newcommand{\ee}{\mathbf{e}}
\newcommand{\XX}{\mathbf{X}}
\newcommand{\RR}{\mathscr{R}}
\newcommand{\FF}{\mathbf{F}}
\renewcommand{\(}{\left(}
\renewcommand{\)}{\right)}
\newcommand{\nn}{\mathbf{n}}
\renewcommand{\tilde}{\widetilde}
\newcommand{\Div}{\operatorname{div}}
\newtheorem{theorem}{Theorem}[section]
\newtheorem{definition}[theorem]{Definition}
\newtheorem{lemma}[theorem]{Lemma}
\newtheorem{proposition}[theorem]{Proposition}
\newtheorem{rmk}{Remark}[section]
\newtheorem{corol}[theorem]{Corollary}
\newtheorem{Theorem}{Theorem}[section]

\begin{document}
\title{Phase-field  approximation of the Willmore flow}

\author{Mingwen Fei}
\address{School of  Mathematics and Statistics, Anhui Normal University, Wuhu 241002, China}
\email{mwfei@ahnu.edu.cn}

\author{Yuning Liu }
\address{NYU Shanghai, 1555 Century Avenue, Shanghai 200122, China,
and NYU-ECNU Institute of Mathematical Sciences at NYU Shanghai, 3663 Zhongshan Road North, Shanghai, 200062, China}
\email{yl67@nyu.edu}

%\date{\today}

\renewcommand{\theequation}{\thesection.\arabic{equation}}
\setcounter{equation}{0}
%%%%%%%%%%%%%%%%%%%%%%%%%%%%%%%%%%%%%%%%%%%%%%
%%%%%%%%%%%%%%%%%%%%%%%%%%%%%%%%%%%%%%%%%%
\begin{abstract}
We  investigate the phase-field approximation of the Willmore flow by rigorously justifying  its sharp interface limit. This is a fourth-order diffusion  equation with a parameter $\epsilon>0$ that is proportional to the thickness of the diffuse interface. We show rigorously that for well-prepared initial data, as $\epsilon$ tends to zero  the level-set of the solution will converge to  the motion by  Willmore flow as long as the classical solution to  the latter exists. This is done by  constructing an  approximate solution from the limiting flow via matched asymptotic expansions, and then  estimating its difference with the real solution. The crucial step   is to prove  a  spectrum inequality of the linearized operator at the optimal profile, which  is a fourth-order  operator written as the square of the Allen-Cahn operator plus a singular   perturbation.

%Our approach employs the spectrum  decomposition with respect  to the optimal profile, and such a decomposition  brings  in integrals of   order up to $\e^{-4}$. The controls of these integrals make use of  the separation-of-variables properties of the asymptotic expansions, and the cancellation properties of the related integrals involving the optimal profile.

 ~
\end{abstract}

\maketitle

\numberwithin{equation}{section}

\setcounter{secnumdepth}{3}
\tableofcontents

\indent

\newpage

\section{Introduction}

\indent

In this paper we consider the  singular limit, as $\varepsilon\rightarrow0$, of the  following equations
\begin{eqnarray}
\left \{
\begin {array}{ll}
\varepsilon^3\p_t \phi_{\e}=\varepsilon^2\Delta \mu_\e-f''(\phi_{\e})\mu_\e,\ \quad&\text{in}\ \ \Omega\times(0,T),\\[3pt]
\varepsilon\mu_\e=-\varepsilon^2\Delta \phi_{\e}+f'(\phi_{\e}),\ &\text{in}\ \ \Omega\times(0,T).\label{model}
\end{array}
\right.
\end{eqnarray}
Here $\Omega\subset \mathbb{R}^N$ ($N\geq 2$) is a bounded domain with smooth boundary $\p\Omega$, the parameter  $\varepsilon>0$ represents  the width of a thin transition layer, and    $\phi_{\e}$ is a scalar   function
defined in the computational (physical) domain  $\Omega$ {which takes value approximately $-1$ inside the
vesicle membrane and approximately $1$ outside it.} The function $f$ is a double equal-well potential taking its
global minimum value $0$ at $\pm1$.
For simplicity we choose $ f(u)=\tfrac{1}{4}(u^2-1)^2$. Physically, $-\mu_\e$ represents the chemical potential.
System \eqref{model} is supplemented with initial and boundary conditions
\begin{eqnarray}
\left \{
\begin {array}{ll}
\phi_{\e}(x,0)=\phi_{\e,0}(x),\ \quad&\text{in}\ \ \Omega\times\{0\},\\[3pt]
\phi_{\e}(x,t)=1,~\p_{\nu}\phi_{\e}(x,t)=0,\ &\text{in}\ \ \p\Omega\times(0,T),\label{BC}
\end{array}
\right.
\end{eqnarray}
where $\p_\nu$ is the outward normal derivative to $\p\Omega$.
 The equation \eqref{model} is the gradient flow of
$$\mathcal{E}(\phi)=\frac{1}{2\varepsilon}\int_{\Omega}
\(\varepsilon\Delta\phi-\e^{-1}f'(\phi)\)^2dx,$$which
  was introduced in \cite{bellettini1993approssimazione} as an approximation of the Willmore energy $\int_{\Sigma}H^2$  where $H$ denotes the mean curvature of a surface $\Sigma$. Such an approximation can be understood formally by observing that
$\mu(\phi)=-\varepsilon\Delta\phi+\e^{-1}f'(\phi)$
 is the variation of the Allen-Cahn energy
$\int_{\Omega}\(\e |\nabla\phi|^2/2+\e^{-1} f(\phi)  \)dx$,
  and the latter converges to the area functional as $\e\to 0$. The rigourous justification is a challenging work, see for instance \cite{MR2253464} for the case $N=2,3$. See  also \cite{MR3383330} for   other types of approximations.

It was shown
informally  in \cite{MR1757511,wang2008asymptotic}     that the level set of solution $\phi_{\e}$ to \eqref{model} converges to a family of compact closed interfaces \begin{equation}\label{involvingsurface}
   \G^0=\bigcup_{t\in [0,T)}\G_t^0\times \{t\},
 \end{equation}
 which evolves under the Willmore flow, i.e.
 \begin{align}
V=\Delta_{\Gamma^0} H+H |A|^2- H^3/2,\label{interface law-geometric1}
\end{align}
where $V$ is the normal velocity of $\Gamma^0$, i.e. projection of the particle velocity to the outer-normal vector of $\G^0$, $H$ is the scalar mean curvature.
%\footnote{With these conventions, a shrinking sphere will have $V<0$ and $H>0$.}
Moreover, $A$ is the second fundamental form, and $|A|^2$ is the sum of the squares of the principle curvatures.  {An equivalent form of \eqref{interface law-geometric1}
 using the signed distance function to $\G^0$ is given by \eqref{distance law} in the appendix.}
The derivation of \eqref{interface law-geometric1} from the Willmore energy can be found in  \cite{MR2882586}.
The evolution of a  closed curve  under (\ref{interface law-geometric1}) was investigated in \cite{MR1897710}, and the  surface case  was investigated in \cite{MR2119722}.  See also \cite{MR2430975} for  recent progresses in Willmore surface theory.

At the application level, Willmore energy is closely related to the Helfrich's theory of  bio-membranes \cite{helfrich1973elastic}, where the volume and area are constrained, and a spontaneous curvature is introduced. See \cite{seifert1997configurations,ou1999geometric} for a comprehensive study of this subject. So various  phase-field models were introduced to study the evolution of bio-membranes, see for instance \cite{MR2062909,MR2531189,MR2846014,MR3023428} for the  analysis and numerical simulations of such models. The model that couples hydrodynamics was investigated in \cite{MR2328722,MR3032974}.

\subsection{Statement of Main Results}
To the best of our knowledge, the  rigorous justification of the  singular limit from \eqref{model} to \eqref{interface law-geometric1}  remains open and it is the task of this work to establish this result.
To set up the problem, we introduce   the signed-distance $d^{(0)}(x,t)$ from $x$ to $\Gamma^0_t$ which takes negative values inside $\G^0_t$ and positive values outside it,  and the  bulk regions
 \begin{equation}\label{def:omegapm}
{\Omega_t^{\pm}\triangleq \{x\in\Omega\mid d^{(0)}(x,t)\gtrless0\}.}
\end{equation}
For a sufficiently small $\delta>0$,  the $\delta$-neighborhood of $\G^0_t$ is denoted by 
\begin{equation}
\Gamma_t^0(\delta)\triangleq \{x\in\Omega\mid |d^{(0)}(x,t)|<\delta\}.
\end{equation}
Moreover, we denote
$\theta(z)=\tanh(\frac{z}{\sqrt{2}})$
   the unique solution to
\begin{align}
\theta''(z)=f'(\theta(z)),\ \ \theta(\pm\infty)=\pm1,\ \ \ \theta(0)=0.\label{ODE-0}
\end{align}

The main result of this work is stated as follows:
\begin{Theorem}\label{main theorem}
Let $\G^0\subset \mathbb{R}^N\times [0,T)$ (with $N=2,3$) be a smooth solution of \eqref{interface law-geometric1}  and  $k\geq 10$ be a fixed integer. {Then there exists $\e_1>0$ such that for every $\e\in (0,\e_1)$}  there are smooth functions {$(\phi_a(x,t),\mu_a(x,t))$ which equal to $(\pm 1,0)$ in $\Omega_t^\pm\backslash \G^0_t(2\delta)$}.    Moreover, they are  approximate solutions   of  \eqref{model}  in the sense that  
\begin{eqnarray}
\left \{
\begin {array}{rll}
\e^3\p_t\phi_a&= \varepsilon^{2}\Delta \mu_a-f''(\phi_a)\mu_a+\e^{k-1}\mathfrak{R}_1,\ \quad&\text{in}\ \ \Omega\times(0,T),\\
\e\mu_a&=-\varepsilon^2\Delta \phi_a+ f'(\phi_a)+\e^{k-1}\mathfrak{R}_2,\ &\text{in}\ \ \Omega\times(0,T),\label{app model}
\end{array}
\right.
\end{eqnarray} where $\mathfrak{R}_1(x,t)$ and $\mathfrak{R}_2(x,t)$  are uniformly bounded  in $\e$ and $(x,t)$. Furthermore, if the initial data of $\phi_\e$ satisfies the following inequality for some   $C_{in}>0$,
\begin{align}\|\phi_\e(\cdot,0)-\phi_a(\cdot,0)\|_{L^2(\Omega)} \leq C_{in}\,\varepsilon^{7/2},\qquad \forall \e\in {(0,\e_1)},
\label{initial error}\end{align}
  then  there exist    $\e$-independent constant { $\Lambda>0$} and $T_{max}\in (0,T]$  such that
\begin{align}
{ \|\phi_\e-\phi_a\|_{L^\infty(0,T_{max};L^2(\Omega))}\leq  \Lambda \,\e^{7/2},\qquad \forall \e\in (0,\e_1)}.
\label{time error}
\end{align}
If $C_{in}$  is sufficiently small (but independent of $\e$), then $T_{max}=T$.
\end{Theorem}
 
% The second part of the theorem can be stated alternatively  by assuming an arbitrary $C_{in}>0$ at the expense of having only  short time convergence in \eqref{time error}, i.e., \eqref{time error} will be valid only over   $[0,T_{\max}]$ for some $T_{\max}\in (0,T)$, which is $\e$-independent. 

	 When $N=3$
  it is not clear to us whether  the power of the  convergence rate  $7/2$  is   optimal.  A detailed  discussion is given in the next subsection where an outline of the proof is described. We shall discuss the admissible initial data $\phi_\e|_{t=0}$ satisfying the  assumption \eqref{initial error} in Remark \ref{verification} below. Such issue seems not be fully addressed in previous works. For instance in \cite{Alikakos1994}, they  only consider the case when the initial data of the Cahn-Hilliard equation  coincides with the constructed approximate solution. When $N=2$, there are  some  rooms to reduce the power and thus obtain the convergence over $[0,T]$ for  arbitrarily large $C_{in}$ in \eqref{initial error}.  We will probably study this case more carefully in a future work.

  \subsection{Outline  of the Proof}

The proof of \eqref{app model} follows the strategy of \cite{Alikakos1994} by constructing an approximate solution  $(\phi_a,\mu_a)$ via matched asymptotic expansions. Since such constructions are  quite sophisticated and technical, we leave them in the appendix. Another  method which is  based on Hilbert expansion is given in \cite{MR2188465}.

In order to show the convergence rate \eqref{time error} under the  assumption  \eqref{initial error}, we shall investigate the equation of  the differences 
\begin{equation}
(\phi,\mu)=(\phi_{\e}-\phi_a,\mu_\e-\mu_a).
\end{equation}
To proceed, we introduce the linearized   Allen-Cahn operator  at $\phi_a$ by
\begin{equation}\label{omegadef}
  \mathscr{L}_\e[\phi]\triangleq\(-\varepsilon\Delta +\e^{-1}f''(\phi_a)\)\phi.
\end{equation}
  In view of  \eqref{app model}
 and \eqref{model},  it can be verified   that   $\phi$ satisfies
\begin{align}
\p_t\phi=-\frac 1{\e^2}\mathscr{L}_\e^2[\phi]
-\frac{1}{\varepsilon^3}f'''(\phi_a)\mu_a\phi+\mathcal{H},\label{error equation}
\end{align}
where $\mathcal{H}$ is the nonlinear terms about $\phi$, defined by \eqref{nonlinearterms} in the sequel.
 Standard energy estimate of (\ref{error equation})   yields
\begin{align}
&\frac{1}{2}\frac{d}{dt}\int_\Omega\phi^2dx+
\frac{1}{\varepsilon^2}
\int_\Omega|\mathscr{L}_\e[\phi]|^2dx+\frac{1}{\varepsilon^3}\int_\Omega f'''(\phi_a)\mu_a\phi^2\,dx=\int_\Omega\mathcal{H}\phi \,dx.\label{error energy-1}
\end{align}
In order to apply the Gr\"{o}nwall's inequality and close the energy estimate \eqref{error energy-1}
, we  first need   to prove  the   spectrum condition of the linear operator on the right of \eqref{error equation}. This fourth-order  operator is highly singular due to both the degeneracy of the Allen-Cahn part \eqref{omegadef} and the singular perturbation $- \varepsilon^{-3}f'''(\phi_a)\mu_a\phi$.  
To analyze it, we employ the spectrum decomposition method used  in \cite{chen1994spectrum} (see also \cite{de1995geometrical,fei2018isotropic}). We set the following operator
\begin{equation}\label{1doperator}
\mathscr{L}[u]\triangleq-\partial_z^2 u+f''(\theta(z)) u,\qquad \forall u(z)\in H^2(-\delta/\e,\delta/\e).
\end{equation} Then we decompose  any $\phi\in H^2(\Omega)$ (not necessarily the difference $\phi_{\e}-\phi_a$) along  $\varphi$, the  first eigenfunction of $\mathscr{L}$ with  homogenous Neumann boundary condition,
\begin{equation}\label{eq:decom1}
\phi(x)=
\varepsilon^{-\frac{1}{2}}Z(s)\varphi(\tfrac{r}{\varepsilon})\zeta(r)+
\phi^\bot(x).
\end{equation}
Here $s=(s_1,\cdots,s_{N-1})$ is the local coordinate of  interface $\G_t$ (a perturbation of $\G_t^0$ in \eqref{involvingsurface}),  and  $r$  stands for the signed-distance to $\G_t$. With an appropriate $\delta>0$, $(r,s)$   forms a local coordinate of the $\delta$-tubular neighborhood of $\G_t$ and are functions of $x$ here. Moreover  in \eqref{1doperator} $z=r/\e$  is the fast variable.
See Section \ref{gemmetry} for the precise definitions of $\G_t$ and related notations.
In \eqref{eq:decom1},
$\zeta$ is a smooth  cut-off function satisfying
\begin{equation}\label{cutoff}
0\leq \zeta\leq 1;~\zeta(r)=\zeta(-r);~\zeta(r)=1 ~\text{for}~|r|\leq\delta/2;~\zeta\in C_c^\infty((-\delta,\delta)).
\end{equation}
Moreover, $Z(s)$ is defined  by
\begin{align}
Z(s)=\varepsilon^{-\frac{1}{2}}\eta^{-1}(s)\int_{-\delta}^{\delta}\phi(r,s)\varphi(\tfrac{r}{\varepsilon})\zeta(r)J^{\frac{1}{2}}(r,s)dr,\label{eq:defz}
\end{align}
where $J(r,s)$ is the Jacobian of the change of variable among level-sets, and $\eta(s)$ is the renormalization constant
\begin{equation}\label{apptension}
  \eta(s)=\varepsilon^{-1}\int_{-\delta}^{\delta}
  \big(\varphi(\tfrac{r}{\varepsilon})\big)^2\zeta^2(r)J^{\frac{1}{2}}
  (r,s)dr.
\end{equation}
With the above definitions, we can state the spectrum condition of the linearized operator:
\begin{Theorem}\label{spectral theorem}
   Let $(\phi_a,\mu_a)$ be the approximate solution in Theorem \ref{main theorem}. Then
there exist  $\tilde{C},\hat{C}>0$ and $\e_0>0$ which are independent of $\e$ such that the inequality
{\begin{subequations}
\begin{align}
&\frac{1}{\varepsilon^2}\int_\Omega|\mathscr{L}_\e[\phi]|^2dx+
\frac{1}{\varepsilon^3}\int_\Omega f'''(\phi_a)\mu_a\phi^2dx \geq 
\tilde{C} \mathcal{K}(t)-\hat{C}\int_{\Omega}\phi^2dx\label{spectral condition}\\
&\quad \text{  with  }\qquad \mathcal{K}(t)\triangleq \|Z\|^2_{H^2(\G_t)}+ \|\phi^\bot\|_{H^2(\Omega)}^2+\e^{-2}\|\phi^\bot\|^2_{H^1(\Omega)}+ \e^{-4}\|\phi^\bot\|^2_{L^2(\Omega)},\label{def K(t)}
\end{align}
\end{subequations}}
holds for every  $\varepsilon\in (0,\e_0)$ and every $\phi\in H^2(\Omega)$.
\end{Theorem}
The above Theorem is proved under any dimension $N\geq 2$, and the constant $\hat{C}$ and $\tilde{C}$ depend on the geometry of the limiting interface \eqref{involvingsurface}.
Let us  outline  the proof of \eqref{spectral condition} by explaining the general consideration behind \eqref{eq:decom1}. It follows from \eqref{eq:decom1}, \eqref{eq:defz} and \eqref{apptension}  that  
  \begin{align}
\int_{-\delta}^{\delta}\phi^\bot(r,s)
\varphi(\tfrac{r}{\varepsilon})\zeta(r)J^{\frac{1}{2}}(r,s)dr=0.
\label{orthogonality}
\end{align}
This implies \eqref{eq:decom1} is a sort of orthogonal decomposition.
Then substituting \eqref{eq:decom1} into the left hand side of \eqref{spectral condition}   will lead to  integrals $I_1,I_3$ involving the squares of  $\phi^\bot$ and $\check{\phi}\triangleq \varepsilon^{-\frac{1}{2}}Z(s)\varphi(\tfrac{r}
   {\varepsilon})\zeta(r)$ respectively, 
 together with an integral $I_2$  about their product $\check{\phi}\phi^\bot$.  See \eqref{nasty} for the precise definitions.  In the course of estimating these three integrals,  the one corresponding to   $\check{\phi}\phi^\bot$ vanishes at the leading order  due to  \eqref{orthogonality}.  The part including merely $\check{\phi}$  carries the most singular  terms  but is very explicit because $\varphi(z)$ is very close to $\theta'(z)$ such that several cancellation techniques  can be applied. In contrast,  $\phi^\bot$ is more abstract but less singular because of a coercivity estimate of \eqref{omegadef}. We refer to Lemma \ref{spectral lemma} and Lemma \ref{equinorm} for the mathematical  rigor of these heuristic arguments.  

%So   \eqref{spectral condition}   together with  the estimates of the nonlinear  terms $\mathcal{H}$ allow us to apply Gr\"{o}nwall's inequality and to deduce  the preservation of smallness of $\phi$. Altogether for sufficiently small $\e$,  $\phi_\e$ behaviors like $\phi_a$ whose  leading order profile is given by $\theta(\tfrac{d^{(0)}(x,t)}\e)$.

\begin{rmk}
In the Cahn-Hilliard case \cite{Alikakos1994,chen1994spectrum}, the equation of  the difference $\phi$ reads
\begin{align}
\partial_t\phi=\Delta\mathscr{L}_\e[\phi]+\text{nonlinear terms}.
\end{align}
Multiplying by $w\triangleq(-\Delta)^{-1}\phi$ and integrating by parts lead to
\begin{align*}
\frac{1}{2}\frac{d}{dt}\int_\Omega|\nabla w|^2dx+\int_\Omega\big(\varepsilon|\nabla \phi|^2+\varepsilon^{-1}f''(\phi_a)\phi^2\big)dx=\int_\Omega(\text{nonlinear terms}\times\phi) \,dx.
\end{align*}
It is proved in \cite{chen1994spectrum} that, there exist $C,\e_0>0$ independent of $\e$ such that
\begin{align}
\int_\Omega\big(\varepsilon|\nabla \phi|^2+\varepsilon^{-1}f''(\phi_a)\phi^2\big)dx\geq -C\int_\Omega|\nabla w|^2dx\label{chen}
\end{align} holds 
for any $\e\in (0,\e_0)$.
The proof heavily relies  on the spectrum condition  of the Allen-Cahn operator \eqref{omegadef}. Combining \eqref{chen} and the estimates of the nonlinear terms, Alikakos-Bates-Chen \cite{Alikakos1994} succeeded in constructing the approximate solutions and  verifying the smallness of $\phi=\phi_\e-\phi_a$.
In contrast, in our case  the difference $\phi$ satisfies \eqref{error equation}, i.e.,
\begin{align}
\partial_t\phi=\varepsilon^{-1}\Delta\mathscr{L}_\e[\phi]-\varepsilon^{-3}f''(\phi_a)\mathscr{L}_\e[\phi]-\varepsilon^{-3}f'''(\phi_a)\mu_a\phi+\text{nonlinear terms}.
\end{align}
The structure  seems not allow a test using $(-\Delta)^{-1}\phi$. Instead  we multiply by $\phi$ to get \eqref{error energy-1}, and then  investigate the  spectrum of    a fourth-order  operator. 

It is worth mentioning that there are other frameworks to  justify the small $\e$-limit of the Cahn-Hilliard equation, see for instance \cite{MR1425577,MR2402921}. It would be interesting  to study \eqref{model} through these frameworks.
\end{rmk}

Having \eqref{spectral condition},  the proof of Theorem \ref{main theorem} relies on the estimate of the integral  on the right hand side of \eqref{error energy-1}.  In view of \eqref{nonlinearterms},      the   integral corresponding to $\mathcal{H}_2$   can be  estimated by choosing a sufficiently large $k$. 
The $\mathcal{H}_1$ term  looks quite singular and nonlinear, and is the source of the major difficulties. 
However,   nice structures emerge  after some integration by parts: some highly singular terms in $\int_{\Omega}\mathcal{H}_1\phi\,dx$ turn out to be non-positive and can be disregarded in the energy estimate. One can refer to  the derivation of \eqref{H1H2sumb} for the details.
Moreover, applying the decomposition \eqref{eq:decom1} to the remaining terms allows  us to  treat  separately   the tangential derivatives and normal ones of $\check{\phi}$. This is crucial  towards  a reasonable convergence rate, as  each normal derivative of $\check{\phi}$ brings a factor of $\e^{-1}$.  We note that similar idea  has been employed in the study of free-boundary problems in hydrodynamics where such estimates are referred to as {\it conormal estimates}. See for instance the monograph \cite{MR2151414} and a recent article \cite{MR3590375}.
These combined with \eqref{spectral condition}
lead to an estimate like 
\begin{equation}\label{heuristic1}
\text{right hand side of \eqref{spectral condition}
}+\int_{\Omega}\mathcal{H}\phi\,dx\leq C_*\e^{-\beta}\|\phi\|_{L^2(\Omega)}^\gamma
\end{equation}
for some $\beta>0, \gamma>2$ and $C_*>0$ which are $\e$-independent. Substituting \eqref{heuristic1} and \eqref{spectral condition} into \eqref{error energy-1} and combining the Gr\"{o}nwall inequality with a continuity argument lead to the proof of \eqref{time error}  under the assumption  \eqref{initial error}. See Section \ref{section:estimate} for the details.

%Since the Sobolev embedding is employed repeatedly, the convergence rate \eqref{time error} depends on $N$. So we restrict ourselves to  lower dimensions $N=2,3$, which meet application purposes. 

In seeking  an estimate  like    \eqref{time error}  with $L^\infty(\Omega\times (0,T))$ norm instead,  one can follow the approach of 
 \cite{fei2018isotropic}  by establishing a hierarchical Sobolev norm estimates. However this will lead to a larger power  than $7/2$ and makes \eqref{initial error} harder to verify.

We organize this paper as follows. In Section \ref{prelimiary}, we set up the geometry of the interface using the standard notations of differentiable manifolds. This setting is slightly different from the one employed in \cite{abels2016sharp,chen1994spectrum,chen2009mass}, and has the advantage of  clarifying the various analyses we shall perform over hypersurfaces. In Section \ref{approaximation}, we give the construction of approximate solution \eqref{app model} and then set up the proof of   \eqref{spectral condition} by reducing it  to the estimates of three integrals $\{I_\ell\}_{1\leq \ell\leq  3}$ in \eqref{nasty}. Loosely speaking we substitute \eqref{eq:decom1} into the left hand side of \eqref{spectral condition}, and this gives rise to the energy  term $I_1$ depending  on $Z^2(s)$, the energy term $I_3$ depending   on $(\phi^\bot)^2$, as well as a cross term $I_2$ depending on $Z(s)\phi^\bot$. The estimate of $I_1$ is done in Section \ref{section:spectrum} after establishing  some  cancellation lemmas. Such techniques bring certain novelties  to the existing literature, to the best of our knowledge.  They tame   the singular terms in $I_1$ which are up to order $\e^{-4}$. In Section \ref{section:spectumI2}, we estimate the cross term $I_2$ using the orthogonality condition \eqref{orthogonality}, together with some geometric analysis on the hypersurface. In Section \ref{section:spectrumI3}, we estimate the energy term $I_3$ corresponding to $\phi^\bot$, and thus complete the proof of Theorem \ref{spectral theorem}. In Section \ref{section:estimate}  we  give the proof of Theorem \ref{main theorem}.

\indent

\section{Preliminaries}\label{prelimiary}
%%%%%%%%%%%%%%%%%%%%%%%%%%%%%%%%%%%%%%%%%%%%%%%%%%%%%%%%%%%%%%

\subsection{Level set and signed-distance function}\label{levelset}
Following \cite{Alikakos1994}, we shall approximate the perturbed level set $\Gamma^\varepsilon=\{(x,t):\phi_{\e}(x,t)=0\}$ for the solution of \eqref{model}.
 For any $t\in[0,T]$,  $d_\e(x,t)$ will denote  the signed-distance from $x$ to $\Gamma^\varepsilon$, taking negative values inside $\G^\e$. For the analytic properties of the signed-distance function, one can refer to \cite{MR3155251}. Among those, we shall use that $d_\e$ is smooth,  and $|\nabla d_\e|=1$ in a neighborhood of  $\Gamma^\varepsilon$. We  write  $d_\e$ and its $k$-th truncation $d^{[k]}$  by
\begin{align}
d_\e(x,t)&=\sum_{i\geq 0}\varepsilon^i d^{(i)}(x,t),\quad d^{[k]}(x,t)=\sum_{0\leq i\leq k}\varepsilon^i d^{(i)}(x,t),\label{d}
\end{align}
respectively. If we  assume   $\G^0$ to  be   the limit of $\G^\e$ and  denote   $d^{(0)}$ as the signed-distance function of $\G^0$, we shall have  $|\nabla d^{(0)}|=1$ in $\G^\e(4\delta)$ for some $\delta>0$. This implies  
\begin{align*}
1=|\nabla d_{\varepsilon}|^2=1+2\varepsilon\nabla d^{(0)}\cdot\nabla d^{(1)}+\sum _{\ell\geq 2}\varepsilon^\ell\bigg(\sum\limits_{0\leq i\leq \ell}\nabla d^{(i)}\cdot\nabla d^{(\ell-i)}\bigg).
\end{align*}
Matching  the order of $\varepsilon$ yields the following equations
\begin{eqnarray}
\nabla d^{(0)}\cdot\nabla d^{(1)}=0;\quad \nabla d^{(0)}\cdot\nabla d^{(\ell)}=-\tfrac{1}{2}\sum\limits_{1\leq i\leq \ell-1}\nabla d^{(i)}\cdot\nabla d^{(\ell-i)},\quad\forall \ell\geq2.\label{equ:gradient d}
\end{eqnarray}
So $d^{[k]}$  satisfies the eikonal equation $|\nabla d^{[k]}|=1+\e^k g(x,t)$
for some   smooth  function  $g(x,t)$. In general $d^{[k]}$  does not   satisfy $|\nabla d^{[k]}|=1$, and this would lead to some complications in the asymptotic  expansions.  The remedy is to define  
% the  $0$-level set of $d^{[k]}$  by
 \begin{equation}\label{nasty2}
 \Gamma_t=\{x\in\Omega:d^{[k]}(x,t)=0\},\quad \G=\bigcup_{t\in [0,T]} \G_t\times \{t\}
 \end{equation}
 for a sufficiently large $k$,
and to work with  the signed-distance $r(x,t)$ of $\Gamma_t$. Through $r(x,t)$   a coordinate system near $\Gamma_t$ will be established, in which the asymptotic expansion  will be performed.
The following estimate can be proved via  the method of characteristics
 \begin{align}
\big\|r(x,t)-d^{[k]}(x,t)\big\|_{C^3(\Gamma_t(2\delta))}
=O(\varepsilon^{k}).\label{app distance}
\end{align}
See for instance \cite[Section 3.2]{Evans2010PDE}  or \cite[Section 5]{Alikakos1994}.
%Actually, according to \eqref{eikonal},  under local coordinates $(r,s)$, %$d^{[k]}$ is a solution to the eikonal  equation
%\begin{equation}\label{eikonal1}
 % \left\{
  %\begin{array}{rll}
 %   |\nabla_\G  d^{[k]}|^2+(\p_r d^{[k]})^2&= 1+\e^k g(x,t)~&\text{in}%~(-2\delta,2\delta)\times\G_t,\\
 %   d^{[k]}(0,s,t)&=0,~&\text{on}~\G_t,
 % \end{array}
  %\right.
%\end{equation}
%whose solvability follows from  characteristic method for small $\delta$, see for instance \cite[Section 3.2]{Evans2010PDE}.
%Since $r=r(x,t)$ satisfies   \eqref{eikonal1} with $g\equiv 0$, the assertion  \eqref{app distance} follows from energy estimate of equation for $d^{[k]}(x,t)-r(x,t)$.

\subsection{Geometry of the interface}\label{gemmetry}
This part  will be used for the proof of \eqref{spectral condition}, where $t$ is a fixed parameter. So we shall suppress  its dependence for various quantities. For instance, we shall simply write the signed-distance function $r(x,t)$ of $\G_t$ as $r(x)$ or $r$. Following  \cite{de1995geometrical},
we assume $\Gamma_t$ be a compact smooth hypersurface and  $\XX_0(s):U\mapsto \Gamma_t$ be  a local parametrization of it with $U$ being an open set in $ \mathbb{R}^{N-1}$.
For a sufficiently small $\delta>0$, the $\delta$-neighborhood of $\Gamma_t$ is well-defined and shall be  denoted by $\Gamma_t(\delta)$, and the projection  $s(x):\Gamma_t(\delta)\mapsto U$ is  a smooth function. So each point $x\in \Gamma_t(\delta)$ allows a unique expression 
\begin{equation}\label{parametrization}
  x=\XX  (s,r)=\XX _0 (s)+r \nn(s),
\end{equation}
 where $s=(s_1,\cdots,s_{N-1})$ are local coordinates of $\Gamma_t$,  and $\nn$ is the unit outward normal vector.
  The decomposition \eqref{parametrization} leads to a local parametrization of $\Gamma_t(\delta)$, and  induces a local coordinates $(s,r)$ in it. The corresponding tangent space is expanded by
 \begin{equation}\label{basis}
 \{\ee_1,\cdots, \ee_{N}\}\triangleq \left\{ {\p_{s_1} \XX },\cdots,  {\p_{s_{N-1}} \XX },\nn\right\}.
 \end{equation}
 We note that  the first $N-1$ components  span the tangent plane of the hypersurface $\Gamma_t^r$,
 the $r$-level set of the signed-distance $r(x)$
\begin{equation}\label{r'levelset}
\G_t^{r}\triangleq \{ x\in\mathbb{R}^N\mid r(x)=r\}.
\end{equation}
%Note that \eqref{basis}  do not necessarily form an orthonormal  basis.
 % If $\Gamma_t$ is regular, the inverse of \eqref{parametrization} can be written by
% \begin{equation*}
%\begin{split}
 %    r(x)&=r(\XX _0 (s)+r \nn(s)),\qquad
  %   s_i(x)=s_i(\XX _0 (s)+r \nn(s)),\qquad \text{for}~ 1\leq i\leq N-1.
%\end{split}
%\end{equation*}
%Differentiating the above two formulas in $r$ leads to
%\begin{equation}\label{orthogonality1}
%  \begin{split}
%    \nabla r(x)=\nn (s),\qquad
%    \nabla s_i(x)\cdot\nabla r=0,\qquad \text{for}~1\leq i\leq N-1.
%  \end{split}
%\end{equation}
It follows from \eqref{parametrization} that ${\p_{s_i} \XX } \cdot  {\p_r \XX }=0$ for
$1\leq i\leq N-1$.
This implies that the  metric $g_{ij}$ of $\G_t(\delta)$ under $(s,r)$  writes
\begin{equation}\label{metric}
  (g_{ij})_{1\leq i,j\leq N} =\begin{pmatrix}
    (  {\p_{s_i} \XX }\cdot {\p_{s_j} \XX })_{1\leq i,j\leq N-1}& 0\\
    0& 1
  \end{pmatrix}.
\end{equation}
Conventionally, we denote  its inverse by $g^{ij}$.
Using \eqref{metric} we can decompose the Euclidean   gradient operator  by
 \begin{equation}\label{gradient}
   \nabla f= \nabla_{\G} f+\nn\p_rf\qquad\text{where}~\nabla_{\G}f\triangleq \sum_{1\leq i,j\leq N-1}g^{ij}\p_{s_j} f \ee_i
    \end{equation}
 is the gradient on $\G_t^r$.   Denoting $g=\det g_{ij}$, we can assume  without loss of generality  that
\begin{equation}\label{boundjacobian}
1/\Lambda\leq \sqrt{g}(r,s)\leq \Lambda,
\end{equation}
for some $\Lambda>0$ which only depends on $\delta$ and $\G_t$. Following \cite{MR3618120}, there holds
 \begin{equation}\label{eq:laplace1}
 \Delta f= \Delta_{\G}f+\tfrac 1{\sqrt{g}}
   \p_{r}\(\sqrt{g} \p_{r} f\)\qquad \text{where}~\Delta_{\G}f\triangleq \sum_{1\leq i,j\leq N-1}\tfrac 1{\sqrt{g}} \p_{s_i}\(\sqrt{g} g^{ij}\p_{s_j} f\)
\end{equation}
  is the   Laplace-Beltrami operators  on $\G_t^r$. The operator   $\Delta_{\G}\big|_{r=0}$ corresponds to the   Laplace-Beltrami operator on $\Gamma_t$. Its  difference to $\Delta_\G$ is measured by  a second order differential operator with smooth coefficients
 \begin{equation}\label{roperator}
 \RR_{\G}=\( \Delta_{\G}-\Delta_{\G}\big|_{r=0}\)/r.
 \end{equation}
Due to  elliptic regularity theory, one has the following estimates:
\begin{proposition}
Let $\Gamma_t$ be a compact smooth hypersurface, then \begin{subequations}
\begin{equation}\label{laplace1}
\sup_{|r|\leq \delta}\|\Delta_{\G} f\|_{L^2(\Gamma_t)}+\|  f\|_{L^2(\Gamma_t)}\lesssim  \|f\|_{H^2(\Gamma_t)}\lesssim \inf_{|r|\leq \delta}\|\Delta_{\G} f\|_{L^2(\Gamma_t)}+\|  f\|_{L^2(\Gamma_t)}.\end{equation}
\begin{equation}\label{roperatorest}
   \|\RR_{\G}f\|_{L^2(\Gamma_t)}\lesssim \|f\|_{H^2(\Gamma_t)}\lesssim \big\|\Delta_{\G}\big|_{r=0} f\big\|_{L^2(\Gamma_t)}+\|  f\|_{L^2(\Gamma_t)}\end{equation}
\end{subequations}
where the symbol $\lesssim$ is understood in the conventional sense in  Section \ref{convention1}.
\end{proposition}
Now we consider the divergence operator in $\Gamma_t(\delta)$.  With
the convention  $s_N=r$,  any vector-field $\FF=\sum_{1\leq i\leq N}F_i\frac \p{\p  x_i }$ can be written under local coordinates $(s,r)$ by
$\FF=\sum_{1\leq j\leq N} \tilde{F}_j\ee_j$ where $\tilde{F}_j=\sum_{1\leq k\leq N}g^{kj}\ee_k\cdot\FF.$
 Then the Euclidean  divergence $\Div \FF=\sum_{1\leq j\leq N}\tfrac 1{\sqrt{g}} \p_{s_j}\(\sqrt{g} \tilde{F}_j\)$ is decomposed  by
\begin{equation}\label{divergence1}
\Div \FF=\Div_{\G} \FF+ \tfrac 1{\sqrt{g}}  \p_r\(\sqrt{g}  \tilde{F}_N\),
\end{equation}
%\footnote{\color{red}For instance, $\Div \nn= \p_r\ln \sqrt{g} $}
where $\Div_{\G}$ is the   divergence operator on $\G_t^r$ \begin{equation}\label{divergence}
\Div_{\G} \FF=\sum_{1\leq j\leq N-1}\tfrac 1{\sqrt{g}} \p_{s_j}\(\sqrt{g} \tilde{F}_j\).\end{equation}
With these notations, it can be verified that
 $ \Delta_{\G}f=\Div_{\G}\nabla_{\G}f,$ consisting with \eqref{eq:laplace1}. Following \cite[Lemma 4]{chen2009mass}, we have the following expansion of   $J(r,s)=\frac{\sqrt{g}(r,s)}{\sqrt{g}(0,s)}$:\begin{equation}\label{jacobian}
J(r,s)=\prod_{1\leq i\leq N-1}(1+r\kappa_i(s))=1+rh(s)+r^2e(s)+O(r^3),
\end{equation}
where $\kappa_i$ are the  principal curvatures of $\G_t$. 
On the other hand, it  follows from \eqref{divergence1} and $\nabla r=\nn$ that $\p_r \ln \sqrt{g}=\Delta r$.  When combined with \eqref{jacobian}, we obtain 
 \begin{subequations}\label{eq:2.1}
\begin{align}
\p_r \ln \sqrt{g}&=\Delta r=h(s)+ b(s) r+a(s) r^2+ r^3O(1),\\
\partial_r^2\ln \sqrt{g}&=\nabla\Delta r\cdot\nabla r= b(s)+c(s) r+ r^2O(1).
\end{align}
\end{subequations}
We end this section by the   integration theory on hypersurfaces.
By coarea formula \cite{MR3409135},
\begin{equation}\label{coareaformula}
\int_{\Gamma_t(\delta)} f(x) dx=\int_{-\delta}^\delta
\(\int_{\G_t^r} f \,d\mathcal{H}^{N-1}\)dr,
\end{equation}
where
 $\mathcal{H}^{N-1}$ is the Hausdorff measure on $\G_t^r$. Under  local coordinates, $d\mathcal{H}^{N-1}=\sqrt{g}(r,s) ds$. We  also need the following formula
\begin{equation}\label{fubini}
\int_{\Gamma_t(\delta)} f(x) dx=\int_{\Gamma_t}\(\int_{-\delta}^\delta f(r,s) J(r,s)dr\)d\mathcal{H}^{N-1},
\end{equation}
where  $\mathcal{H}^{N-1}$ is the Hausdorff measure on $\G_t$, which writes  $d\mathcal{H}^{N-1}=\sqrt{g}(0,s) ds$ under  local coordinates.
Finally we state   the divergence theorem for tangential vector field:  
\begin{equation}\label{gauss}
\int_{\Gamma_t^r} \Div_{\G}   {\FF} \,d\mathcal{H}^{N-1}=0\quad \text{if}~\FF\cdot\nn=0,
\end{equation}
where   $\mathcal{H}^{N-1}$ is interpreted in the same way as in \eqref{coareaformula}.
See for instance \cite{MR1301070} for the proof of these formulas.

\subsection{Conventions}\label{convention1}
Throughout this work, $O(\e^\ell)$ will denote $\e^\ell g_\e(x,t)$ for some  $g_\e(x,t)$ that is uniformly bounded in $(x,t)$ and $\e$. If there exists a  constant $C$ which is independent of $\e$ so that the  inequality $X\leq CY$ holds, we shall briefly write $X\lesssim Y$. Similarly, $X\gtrsim Y$ means $X\geq CY$. In various estimates, the generic constant $C$ might change from line to line and we shall not relabel them for simplicity. 

For brevity, we shall not relabel  a function under equivalent  coordinates. For instance, for a function $f=f(x)$, its counterpart under local coordinates $(r,s)$ will be simply denoted by $f(r,s)$. Its integration in $\G_t(\delta)$ will be understood via \eqref{fubini}
or \eqref{coareaformula} depending on the context.
We make the convention that a function $h$ living on $\G_t$  shall be denoted by $h(s)$, even though a global coordinate of $\G_t$ does not exist in general. 
We shall omit $d\mathcal{H}^{N-1}$  in a surface integration like \eqref{gauss} if it is clear from the context.

%%%%%%%%%%%%%%%%%%%%%%%%%%%%%%%%%%%%%%%%%%%%%%%%%%%%%%%%%%%%%%%

\section{Approximate solutions and difference estimate}\label{approaximation}
\indent
In this section
we will build an approximate solution satisfying \eqref{app model}. This is done by gluing      the inner/outer expansions constructed  through the matched asymptotic expansions. We start by  an outline of the construction and leave the details to   Appendix \ref{innerexpan}.

In contrast to the  Cahn-Hilliard case \cite{Alikakos1994}, the outer expansions here are much simpler, and    thus  the boundary layer expansions can be avoided.
Recall  \eqref{def:omegapm} for the definition of  $\Omega_t^{\pm}$, it follows  from exactly the same calculation  as in  \cite{Alikakos1994} that
\begin{align}\phi_a^O(x,t)=\pm
1_{\Omega_t^\pm},\quad \mu_a^O(x,t)=0. \label{out-k}
\end{align}
We  construct the inner expansions in $\Gamma_t(2\delta)$   by
  \begin{align}\label{in-k}
\phi_a^I(x,t)=\sum_{0\leq i\leq k}\varepsilon^i\tilde{\phi}^{(i)}(z,x,t)\Big|_{z=\tfrac{ r(x,t) }{\varepsilon}},\qquad
\mu_a^I(x,t)=\sum_{0\leq i\leq k}\varepsilon^i\tilde{\mu}^{(i)}(z,x,t)\Big|_{z=\tfrac{ r(x,t) }{\varepsilon}},
\end{align}
where $z$ is the fast variable and  $r(x,t)$ is the signed-distance function of $\G_t$. See \eqref{nasty2} for the definition of $\G_t$. While we obtain recursive formulas for $(\widetilde{\phi}^{(i)},\widetilde{\mu}^{(i)})$ in the appendix, the first  few terms are determined in Subsection \ref{subsection matching 01} and \ref{subsection matching 2}:
\begin{equation}\label{phi12}
\widetilde{\phi}^{(0)}(z,x,t)=\theta(z),\quad
\widetilde{\phi}^{(1)}(z,x,t)=0,\quad
\widetilde{\phi}^{(2)}(z,x,t)=D^{(0)}(x,t)\theta'(z)\alpha(z),
\end{equation}
\begin{subequations}\label{mu12}
\begin{align}
\widetilde{\mu}^{(0)}(z,x,t)&=-\Delta d^{(0)}(x,t)\theta'(z),\\
\widetilde{\mu}^{(1)}(z,x,t)&=-\Delta d^{(1)}(x,t)\theta'(z)+D^{(0)}(x,t)z\theta'(z),\\
\widetilde{\mu}^{(2)}(z,x,t)&=\Delta d^{(0)} {D^{(0)}}\theta'(z)\gamma_1(z)
+\nabla d^{(0)}\cdot\nabla {D^{(0)}}\theta'(z)\gamma_2(z)
+ {D^{(1)}}z\theta'\nonumber \\&\qquad
+\mu_2(x,t)\theta'+\chi^{(0)}d^{(0)}\theta'(z)\gamma_3(z),
\end{align}
\end{subequations}
where  $\{d^{(i)}\}_{i=0}^2$ are defined by \eqref{distance law}, \eqref{equation of d1} and  \eqref{equation of dk} (with $k=2$) successively, and
$D^{(0)},D^{(1)}$ are given by
\begin{equation}
D^{(i)}=\sum\limits_{0\leq \ell \leq i}\(\nabla\Delta d^{(\ell)}\cdot\nabla d^{(i-\ell)}+\tfrac 12\Delta d^{(\ell)}\Delta d^{(i-\ell)}\).\label{def:D0D1}
\end{equation}
The functions $\alpha(z),\{\gamma_j(z)\}_{j=1}^3$ and $\mu_2(x,t)$ are given by \eqref{expression of phi2}, \eqref{eq:2.7} and \eqref{expression of mu2-1}  respectively.
The following lemma is a consequence of L'H\^{o}pital's rule.
\begin{lemma}\label{evenodd1}
$\alpha(z)$ is odd and $\gamma_1(z),\gamma_3(z)$  are  even, and they  grow at most quadratically at infinity. Moreover, $\gamma_2(z)=-z^2/2$.
\end{lemma}
Using the cut-off function \eqref{cutoff}, we can glue the inner and outer expansions by
\begin{equation}
  \begin{split}
    \phi_a(x,t)=
 \phi_a^I+
 (1-\zeta(r))(\phi_a^O-\phi_a^I),\quad
 \mu_a(x,t)=
 \mu_a^I+
 (1-\zeta(r))(\mu_a^O-\mu_a^I). \label{definition app}
  \end{split}
\end{equation}  
\begin{proposition}\label{construction of app solution}
The functions $\phi_a$ and $\mu_a$ in \eqref{definition app} are smooth and  fulfill   \eqref{app model} in $\Omega\times (0,T)$. Moreover, they have the following expansions in $\G_t(\delta)$:
\begin{equation}\label{eq:expan phia}
\begin{split}
\phi_a(x,t)&=\theta(\tfrac{r(x,t)}\e)+\varepsilon^2\widetilde{\phi}^{(2)}
\(\tfrac{r(x,t)}\e,x,t\)+O(\varepsilon^3),\\
\mu_a(x,t)&=\sum_{0\leq i\leq 2}\varepsilon^i\widetilde{\mu}^{(i)}
\(\tfrac{r(x,t)}\e,x,t\) +O(\varepsilon^3).
\end{split}
\end{equation}
\end{proposition}
\begin{proof}
It is clear that $\phi_a$ is a smooth function that coincides with $\phi_a^I$ inside $\G_t(\delta/2)$ and with $\phi_a^O$ outside $\G_t(\delta)$.
 Moreover, the difference  $\phi_a-\phi_a^I$ decays at the order of $O(e^{-C/\e})$ when it is restricted in $\Omega\backslash \G_t(\delta/2)$, due to the matching condition \eqref{matching condition}. The same statement is valid for $\mu_a$. So
substituting  \eqref{out-k},  \eqref{in-k}, \eqref{phi12} and \eqref{mu12} into \eqref{definition app} yields \eqref{eq:expan phia}.
To verify (\ref{app model}), we need to employ Proposition \ref{endprop1}. We rewrite \eqref{definition app} by
\begin{equation}\label{nasty6}
  \begin{split}
    \phi_a(x,t)=&
 \hat{\phi}^I_a+
 (1-\zeta(r))(\phi_a^O-\phi_a^I)+\phi_a^I-\hat{\phi}^I_a,\\
 \mu_a(x,t)=&
 \hat{\mu}^I_a+
 (1-\zeta(r))(\mu_a^O-\mu_a^I)+\mu_a^I-\hat{\mu}^I_a,
  \end{split}
\end{equation}
where $ \hat{\phi}^I_a,\hat{\mu}^I_a$ are defined in Proposition \ref{endprop1}.
It follows from \eqref{app distance} that $\G(2\delta)\subset\G^0(3\delta)$ for a  sufficiently large  $k$. So the first terms on the right hand side of \eqref{nasty6} fulfill \eqref{newapprequ-5} and \eqref{newapprequ-6} respectively. For the second parts, one can employ \eqref{matching condition} to obtain the exponential decay in $\e$. Finally, it follows from \eqref{app distance} that
\begin{align*}
\phi_a^I-\hat{\phi}^I_a&=\sum_{0\leq i\leq k}\varepsilon^i \tilde{\phi}^{(i)}(z,x,t)\Big|^{z=
\frac{ r(x,t) }{\varepsilon}}_{z=
\frac{ d^{[k]}(x,t) }{\varepsilon}}
=O(\varepsilon^{k-1}).
\end{align*}
Similar consideration leads to
\begin{align*}
&\partial_t(\phi_a^I-\hat{\phi}^I_a)=O(\varepsilon^{k-2}),~\mu_a^I-\hat{\mu}^I_a=O(\varepsilon^{k-1}),\\
& \Delta(\phi_a^I- \hat{\phi}^I_a)=O(\varepsilon^{k-3}),~\Delta(\mu_a^I-\hat{\mu}^I_a)=O(\varepsilon^{k-3}).
\end{align*}
So  (\ref{app model}) is proved.  
\end{proof}
%$\lambda_1\triangleq \inf_{\|q\|_{L^2(I_\e)}=1}\int_{I_{\ve}}\big((q')^2+f''(\theta)q^2\big)dz$
\begin{lemma}\label{spectral lemma}
Let  $\lambda_1$ be    the principal eigenvalue of the  operator \eqref{1doperator}  with homogenous Neumann boundary condition, and  $\varphi(z)$ be   the corresponding   eigenfunction, normalized    such that $\|\varphi\|_{L^2(I_\e)}=\|\theta'\|_{L^2(I_\e)}$. Denote $I_\e=(-\delta/\e,\delta/\e)$, then   
 \begin{align}
\|\varphi-\theta'\|_{W^{2,\infty}(I_\varepsilon)}=O(e^{-\frac{C}{\varepsilon}}),\label{difference of kernal}
\end{align}
and $\lambda_1=O(e^{-\frac{C}{\ve}})$.
Moreover,   for some $\Lambda_3>0$ there holds 
\begin{equation}\label{eigen:f_2der}
  \int_{I_\e}\(q'^2+f''(\theta)q^2\)\, dz\geq \Lambda_3\int_{I_\e} (q'^2+q^2)dz,\qquad\forall q\perp \varphi~\text{in}~L^2(I_\e).
\end{equation}

\end{lemma}
\begin{rmk}
The above lemma  is mainly due to \cite{chen1994spectrum}.
There the author defines  the $\e$-dependent constant $\beta=\big(\|\theta'\|_{L^2(I_\varepsilon)}\big)^{-1}$ and   $\varphi$ such that $\|\varphi\|_{L^2(I_\e)}=1$ and $\varphi-\beta\theta'=O(e^{-C/\e})$. In this work,  $\beta$ is absorbed  into $\varphi$ for the convenience of Section \ref{section:spectrum}.
\end{rmk}

\begin{lemma}\label{equinorm}
The renormalization constant \eqref{apptension} satisfies
\begin{equation}\label{tension}
  \eta(s)=\int_{\mathbb{R}}(\theta'(z))^2dz+O(\varepsilon).
\end{equation}
Moreover, there exist   $\e_0>0$ and $\Lambda_5>1$ such that for every $t\in[0,T]$,
\begin{equation}
%&\int_{\Omega\backslash\Gamma_t(\delta)}\phi^2 dx=\int_{\Omega\backslash\Gamma_t(\delta)}(\phi^\bot)^2 dx,\label{decompose L2norm 2}\\
\Lambda_5\int_{\Gamma_t(\delta)}\phi^2dx\geq  \int_{\Gamma_t} Z^2(s)+\int_{\Gamma_t(\delta)}\big(\phi^\bot\big)^2dx\geq \frac 1{\Lambda_5}\int_{\Gamma_t(\delta)}\phi^2dx,~\forall \varepsilon\in (0,\e_0).\label{decompose L2norm}
\end{equation}
\end{lemma}
\begin{proof}
The asymptotic formula of $\eta$  follows from a change of variable together with \eqref{difference of kernal},
\begin{equation*}
\begin{split}
\eta(s)
&=\int_{-\frac{\delta}{\varepsilon}}^{\frac{\delta}{\varepsilon}}\big(\varphi(z)- \theta'(z)+ \theta'(z)\big)^2\zeta^2(\varepsilon z)J^{\frac{1}{2}}(\varepsilon z,s)dz
\\&=\int_{-\frac{\delta}{\varepsilon}}^{\frac{\delta}{\varepsilon}}\(\theta'(z)\)^2\zeta^2(\varepsilon z)J^{\frac{1}{2}}(\varepsilon z,s)dz+O(e^{-\frac C\e}).
\end{split}
\end{equation*}
Then expanding $\zeta^2(r)J^{\frac 12}(r,s)$ at $r=0$ and using \eqref{jacobian} lead to \eqref{tension}.   To prove \eqref{decompose L2norm}, it follows from  \eqref{fubini}, \eqref{boundjacobian} and  (\ref{orthogonality})   that
\begin{equation*}
\begin{split}
\int_{\Gamma_t(\delta)}\phi^2(x)dx
&=\int_{\Gamma_t}\int_{-\delta}^{\delta}\phi^2(r,s)J(r,s)dr\nonumber\\&\gtrsim\int_{\Gamma_t}\int_{-\delta}^{\delta}\phi^2(r,s)J^{\frac{1}{2}}(r,s)dr
\\&= \int_{\Gamma_t} Z^2(s)\(\e^{-1}\int_{-\delta}^{\delta}\varphi^2(\tfrac r\e)
\zeta^2(r)J^{\frac{1}{2}}(r,s)dr\)
+\int_{\Gamma_t}\int_{-\delta}^{\delta}\big(\phi^\bot\big)^2J^{\frac{1}{2}}(r,s)dr.
\end{split}
\end{equation*}
This together with  \eqref{apptension},  \eqref{tension} and \eqref{boundjacobian}    implies the first inequality in \eqref{decompose L2norm}. Reversing the above estimates gives the second inequality and the proof is completed.
\end{proof}
Decomposition \eqref{eq:decom1} is not convenient to manipulate as it involves the eigenfunction $\varphi$. Due to \eqref{difference of kernal}, we shall instead  decompose $\phi$ through $\theta'$. To this end, we set
$\rho(r)=\varepsilon^{-\frac{1}{2}}\big(\varphi(\tfrac{r}{\varepsilon})-\theta'(\tfrac{r}{\varepsilon})\big).$
Thanks to (\ref{difference of kernal}) there holds
\begin{equation}
  \rho(r),\rho'(r),\rho''(r)=O(e^{-\frac C\e}).\label{decompose difference}
\end{equation}
We rewrite \eqref{eq:decom1} by
\begin{equation}
  \phi(x)=\underbrace{\rho(r)Z(s)\zeta(r)}_{\triangleq\phi_e(r,s)}
  +\underbrace{\varepsilon^{-\frac{1}{2}} Z(s)\theta'(\tfrac{r}{\varepsilon})\zeta(r)}_{\triangleq\phi^\top(r,s)}+\phi^\bot(x),\label{decompose phi}
\end{equation}
according to which we can expand the terms on the left hand side of \eqref{spectral condition} by   
\begin{equation}\label{split1}
  \begin{split}
    \frac{1}{\e^2}\int_{\Omega}|\mathscr{L}_\e[\phi]|^2dx
=&\frac{1}{\e^2}\int_{\Omega}\(|\mathscr{L}_\e[\phi_e]|^2+|\mathscr{L}_\e[\phi^\top]|^2+|\mathscr{L}_\e[\phi^\bot]|^2\)dx\\
&+\frac{2}{\e^2}\int_{\Omega}\(\mathscr{L}_
\e[\phi_e]\mathscr{L}_\e[\phi^\top+\phi^\bot]+\mathscr{L}_\e[\phi^\top]
\mathscr{L}_\e[\phi^\bot]\)dx,\\
\frac{1}{\varepsilon^3}\int_{\Omega} f'''(\phi_a)\mu_a\phi^2dx
&=\frac{1}{\varepsilon^3}\int_{\Omega} f'''(\phi_a)\mu_a\((\phi_e)^2+(\phi^\bot)^2+(\phi^\top)^2\)dx\\
&+\frac{2}{\varepsilon^3}\int_{\Omega} f'''(\phi_a)\mu_a\(\phi_e(\phi^\top+\phi^\bot)+\phi^\bot\phi^\top \) dx.
  \end{split}
\end{equation}
Using (\ref{decompose difference}), the terms in the above formulas that include $\phi_e$ can be treated by
\begin{equation}\label{split2}
  \begin{split}
    &\frac{1}{\varepsilon^2}\int_{\Omega}|\mathscr{L}_\e[\phi_e]|^2dx
+\frac{2}{\varepsilon^2}\int_{\Omega}\mathscr{L}_\e[\phi^\bot+\phi^\top]\mathscr{L}_\e[\phi_e]dx
\\
&+\frac{1}{\varepsilon^3}\int_{\Omega} f'''(\phi_a)\mu_a\big(\phi_e\big)^2dx+\frac{2}{\varepsilon^3}\int_{\Omega} f'''(\phi_a)\mu_a\phi_e(\phi^\top+\phi^\bot) dx\\&\geq O(e^{-\frac{C}{\varepsilon}})\sup_{|r|\leq \delta}\int_{\Gamma_t} |\Delta_{\G}Z(s)|^2+O(e^{-\frac{C}{\varepsilon}})\int_{\Gamma_t}Z^2(s)
+O(e^{-\frac{C}{\varepsilon}})\int_{\Omega}
|\mathscr{L}_\e[\phi^\bot]|^2dx.
  \end{split}
\end{equation}
Here    $\sup_{|r|\leq \delta}$ appears because the coefficients of $\Delta_\G$ depend on $r$ smoothly, according to \eqref{eq:laplace1}. Note that  $O(e^{-\frac{C}{\varepsilon}})$ does not have a fixed sign.
Combining \eqref{split2} with \eqref{split1} yields
\begin{align}
&\frac{1}{\varepsilon^2}\int_{\Omega}|\mathscr{L}_\e[\phi]|^2dx+\frac{1}{\varepsilon^3}\int_{\Omega} f'''(\phi_a)\mu_a\phi^2dx\nonumber\\&\geq O(e^{-\frac{C}{\varepsilon}})\sup_{|r|\leq \delta}\int_{\Gamma_t} \big(\Delta_{\G}Z(s)\big)^2+O(e^{-\frac{C}{\varepsilon}})\int_{\Gamma_t}Z^2(s)
+I_1+I_2+I_3,\label{decompose energy}
\end{align}
where $I_1,I_2,I_3$ are defined by
\begin{subequations}\label{nasty}
  \begin{align}
 I_1&=\frac{1}{\varepsilon^2}
 \int_{\Omega}|\mathscr{L}_\e[\phi^\top]|^2dx+\frac{1}{\varepsilon^3}\int_{\Omega} f'''(\phi_a)\mu_a\big(\phi^\top\big)^2dx,\label{decompose energy-1}\\
 I_2&=\frac{2}{\varepsilon^2}\int_{\Omega} \mathscr{L}_\e[\phi^\top]\mathscr{L}_\e[\phi^\bot] dx+\frac{2}{\varepsilon^3}\int_{\Omega} f'''(\phi_a)\mu_a\phi^\top\phi^\bot dx,\label{decompose energy-2}\\
  I_3&=\frac{1+O(e^{-\frac{C}{\varepsilon}})}{\varepsilon^2}\int_{\Omega}|\mathscr{L}_\e[\phi^\bot]|^2dx+\frac{1}{\varepsilon^3}\int_{\Omega} f'''(\phi_a)\mu_a\big(\phi^\bot\big)^2dx\triangleq I_{31}+I_{32}.\label{decompose energy-3}
\end{align}
\end{subequations}
Theorem \ref{spectral theorem} will be a consequence of the estimates of \eqref{decompose energy-1}-\eqref{decompose energy-3}, which will be done in the following three sections.

%which is equivalent to
%\begin{align}
%&\frac{1}{2}\frac{d}{dt}\int_\Omega\phi^2dx+\int_\Omega\bigg(\big(\Delta \phi\big)^2- 2\varepsilon^{-2}f''(\phi_a)\phi\Delta \phi\nonumber\\&\qquad\qquad\qquad+\big(\varepsilon^{-3}f'''(\phi_a)\mu_a+\varepsilon^{-4}(f''(\phi_a))^2\big)\phi^2\bigg)dx
%\nonumber\\&=\int_\Omega\mathcal{H}\phi dx.\label{error energy-2}
%\end{align}

%
%we only need to consider
%\begin{align}
%\frac{1}{\varepsilon^2}\int_{\Gamma^\varepsilon(\delta)}\omega^2dx+\frac{1}{\varepsilon^3}\int_{\Gamma^\varepsilon(\delta)} f'''(\phi_a)\mu_a\phi^2dx.\label{singular term}
%%\\&=\int_{\Gamma^\varepsilon(\delta)}\bigg(\big(\Delta \phi\big)^2- 2\varepsilon^{-2}f''(\phi_a)\phi\Delta \phi+\big(\varepsilon^{-3}f'''(\phi_a)\mu_a+\varepsilon^{-4}(f''(\phi_a))^2\big)\phi^2\bigg)dx.
%%\\&=\frac{1}{\varepsilon^2}\int_{\Gamma^\varepsilon}\int_{-\delta}^{\delta}\omega^2(r,s,t)J(r,s)drds+\frac{1}{\varepsilon^3}\int_{\Gamma^\varepsilon}\int_{-\delta}^{\delta} \big(f'''(\phi_a)\mu_a\phi^2\big)(r,s,t)J(r,s)drds,
%\end{align}
%here we have used the coordinates transformation from $(x,t)$ to $(r,s,t)$ with $r=d_{\varepsilon}$ and $s=(s_1,s_2)\in\Gamma^\varepsilon$.
%%%%%%%%%%%%%%%%%%%%%%%%%%%%%%%%%%%%%%%%%%%%%%%%%%%%%%%%%%%%%%%%
%%%%%%%%%%%%%%%%%%%%%%%%%%%%%%%%%%%%%%%%%%%%%%%%%%%%%%%%%%%%%%%%
%%%%%%%%%%%%%%%%%%%%%%%%%%%%%%%%%%%%%%%%%%%%%%%%%%%%%%%%%%%%%%%%
%%%%%%%%%%%%%%%%%%%%%%%%%%%%%%%%%%%%%%%%%%%%%%%%%%%%%%%%%%%%%%%%
\section{Spectrum Condition: Estimates of kernel terms.}\label{section:spectrum}

%Define the following linear operator
%\begin{align*}
%\mathcal{L}\phi=\Delta^2 \phi- 2\varepsilon^{-2}f''(\phi_a)\Delta \phi+\bigg(\varepsilon^{-3}f'''(\phi_a)\mu_a+\varepsilon^{-4}(f''(\phi_a))^2\bigg)\phi.
%\end{align*}

  The main result of this section is  stated below concerning \eqref{decompose energy-1}. The proof is given in the end of this section after establishing a few technical lemmas.
 \begin{proposition}\label{propofI1}
 There exists $\Lambda_4>0$  and  generic constant $C>0,\e_0>0$ such that 
 \begin{equation}\label{I1:lower bound}
I_1 \geq \Lambda_4\|Z\|^2_{H^2(\Gamma_t)} -C\|Z\|^2_{L^2(\Gamma_t)},\qquad\forall \e\in (0,\e_0).
\end{equation}
 \end{proposition}
We shall decompose $I_1$ into several parts, and treat them in a series   of lemmas.   To start with, we need the decomposition of Euclidean Laplace operator. Recall \eqref{eq:laplace1} and \eqref{decompose phi}, we can expand $\Delta\phi^\top$ by
\begin{equation}\label{eq:laplace2}
\begin{split}
\Delta\phi^\top&=\p_r^2\phi^\top+\p_r (\ln \sqrt{g})\partial_{r}\phi^\top+\Delta_{\Gamma}\phi^\top
\\&=\varepsilon^{-\frac{5}{2}} Z(s)\theta'''(z)\zeta(r)+\varepsilon^{-\frac{3}{2}}\p_r (\ln \sqrt{g}) Z(s)\theta''(z)\zeta(r)
\\&\quad+\varepsilon^{-\frac{1}{2}}\Delta_{\Gamma}Z(s) \theta'(z)\zeta(r)+Z(s)A(r,s),~\text{with}~z=r(x,t)/\e
\end{split}
\end{equation}
and   $\Delta_{\Gamma}$ is defined by  \eqref{eq:laplace1}. Moreover, according to \eqref{cutoff},
\begin{align}
A(r,s)=2\varepsilon^{-\frac{3}{2}} \theta''(\tfrac r\e)\zeta'(r)+\varepsilon^{-\frac{1}{2}} \theta'(\tfrac r\e)\zeta''(r)+\varepsilon^{-\frac{1}{2}} \p_r (\ln \sqrt{g})\theta'(\tfrac r\e)\zeta'(r)= O(e^{-\frac{C}{\varepsilon}}).\label{zeta term}
\end{align}

\begin{lemma}
 The following expansions hold in $\G_t(\delta)$ with $z=r(x,t)/\e$:
  \begin{subequations}
    \begin{align}
\label{good  term}
\theta'''(z)-&f''(\phi_a)\theta'(z)=-\varepsilon^2f'''(\theta)\widetilde{\phi}^{(2)}\theta'(z)+
\theta'(z)O(\varepsilon^3),\\
\label{expansion2}
  f'''(\phi_a)\mu_a&= -6h(s)
    \theta\theta'+3\varepsilon h^2(s) z\theta\theta'+6\varepsilon^2\big(\theta\widetilde{\mu}^{(2)}+
\widetilde{\phi}^{(2)}\widetilde{\mu}^{(0)}\big)\nonumber\\&-6\varepsilon^2\theta\theta'\( a(s)z^2-\Delta d^{(2)}
- c(s) z^2+  zD^{(1)}- h(s)b(s) z^2\)
+O(\varepsilon^3).
    \end{align}
  \end{subequations}
\end{lemma}
\begin{proof}
Recall that $f(u)=\tfrac{1}{4}(u^2-1)^2$ and that  $\theta''=f'(\theta)$. Then \eqref{good  term} follows from  \eqref{eq:expan phia} and the Taylor expansion.
To prove \eqref{expansion2} we need
$
f'''(\phi_a)=6\theta+6\varepsilon^2\widetilde{\phi}^{(2)}+
O(\varepsilon^3),$
following from  \eqref{eq:expan phia} and $f'''(u)=6u$.
Then it follows from \eqref{eq:expan phia},  \eqref{mu12} and \eqref{def:D0D1} that
\begin{align*}
f'''(\phi_a)\mu_a&=6\theta\widetilde{\mu}^{(0)}+
6\varepsilon\theta\widetilde{\mu}^{(1)}
+6\varepsilon^2\(\theta\widetilde{\mu}^{(2)}+
\widetilde{\phi}^{(2)}\widetilde{\mu}^{(0)}\)+O(\varepsilon^3)
\nonumber\\&=
-6\theta\theta'\big(\Delta d^{(0)}+\varepsilon\Delta d^{(1)}\big)+6\varepsilon z\theta\theta'D^{(0)}+6\varepsilon^2\big
(\theta\widetilde{\mu}^{(2)}+\widetilde{\phi}^{(2)}\widetilde{\mu}^{(0)}
\big)+O(\varepsilon^3)\\
&= -6\theta\theta'\Delta r+6\varepsilon^2\theta\theta'\Delta d^{(2)}+6\varepsilon z\theta\theta'\nabla\Delta r \cdot\nabla r+3\varepsilon z\theta\theta'(\Delta r)^2
\nonumber\\&\quad-6\varepsilon^2 z\theta\theta'D^{(1)}+
6\varepsilon^2\big(\theta\widetilde{\mu}^{(2)}+
\widetilde{\phi}^{(2)}\widetilde{\mu}^{(0)}\big)
+ O(\e^3).
\end{align*}
Note that in the last step, we use \eqref{app distance} to express the leading order terms  in term of $r$. Then using \eqref{eq:2.1}, we can expand various terms about $r$ and then replace them by $z=r/\e$. These   lead to the following formula which implies \eqref{expansion2}
\begin{equation*}
\begin{split}
  f'''(\phi_a)\mu_a&= -6h(s)
    \theta\theta'+3\varepsilon h^2(s) z\theta\theta'-6\varepsilon^2 a(s)z^2\theta\theta'+6\varepsilon^2\theta\theta'\Delta d^{(2)}
+6\varepsilon^2 c(s) z^2\theta\theta'
\\&\qquad-6\varepsilon^2 z\theta\theta'D^{(1)}
+6\varepsilon^2 h(s)b(s) z^2\theta\theta'
+6\varepsilon^2\big(\theta\widetilde{\mu}^{(2)}+
\widetilde{\phi}^{(2)}\widetilde{\mu}^{(0)}\big)
+O(\varepsilon^3).
  \end{split}
  \end{equation*}
\end{proof}

To treat the integrals in \eqref{decompose energy-1}, we need the following  three lemmas.  \begin{lemma}\label{reduce2}
Assume $\omega\in C^1(\overline{\Gamma_t(\delta)})$ and  $\tilde{\theta}$ satisfies
 \begin{equation}\label{optimalfileclass}
 \tilde{\theta}(z)\in L^1(\mathbb{R}),~ \tilde{\theta}(-z)=-\tilde{\theta}(z),~\text{and}~\lim_{z\to\pm \infty}\tilde{\theta}(z)=0~\text{exponentially}.
 \end{equation}
Then   there exists an $\e$-independent constant $C$ such   that
  \begin{equation*}
    \left|\int_{\Gamma_t} \(\int_{-\delta}^\delta u(s)\omega(r,s) \tilde{\theta}(\tfrac r\e) dr\)\right|\leq C \e^2 \int_{\Gamma_t} |u(s)|,\quad \forall u\in L^1(\G_t).
  \end{equation*}

\end{lemma}
\begin{proof}
Expanding  $\omega(r,s)$ in terms of $r$ and using the oddness of $\tilde{\theta}$,
  \begin{equation*}
    \begin{split}
      \int_{\Gamma_t} \int_{-\delta}^\delta u(s)\omega(r,s) \tilde{\theta}(\tfrac r\e) dr
    =& \int_{\Gamma_t} \int_{-\delta}^\delta u(s)\(\omega(0,s)+r\p_r \omega(\xi(r),s)\) \tilde{\theta}(\tfrac r\e) dr\\
    =&\e\int_{\Gamma_t} \int_{-\delta}^\delta u(s)\p_r \omega(\xi(r),s)  \tilde{\theta}(\tfrac r\e)\tfrac r\e dr.
    \end{split}
  \end{equation*}
  This leads to the desired result after a change of variable $r=\e z$.
\end{proof}
The next lemma handles integrals involving $\Delta_{\G}$:
\begin{lemma}\label{roperatorlemma}
Assume $\tilde{\theta}(z)$ satisfies \eqref{optimalfileclass} and $\omega\in C^1(\overline{\Gamma_t(\delta)})$. Then for any $ u\in L^2(\Gamma_t)$ and $v\in H^2(\Gamma_t)$,
  \begin{equation*}
  \begin{split}
  &\left|\int_{\Gamma_t} \int_{-\delta}^\delta  u(s)\Delta_{\G}v(s) \tilde{\theta}(\tfrac r\e) \omega(r,s) dr\right|\\
  \leq  &\nu\e^2\int_{\Gamma_t} \(\big|\Delta_{\G}\big|_{r=0}v(s)\big|^2+|v(s)|^2\)+ \tfrac{C \e^2}{\nu}  \int_{\Gamma_t}    u^2(s) , \quad\forall \nu>\e^{\frac 1{100}},
  \end{split}
\end{equation*}
where the constant $C$ depends on $\omega, \tilde{\theta}$ but not on $\varepsilon,\nu, u$ or $v$.
\end{lemma}
\begin{proof}
Recall \eqref{roperator},  we have the  expansion
$\Delta_{\G}= \Delta_{\G}\big|_{r=0}+r   \RR_{\G}$.
So
  \begin{equation*}
  \begin{split}
  &\int_{\Gamma_t} \int_{-\delta}^\delta  u(s)\Delta_{\G}v(s) \tilde{\theta}(\tfrac r\e) \omega(r,s) dr\\
  =&\int_{\Gamma_t} \int_{-\delta}^\delta  u(s)\big(\Delta_{\G}\big|_{r=0}v(s)\big) \tilde{\theta}(\tfrac r\e) \omega(r,s) dr+\e\int_{\Gamma_t} \int_{-\delta}^\delta  u(s) \RR_{\G} v(s)\(\tfrac r\e\tilde{\theta}(\tfrac r\e)\) \omega(r,s) dr.
  \end{split}
\end{equation*}

The first integral can be handled by Lemma \ref{reduce2}: for any $\nu>\e^{\frac 1{100}}$,
\begin{align*}
 \left| \int_{\Gamma_t} \int_{-\delta}^\delta  u(s)\big(\Delta_{\G}\big|_{r=0}v(s)\big) \tilde{\theta}(\tfrac r\e) \omega(r,s) dr\right|&\leq   \frac{\nu \e^2}2 \int_{\Gamma_t} \big(\Delta_{\G}\big|_{r=0}v\big)^2 + C\e^2 \nu^{-1}\int_{\Gamma_t} u^2 .
\end{align*}

For  the second one, we use \eqref{roperatorest} and  the Cauchy-Schwarz inequality to  yield
   \begin{equation*}
  \begin{split}
  \left|\int_{\Gamma_t} \int_{-\delta}^\delta  u(s) \RR_{\G} v(s)\(\tfrac r\e\tilde{\theta}(\tfrac r\e)\) \omega(r,s) dr\right|
  \leq \frac{\nu \e}2 \int_{\Gamma_t} \(\big(\Delta_{\G}\big|_{r=0}v\big)^2+|v|^2\) + \e C\nu^{-1}\int_{\Gamma_t} u^2,
  \end{split}
\end{equation*}
where $C$ depends on $\omega,\tilde{\theta}$.
The above three estimates  together imply  the desired result.
\end{proof}

\begin{lemma}
Assume  $\tilde{\theta}(z)$ satisfies \eqref{optimalfileclass}.  Then there exist $\Lambda_1,\Lambda_2>0$  such that, for every $\e<\e_0$ and every $v\in H^2(\Gamma_t)$, the following two inequalities hold
  \begin{subequations}
    \begin{align}
      \label{laplace2}
 &\Lambda_1 \e \int_{\Gamma_t} \big(\Delta_{\G}\big|_{r=0}v(s)\big)^2 \leq     \int_{\Gamma_t}\int_{-\delta}^{\delta} |\Delta_{\G}v(s)|^2
      \tilde{\theta}^2(\tfrac r\e)J(r,s)dr+\e\int_{\Gamma_t} v^2(s),\\
      &\int_{\Gamma_t}\int_{-\delta}^{\delta} |\Delta_{\G}v(s)|^2
      \tilde{\theta}^2(\tfrac r\e)J(r,s)dr\leq  \Lambda_2 \e\int_{\Gamma_t} \(v^2(s)+ \big(\Delta_{\G}\big|_{r=0}v\big)^2\).\label{laplace3}
    \end{align}
  \end{subequations}
\end{lemma}
\begin{proof}
We only prove the first inequality. The other one  follows in exactly the same manner.
  In view of \eqref{jacobian}, we have $J(r,s)\geq 1/2$.  Since $\tilde{\theta}^2(z)$ decays exponentially, so does $z\tilde{\theta}^2(z)$. As a result, we employ \eqref{laplace1} and obtain  for some $\Lambda_1>0$ that
  \begin{equation*}
 \begin{split}
 &\e\int_{\Gamma_t} v^2(s)+ \int_{\Gamma_t}\int_{-\delta}^{\delta} |\Delta_{\G}v(s)|^2
      \tilde{\theta}^2(\tfrac r\e)J(r,s)dr\\
      \geq   &\e\int_{\Gamma_t} v^2(s)+\tfrac 1\Lambda\int_{\Gamma_t}\inf_{|r|\leq \delta} |\Delta_{\G}v(s)|^2
     \(\int_{-\delta}^{\delta} \tilde{\theta}^2(\tfrac r\e)dr\)\geq \e\Lambda_1 \int_{\Gamma_t} \big(\Delta_{\G}\big|_{r=0}v\big)^2.
 \end{split}
  \end{equation*}
\end{proof}
We are in a position to estimate \eqref{decompose energy-1}.  For simplicity, we  shall suppress  the arguments of a function if it is clear from the context, and omit $d\mathcal{H}^{N-1}$ in a surface  integral   $\int_{\G_t}$ (recall \eqref{fubini}).
 Meanwhile, we shall often employ the  following  inequality without mentioning  \begin{equation}\label{cauchy-schwarz}
 ab\leq \nu a^2+\ 4b^2/\nu,~\text{with the convention that}~\nu\geq \e^{\frac 1{100}}.
 \end{equation}
So the choice of  $\nu$ will not significantly affect the order of $\e$.
Using \eqref{omegadef} and    \eqref{eq:laplace2} yields
 \begin{equation}\label{expansion phi top}
\begin{split}
\mathscr{L}_\e[\phi^\top]=&\varepsilon^{-\frac{3}{2}} Z(s)\big(\theta'''(z)-f''(\phi_a)\theta'(z)\big)\zeta(r)
+\varepsilon^{-\frac{1}{2}}\p_r (\ln \sqrt{g}) Z(s)\theta''(z)\zeta(r)\\
&+\varepsilon^{\frac{1}{2}}\Delta_{\Gamma}Z(s) \theta'(z)\zeta(r)+\varepsilon Z(s)A(r,s).
\end{split}
 \end{equation}
Substituting the above formula into \eqref{decompose energy-1} leads to the following expansion:
 \begin{align}
I_1&=\varepsilon^{-2}\int_{\Gamma_t}\int_{-\delta}^{\delta}\bigg(\varepsilon^{-\frac{3}{2}} Z(s)\big(\theta'''(z)-f''(\phi_a)\theta'(z)\big)\zeta(r)
+\varepsilon^{-\frac{1}{2}}\p_r (\ln \sqrt{g}) Z(s)\theta''(z)\zeta(r)\nonumber\\&\qquad\qquad\qquad+\varepsilon^{\frac{1}{2}}\Delta_{\Gamma}Z(s) \theta'(z)\zeta(r)+\varepsilon Z(s)A(r,s)\bigg)^2J(r,s)dr
\nonumber\\&\quad+\varepsilon^{-4}\int_{\Gamma_t}\int_{-\delta}^{\delta}f'''(\phi_a)\mu_aZ^2(s)\big(\theta'(z)\big)^2\zeta^2(r)J(r,s)dr\triangleq I_{11}+I_{12}+I_{13}+I_{14},\label{I1}
\end{align}
where we define
\begin{align*}
I_{11}&=\varepsilon^{-3}\int_{\Gamma_t}
\int_{-\delta}^{\delta}Z^2(s)\Big(\varepsilon^{-1} \big(\theta'''(z)-f''(\phi_a)\theta'(z)\big)+ \p_r (\ln \sqrt{g}) \theta''(z)\Big)^2\zeta^2(r)J(r,s)dr,
\\ I_{12}&=2\varepsilon^{-3}\int_{\Gamma_t}\int_{-\delta}^{\delta}Z(s)\Delta_{\Gamma}Z(s)\theta'(z)\big(\theta'''(z)-f''(\phi_a)\theta'(z)\big)\zeta^2(r)J(r,s)dr
\nonumber\\&\quad+2\varepsilon^{-2} \int_{\Gamma_t}\int_{-\delta}^{\delta}Z(s)\Delta_{\Gamma}Z(s)\theta''(z)\theta'(z) \zeta^2(r)\p_r (\ln \sqrt{g})J(r,s) dr
\nonumber\\&\quad+\varepsilon^{-1}\int_{\Gamma_t}\int_{-\delta}^{\delta} \big(\Delta_{\Gamma}Z(s)\big)^2\big(\theta'(z)\big)^2\zeta^2(r)J(r,s)dr,
\\ I_{13}&=\varepsilon^{-4}\int_{\Gamma_t}\int_{-\delta}^{\delta}f'''(\phi_a)\mu_aZ^2(s)\big(\theta'(z)\big)^2\zeta^2(r)J(r,s)dr,
\\ I_{14}&=2\varepsilon^{-1}\int_{\Gamma_t}\int_{-\delta}^{\delta}\bigg(\varepsilon^{-\frac{3}{2}} Z(s)\big(\theta'''(z)-f''(\phi_a)\theta'(z)\big)\zeta(r)
+\varepsilon^{-\frac{1}{2}} Z(s)\theta''(z)\zeta(r)\p_r (\ln \sqrt{g})\nonumber\\&\qquad+\varepsilon^{\frac{1}{2}}\Delta_{\Gamma}Z(s) \theta'(z)\zeta(r)\bigg)Z(s)A(r,s)J(r,s)dr
+\int_{\Gamma_t}\int_{-\delta}^{\delta}Z^2(s)A^2(r,s)J(r,s)dr.
\end{align*}

\begin{lemma}
There exist $\e_0>0$ and $C>0$ such that   
\begin{equation}
  \begin{split}
I_{11}\geq\varepsilon^{-3}\int_{\Gamma_t}
\int_{-\delta}^{\delta}Z^2(s)h^2(s)\big(\theta''\big)^2\zeta^2(r)dr-C\int_{\Gamma_t} Z^2(s),\qquad\forall \e\in (0,\e_0).\label{I11}
  \end{split}
\end{equation}
\end{lemma}
\begin{proof}
It follows from  \eqref{phi12} and \eqref{jacobian} that $\widetilde{\phi}^{(2)}, h(s),b(s)$ are uniformly bounded functions.
  Using \eqref{good  term}, \eqref{eq:2.1} and \eqref{phi12}, we can   treat $I_{11}$ by
\begin{align*}
I_{11}&=\frac 1{\varepsilon^3}\int_{\Gamma_t}\int_{-\delta}^{\delta}Z^2(s)\(-\varepsilon f'''(\theta)\widetilde{\phi}^{(2)}\theta'+\theta'O(\varepsilon^2)
+\theta''\big(h(s)+\varepsilon zb(s)+z^2 O(\varepsilon^2)\big)\)^2\zeta^2(r)J dr
\nonumber\\&=\frac 1{\varepsilon^3}\int_{\Gamma_t}\int_{-\delta}^{\delta}Z^2(s)
\bigg(\big(\theta''\big)^2h^2(s)-2\varepsilon f'''(\theta)\theta'\theta''(z)\widetilde{\phi}^{(2)}h(s)
+2\varepsilon z(\theta'')^2h(s)b(s)\nonumber\\& \qquad\qquad+
\underbrace{\theta'O(\varepsilon^2)+\theta''(z)O(\varepsilon^2)
+z\theta'O(\varepsilon^2)+
z^2\big(\theta''\big)^2O(\varepsilon^2)}_{\triangleq \mathcal{J}}+O(\e^3)\bigg)\zeta^2(r)J dr.
\nonumber
\end{align*}
 Note that  the terms corresponding to $\mathcal{J}$ decay in $z$ exponentially to $0$. So we can  gain a factor  $\varepsilon$  after   a change of variable $r\to \e z$ and yield 
\begin{equation*}
\begin{split}
  I_{11}&\geq\frac 1{\varepsilon^3}\int_{\Gamma_t}\int_{-\delta}^{\delta}Z^2 ( \theta'')^2h^2\zeta^2(r)J dr\\
  &-\frac 2{\e^2}\int_{\G_t}\int_{-\delta}^\delta Z^2
\(   f'''(\theta)\theta'\theta''\widetilde{\phi}^{(2)}h
- z(\theta'')^2hb\)\zeta^2(r)J dr-C\int_{\Gamma_t} Z^2.
\end{split}
\end{equation*}
It remains to treat the second term on the right hand side. Recall
\eqref{phi12} (or \eqref{expression of phi2}) that $\widetilde{\phi}^{(2)}(z,x,t)=D^{(0)}(x,t)\theta'(z)\alpha(z)$ for some bounded odd function $\alpha(z)$. So applying Lemma \ref{reduce2} with $\tilde{\theta}=f'''(\theta)\theta'^2\theta''\alpha$, $\omega(r,s)=J(r,s)D^{(0)}h(s)\zeta^2(r)$ and $u(s)=Z^2(s)$ gives the estimate for   its first part, and similar argument gives the estimate of the second part. Altogether, we have
\begin{equation*}
  -\frac 2{\e^2}\int_{\G_t}\int_{-\delta}^\delta Z^2
\( f'''(\theta)\theta'\theta''\widetilde{\phi}^{(2)}h
- z(\theta'')^2hb\)\zeta^2(r)J dr\geq -C\int_{\Gamma_t} Z^2.
\end{equation*}
This together with \eqref{jacobian} leads to
\begin{equation*}
I_{11}\geq\varepsilon^{-3}\int_{\Gamma_t}
\int_{-\delta}^{\delta}Z^2(s)h^2(s)\big(\theta''\big)^2\zeta^2(r)\(1+h(s)r+e(s)r^2+O(r^3)\)dr-C\int_{\Gamma_t} Z^2(s).
\end{equation*}
Finally, as $\big(\theta''\big)^2\zeta^2(r) r$ is odd in $r$,    the integral including $h(s)$  will vanish. So we arrive at the desired result by a change of variable.
\end{proof}
\begin{lemma}
There exist $\Lambda_1>0$, $\e_0>0$ and $C>0$ such that \begin{align}
I_{12}
&\geq   \frac{\Lambda_1}2 \int_{\Gamma_t}  \big(\Delta_{\Gamma}\big|_{r=0}Z(s)\big)^2-C\int_{\Gamma_t} Z^2(s),\qquad \forall \e\in (0,\e_0).
\label{I12}
\end{align}
\end{lemma}
\begin{proof}
We write $I_{12}=I_{121}+I_{122}+I_{123}$ with obvious definitions.
The first part  can be treated by \eqref{good  term}, a change of variable $r=\e z$ and the Cauchy-Schwarz inequality:
\begin{align*}
I_{121}=&
2\varepsilon^{-3}\int_{\Gamma_t}\int_{-\delta}^{\delta}Z(s)
\Delta_{\Gamma}Z(s)\theta'\(-\varepsilon^2f'''(\theta)
\widetilde{\phi}^{(2)}\theta'+\theta'(z)O(\varepsilon^3)\)
\zeta^2(r)J(r,s)dr\\
\geq &-C\int_{\Gamma_t} Z^2(s)-\nu\sup_{|r|<\delta}\int_{\Gamma_t} \big(\Delta_{\Gamma} Z(s)\big)^2.
\end{align*}
Retain that the coefficients of $\Delta_\G$ depends on $r$ smoothly,
 $I_{122}$ can be treated by Lemma \ref{roperatorlemma} with $\tilde{\theta}=\theta''\theta'$ and $\omega=\zeta^2 \p_r(\ln \sqrt{g})J(r,s)$: for any  $\nu>\e^{\frac 1{100}}$,
\begin{equation*}
  I_{122}\geq -\frac C\nu\int_{\Gamma_t} Z^2(s)-\nu\int_{\Gamma_t} \big(\Delta_{\Gamma}\big|_{r=0}Z(s)\big)^2.
\end{equation*}
 $I_{123}$ is nonnegative and can be treated by \eqref{laplace2}
 \begin{equation*}
   I_{123}\geq \Lambda_1 \int_{\Gamma_t}  \big(\Delta_{\Gamma}\big|_{r=0}Z(s)\big)^2-C\int_{\Gamma_t} Z^2(s).
 \end{equation*}
 By choosing a sufficiently small $\nu$,
the above three inequalities and  \eqref{laplace1} imply \eqref{I12}.
\end{proof}

\begin{lemma}
There exist $\e_0>0$ and $C>0$ such that    
\begin{align}
I_{13}\geq-3\varepsilon^{-3}\int_{\Gamma_t}\int_{-\delta}^{\delta}Z^2(s)h^2(s)z\theta\theta'^3\zeta^2(r)dr-C\int_{\Gamma_t} Z^2(s),\qquad \forall \e\in (0,\e_0).\label{I4:lower bound}
\end{align}
\end{lemma}
\begin{proof}
To treat $I_{13}$, we employ \eqref{expansion2}.
Note that  $\theta$ is odd and $\zeta,\theta'$ are even. So in the expression for $I_{13}$, the terms about $a(s),b(s),c(s)$ can be treated by Lemma \ref{reduce2}. As a result
\begin{equation}\label{I4}
  \begin{split}
&I_{13} \geq\underbrace{\varepsilon^{-4}\int_{\Gamma_t}\int_{-\delta}^{\delta}\(-6h(s)
\theta\theta'+3\varepsilon h^2(s) z\theta\theta'\)Z^2(s) \theta'^2\zeta^2(r)Jdr}_{\triangleq I_{131}}-C\int_{\Gamma_t} Z^2(s)
\\&\quad+\underbrace{6\varepsilon^{-2}\int_{\Gamma_t}\int_{-\delta}^{\delta}
\(\theta\theta'\big(\Delta d^{(2)}-zD^{(1)}\big)
+\big(\theta\widetilde{\mu}^{(2)}+
\widetilde{\phi}^{(2)}\widetilde{\mu}^{(0)}\big)\)Z^2 \theta'^2\zeta^2(r)Jdr.}_{\triangleq I_{132}}
  \end{split}
\end{equation}

Employing  the expansion of $J(r,s)$ in \eqref{jacobian}, we can write $I_{131}$ by
\begin{equation*}
  \begin{split}
    I_{131}&=\varepsilon^{-4}\int_{\Gamma_t}\int_{-\delta}^{\delta}\(-6h(s)
\theta\theta'+3\varepsilon h^2(s) z\theta\theta'\)Z^2(s)\theta'^2
\\&\qquad\qquad\qquad\cdot\(1+\varepsilon zh(s)+\varepsilon^2z^2e(s)+z^3O(\varepsilon^3)\)\zeta^2(r)dr
\\&=-6\varepsilon^{-4}\int_{\Gamma_t}
\int_{-\delta}^{\delta}Z^2(s)h(s)\theta\theta'^3\zeta^2(r)dr
-3\varepsilon^{-3}\int_{\Gamma_t}\int_{-\delta}^{\delta}Z^2(s)h^2(s)z\theta\theta'^3\zeta^2(r) dr
\\&\quad+3\varepsilon^{-2}\int_{\Gamma_t}\int_{-\delta}^{\delta}Z^2(s)h(s)\big(h^2(s)-2e(s)\big)z^2\theta\theta'^3\zeta^2(r) dr
\\&\quad+{\varepsilon^{-1}\int_{\Gamma_t}\int_{-\delta}^{\delta}Z^2(s)\tilde{\theta}(\tfrac r\e)O(1)\zeta^2(r) dr},
  \end{split}
\end{equation*}
where $\tilde{\theta}(z)\in L^1(\mathbb{R})$ and decays exponentially to $0$ at $\pm\infty$.
Since the first and the third term above vanish, we obtain 
\begin{equation}
  \begin{split}
 I_{131}&\geq-3\varepsilon^{-3}\int_{\Gamma_t}\int_{-\delta}^{\delta}Z^2(s)h^2(s)z\theta\theta'^3\zeta^2(r)dr-C\int_{\Gamma_t} Z^2(s).\label{I41}
  \end{split}
\end{equation}

We now turn to $I_{132}$ in \eqref{I4}. Note that
 the analytic expression for $\widetilde{\mu}^{(2)},\widetilde{\phi}^{(2)},\widetilde{\mu}^{(0)}$ are given in \eqref{phi12} and \eqref{mu12} respectively, and  $\{d^{(i)}\}_{i=0}^2$ are smooth functions in $\G(2\delta)$ that are defined by \eqref{distance law}, \eqref{equation of d1} and  \eqref{equation of dk} successively.  Combining Lemma \ref{reduce2} with  the fact that $\widetilde{\phi}^{(2)}\widetilde{\mu}^{(0)}$ is odd in $z$, we infer that
\begin{align*}
&6\varepsilon^{-2}\int_{\Gamma_t}\int_{-\delta}^{\delta}\widetilde{\phi}
^{(2)}\widetilde{\mu}^{(0)}Z^2(s)\theta'^2
\zeta^2(r)Jdr
 \geq-C\int_{\Gamma_t} Z^2(s).
\end{align*}
The terms about $\Delta d^{(2)}$ can be treated similarly, so we employ the last formula in \eqref{mu12}
\begin{align*}
I_{132}&\geq6\varepsilon^{-2}\int_{\Gamma_t}\int_{-\delta}^{\delta}
\(-\theta\theta'zD^{(1)}
+\theta\widetilde{\mu}^{(2)}\) Z^2(s)\theta'^2\zeta^2(r)Jdr-C\int_{\Gamma_t} Z^2(s)
\nonumber\\&=6\varepsilon^{-2}\int_{\Gamma_t}\int_{-\delta}^{\delta}Z^2(s)\Delta d^{(0)}D^{(0)}\gamma_1(z)\theta\theta'^3\zeta^2(r)Jdr
\nonumber\\&\quad+6\varepsilon^{-2}\int_{\Gamma_t}\int_{-\delta}^{\delta}Z^2(s)\nabla d^{(0)}\cdot\nabla D^{(0)}\gamma_2(z)\theta\theta'^3\zeta^2(r)Jdr
\nonumber\\&\quad+6\varepsilon^{-2}\int_{\Gamma_t}
\int_{-\delta}^{\delta}Z^2(s)\theta\theta'^3\mu_2(x,t)\zeta^2(r)Jdr
\nonumber\\&\quad+6\varepsilon^{-2}\int_{\Gamma_t}
\int_{-\delta}^{\delta}Z^2(s)\chi^{(0)}d^{(0)}\gamma_3(z)
\theta\theta'^3\zeta^2(r)Jdr-C\int_{\Gamma_t} Z^2(s).
\end{align*}
 Retaining Lemma \ref{evenodd1}  that   $\gamma_\ell(z)\theta\theta'^3$ (with $\ell=1,2,3$) are odd functions. So it follows from  Lemma \ref{reduce2}  that
$I_{132}\geq-C\int_{\Gamma_t} Z^2(s)$. This together with
  (\ref{I4}), (\ref{I41})  yields \eqref{I4:lower bound}.
\end{proof}

 \begin{proof}[Proof of Proposition \ref{propofI1}]
 It follows from (\ref{zeta term})  and \eqref{laplace1}
 that
 \begin{align}
I_{14}&\geq-\e\int_{\Gamma_t} Z^2(s)-\e\int_{\Gamma_t} \big(\Delta_{\Gamma}\mid_{r=0}Z(s)\big)^2.\label{I14}
\end{align}
%On the other hand, expanding $J(r,s)$ in \eqref{I11} using \eqref{jacobian}
%and  by a change of variable $r=z/\e$,
%\begin{equation*}
%  \begin{split}
%I_{11}&\geq\varepsilon^{-3}\int_{\Gamma_t}
%\int_{-\delta}^{\delta}Z^2(s)h^2(s)\big(\theta''\big)^2\zeta^2(r)\(1+r h(s)+r^2 e(s)+z^3O(\e^3)\)dr-C\int_{\Gamma_t} Z^2(s)\\
%&\geq\varepsilon^{-3}\int_{\Gamma_t}
%\int_{-\delta}^{\delta}Z^2(s)h^2(s)\big(\theta''\big)^2\zeta^2(r)dr-2C\int_{\Gamma_t} Z^2(s).
%  \end{split}
%\end{equation*}
%Note that in the second step,  we can integrate the function $z\big(\theta''\big)^2\zeta^2(r)$  in $r$ and obtain a vanishing integral.
  Substituting  \eqref{I11},   \eqref{I12}, (\ref{I4:lower bound}) and (\ref{I14})  into (\ref{I1}) leads to
\begin{align*}
I_{1}&\geq\varepsilon^{-3}\int_{\Gamma_t}Z^2(s)
h^2(s)\int_{-\delta}^{\delta} \( \theta''^2-3z\theta\theta'^3\)\zeta^2(r)dr
+\frac{\Lambda_1}{2}\int_{\Gamma_t} \big(\Delta_{\Gamma}\mid_{r=0}Z(s)\big)^2
-C\int_{\Gamma_t} Z^2(s)
%\nonumber\\&=\varepsilon^{-2}\int_{\Gamma_t}Z^2(s)h^2(s)
%\int_{-\frac{1}{\varepsilon}}^{\frac{1}{\varepsilon}}
%\(\theta''^2-3z\theta(\theta')^3\)dz+\frac{\Lambda_1}{2}
%\int_{\Gamma_t} \big(\Delta_{\Gamma}\mid_{r=0}Z(s)\big)^2
%-C\int_{\Gamma_t} Z^2(s)
%\nonumber\\&\quad+\varepsilon^{-3}\int_{\Gamma_t}\int_{-\delta}%^{\delta}Z^2(s)h^2(s)
%\underbrace{\(\theta''^2-3z\theta(\theta')^3\)\big(\zeta^2(r)-1\big)}%_{O(e^{-\frac C\e})}dr.
\end{align*}
{To treat the first integral on the right,  we use the exponential  decay of the integrand:
\begin{equation}\label{vanishing integral}
\int_{-\delta}^{\delta}\( \theta''^2-3z\theta\theta'^3\)\zeta^2(r)dr=O(e^{-C/\e})+\e\int_{\mathbb{R}}\( \theta''^2-3z\theta\theta'^3\)\,dz.
\end{equation}
If the last integral vanishes, we shall have
\begin{align*}
I_{1}&\geq  \frac{\Lambda_1}{2}
\int_{\Gamma_t} \big(\Delta_{\Gamma}\mid_{r=0}Z(s)\big)^2
-C\int_{\Gamma_t} Z^2(s).\end{align*}
On the other hand, it follows from \eqref{gauss} that
\begin{equation}
 \int_{\Gamma_t}\left|\nabla_{\Gamma}\big|_{r=0}Z(s)\right|^2
 =-\int_{\Gamma_t}Z(s) \(\Delta_{\G}\big|_{r=0}Z(s)\).
\end{equation}
Using   the above two statements and choosing a large constant $C$ lead to \eqref{I1:lower bound}.}
To compute the last integral in \eqref{vanishing integral}, we first note that $\theta$ satisfies
 \begin{equation}\label{struggle1}
 \sqrt{2}\theta'=1-\theta^2,~\text{and}~-\sqrt{2}\theta\theta'=\theta''.
 \end{equation}
  Applying these formulas together with  integration by parts gives
\begin{align*}
& \int_{\mathbb{R}}\( \theta''^2-3z\theta  \theta'^3\)dz
= \int_{\mathbb{R}}\( \theta''^2+\tfrac{3}{\sqrt{2}}z\theta'' \theta'^2\)dz
= \int_{\mathbb{R}}\(2\(\theta\theta'\)^2-\tfrac{1}{\sqrt{2}}\theta'^3\)dz.
\end{align*}
Applying \eqref{struggle1} again and performing a change of variable,
\begin{align*}
& \int_{\mathbb{R}}\(2\(\theta\theta'\)^2-\tfrac{1}{\sqrt{2}}\theta'^3\)dz
=\tfrac{1}{\sqrt{2}}\int_{\mathbb{R}}\big(\theta'\big)^2\big(2\sqrt{2}\theta^2-\theta'\big)dz
=\tfrac{1}{2\sqrt{2}}\int_{-1}^1\big(1-\theta^2\big)\big(5\theta^2-1\big)d\theta =0.
\end{align*}
\end{proof}
\section{Spectrum Condition: Estimates of  cross terms}\label{section:spectumI2}
 We shall prove  the following result in the end of this section.\begin{proposition}\label{propcross}
 There exist $C>0$ and $\e_0>0$  such that
 \begin{equation}
\begin{split}
I_{2}&\geq-C\int_{\Omega}\phi^2dx -\nu\int_{\Gamma_t} \big(\Delta_{\G}\big|_{r=0}Z(s)\big)^2
-\frac{1}{2}I_{31}-\nu\varepsilon^{-4}\int_{\G_t(\delta)}
\big(\phi^\bot\zeta\big)^2dx\\&\quad
-\nu\varepsilon^{-2}\int_{\G_t(\delta)}\big|\nabla_{\Gamma}
\big(\phi^\bot\zeta\big)\big|^2dx
 -\frac{C \e}\nu  \int_{\G_t(\delta)}\big(\Delta_
{\Gamma}\big|_{r=0}\phi^\bot\big)^2\zeta^2dx\label{I2:lower bound}
\end{split}
\end{equation}
holds for every $\e\in (0,\e_0)$ and $\nu \in (\e^{\frac 1{100}},1)$.
 \end{proposition}

Recall  \eqref{nasty} for the definitions of $I_2$ and $I_{31}$.
Under  coordinates  $(r,s)\in (-2\delta,2\delta)\times\Gamma_t$, we can employ \eqref{eq:laplace2} to write  $I_2$ as
\begin{equation}\label{decompose:I2}
  I_2= I_{21}+I_{22}+I_{23}+I_{24}+I_{25},
\end{equation}
where we define
\begin{subequations}
  \begin{align}
I_{21}=&-2\varepsilon^{-\frac{7}{2}}\int_{\Gamma_t}\int_{-\delta}^{\delta} Z(s)\big(\theta'''(z)-f''(\phi_a)\theta'(z)\big)\zeta(r)
\mathscr{L}_\e[\phi^\bot]J(r,s)dr,\\
I_{22}=&-2\varepsilon^{-\frac{5}{2}}\int_{\Gamma_t}\int_{-\delta}^{\delta}Z(s)\theta''(z)\zeta(r)\p_r (\ln \sqrt{g})
\mathscr{L}_\e[\phi^\bot]
J(r,s)dr,
\\I_{23}=&-2\varepsilon^{-\frac{3}{2}}\int_{\Gamma_t}
\int_{-\delta}^{\delta}\Delta_{\Gamma}Z(s) \theta'(z)\zeta(r)
\mathscr{L}_\e[\phi^\bot]J(r,s)dr,
\\I_{24}=&-2\varepsilon^{-1}\int_{\Gamma_t}\int_{-\delta}^{\delta} Z(s)A(r,s)\mathscr{L}_\e[\phi^\bot]J(r,s)dr,
\\I_{25}=&2\varepsilon^{-\frac{7}{2}}\int_{\Gamma_t}\int_{-\delta}^{\delta}f'''(\phi_a)\mu_a Z(s)\theta'(z)\zeta(r)\phi^\bot J(r,s)dr.\label{def:I25}
\end{align}
\end{subequations}
%and
%\begin{align*}
%I_3=\varepsilon^{-2}\int_{\Gamma^\varepsilon}\int_{-1}^{1}\bigg(\varepsilon\big(\p_r^2\phi^R+\partial_{r}\phi^R\Delta d_\e+\Delta_{\Gamma^\varepsilon}\phi^R\big)-\varepsilon^{-1}f''(\phi_a)\phi^R\bigg)^2drds.
%\end{align*}
%
% Moreover, noting that
%\begin{align*}
%\varepsilon^{-3}f'''(\phi_a)\mu_a\phi^2&=\varepsilon^{-4}f'''(\phi_a)\mu_aZ^2(s,t)\big(\theta'(z)\big)^2+2\varepsilon^{-\frac{7}{2}}f'''(\phi_a)\mu_a Z(s,t)\theta'(z)\phi^R\nonumber\\&\quad+\varepsilon^{-3}f'''(\phi_a)\mu_a\big(\phi^R\big)^2,
%\end{align*}
%then
%\begin{align}
%\frac{1}{\varepsilon^3}\int_{\Gamma^\varepsilon(\delta)} f'''(\phi_a)\mu_a\phi^2dx&=\varepsilon^{-4}\int_{\Gamma^\varepsilon}\int_{-1}^{1}f'''(\phi_a)\mu_aZ^2(s,t)\big(\theta'(z)\big)^2J(r,s)drds
%\nonumber\\&\quad+2\varepsilon^{-\frac{7}{2}}\int_{\Gamma^\varepsilon}\int_{-1}^{1}f'''(\phi_a)\mu_a Z(s,t)\theta'(z)\phi^RJ(r,s)drds
%\nonumber\\&\quad+\varepsilon^{-3}\int_{\Gamma^\varepsilon}\int_{-1}^{1}f'''(\phi_a)\mu_a\big(\phi^R\big)^2J(r,s)drds
%\nonumber\\&\triangleq I_4+I_5+I_6.\label{singular term-2}
%\end{align}

%\subsubsection{$I_4,I_5,I_6$}
%\indent
%
%In this subsection we focus on the second term in (\ref{singular term}), that is, $I_4+I_5+I_6$.

\begin{lemma}
 There exist $C>0$ and $\e_0>0$   such that  \begin{equation}\label{I21}
I_{21}\geq-C\int_{\Gamma_t} Z^2(s)-\frac{1}{4}I_{31},\qquad \forall \e\in (0,\e_0).
\end{equation}
\end{lemma}
\begin{proof}
  We first recall \eqref{phi12} for   $\widetilde{\phi}^{(2)}$ as well as (\ref{decompose energy-3}) for the definition of $I_{31}$.
It follows from (\ref{good  term}) and the Cauchy-Schwarz inequality that
\begin{equation}
  \begin{split}
    I_{21}&=2\varepsilon^{-\frac{3}{2}} \int_{\Gamma_t}\int_{-\delta}^{\delta}Z(s)f'''(\theta)\widetilde{\phi}^{(2)}\theta'\zeta(r)
\mathscr{L}_\e[\phi^\bot]J(r,s)dr
\\&\quad-2\varepsilon^{-\frac{1}{2}}\int_{\Gamma_t}\int_{-\delta}^{\delta} Z(s)\theta'O(1)\zeta(r)
\mathscr{L}_\e[\phi^\bot]J(r,s)dr
\\&\geq -C\varepsilon^{-1}\int_{\Gamma_t}\int_{-\delta}^{\delta}Z^2(s)\theta'^2\zeta^2(r)J(r,s)dr-\frac{1}{4}I_{31}.
%\\&\geq-C\int_{\Gamma_t} Z^2(s)-\frac{1}{4}I_{31}.
  \end{split}
\end{equation}
So  the desired result  follows from \eqref{jacobian} and  a change of variable.
\end{proof}
\begin{lemma}
 There exist $C>0$ and $\e_0>0$   such that
\begin{align}
I_{22}+I_{25}&\geq-\nu\int_{\Gamma_t} \big(\Delta_{\Gamma}|_{r=0}Z\big)^2-C(1+\nu^{-1})\int_{\Gamma_t} Z^2(s)-\nu\varepsilon^{-4} \int_{\G_t(\delta)}
\big(\phi^\bot\zeta\big)^2dx\nonumber\\&\quad-Ce^{-\frac{C}{\varepsilon}}\int_{\G_t(\delta)\backslash\G_t(\delta/2)}\big(\phi^\bot\big)^2dx
-\nu\varepsilon^{-2}\int_{\G_t(\delta)}\big|\nabla_{\Gamma}
\big(\phi^\bot\zeta\big)\big|^2dx
-\frac{1}{8}I_{31}\label{I22I25}
\end{align}
holds for every $\e\in (0,\e_0)$ and $\nu\in (\e^{\frac 1{100}},1)$.
\end{lemma}

\begin{proof}
\underline{Estimate of $I_{22}$:}\  It follows from \eqref{eq:2.1},  a change of variable $z=r/\e$ and the Cauchy-Schwarz inequality that
\begin{align*}
I_{22}&=-2\varepsilon^{-\frac{5}{2}}\int_{\G_t(\delta)}Z(s)\theta''\zeta(r)\big(h(s)+ O(\e)z\big)
\mathscr{L}_\e[\phi^\bot]dx
\nonumber\\&\geq-2\varepsilon^{-\frac{5}{2}}\int_{\G_t(\delta)}Z(s)\theta''\zeta(r)h(s)
\mathscr{L}_\e[\phi^\bot]dx
-C\int_{\Gamma_t} Z^2(s)-\frac{1}{8}I_{31}.
\end{align*}
On the other hand, using (\ref{eq:expan phia})  we can write   $f''(\phi_a)=f''(\theta)+O(\e^2).$  This together with   integration by parts leads to
\begin{align}
I_{22}&\geq\underbrace{2\varepsilon^{-\frac{3}{2}}\int_{\G_t(\delta)}\Delta \(Z(s)h(s)\theta''\zeta(r)\)\phi^\bot dx}_{\triangleq I_{221}}
- 2\varepsilon^{-\frac{7}{2}}\int_{\Gamma_t}\int_{-\delta}^{\delta} Z(s)h(s)\theta''f''(\theta)\phi^\bot \zeta(r)J(r,s)dr
\nonumber\\&\quad-\underbrace{ \int_{\Gamma_t}\int_{-\delta}^{\delta} O(\varepsilon^{-\frac{3}{2}}) Z(s)h(s)\theta'' \phi^\bot \zeta(r) J(r,s) dr}_{\triangleq I_{222}}
-C\int_{\Gamma_t} Z^2(s)-\frac{1}{8}I_{31}.\label{I22-1}
\end{align}
 Regarding $I_{221}$, we have the following identity due to (\ref{eq:laplace1}),
 \begin{equation}\label{expansion11}
 \begin{split}
 &2\varepsilon^{-\frac{3}{2}} \Delta \big(Z(s)h(s)\theta''(z)\zeta(r)\big)\phi^\bot  \\&=2\varepsilon^{-\frac{3}{2}}\(\Delta Z(s)h(s)+2\nabla Z(s)\cdot \nabla h(s)+ Z(s)\Delta h(s)\)\theta''(z)\zeta(r)\phi^\bot
\\
&\quad+2\varepsilon^{-\frac{3}{2}}Z(s)h(s)\(\varepsilon^{-2}\theta''''(z)\zeta(r)
+2\varepsilon^{-1}\theta'''(z)\zeta'(r)+\theta''(z)\zeta''(r)\)\phi^\bot \\&\quad+2\varepsilon^{-\frac{3}{2}}Z(s)h(s)
\(\varepsilon^{-1}\theta'''(z)\zeta(r)+
\theta''(z)\zeta'(r)\)\p_r (\ln \sqrt{g})\phi^\bot,~\text{where}~z=r/\e.
 \end{split}
 \end{equation}
 Note that the term  $4\varepsilon^{-\frac{3}{2}}\nabla \big(Z(s)h(s)\big)\cdot\nabla\big(\theta''(z)\zeta(r)\big)\phi^\bot
$ which would have appeared in \eqref{expansion11} vanishes because of the decomposition \eqref{gradient} and orthogonality. Note that $\theta'',\theta'''=O(e^{-\frac C\e})$ in $\G_t(\delta)\backslash \G_t(\delta/2)$, and these relations hold for  the terms that are multiplied by  the derivatives of $\zeta(r)$. So   in $\G_t(\delta)$ 
\begin{equation*}
\begin{split}
&2\varepsilon^{-\frac{3}{2}} \Delta \big(Z(s)h(s)\theta''(\tfrac r\e)\zeta(r)\big)\phi^\bot
\\&=2\varepsilon^{-\frac{3}{2}}\(\Delta Z(s)h(s)+2\nabla Z(s)\cdot\nabla h(s)+Z(s)\Delta h(s)\)\theta''(\tfrac r\e)\zeta(r)\phi^\bot
\\&\quad+2\varepsilon^{-\frac{3}{2}}Z(s)h(s)
\(\varepsilon^{-2}\theta''''(\tfrac r\e)
+\varepsilon^{-1}\theta'''(\tfrac r\e)\p_r (\ln \sqrt{g})
\)\zeta(r)\phi^\bot+Z(s)h(s)\phi^\bot O(e^{-\frac C\e})\chi_{\G_t(\delta)\backslash \G_t(\delta/2)}
\end{split}
\end{equation*}
where  $\chi_A$ denotes the characteristic function of a set $A$.
Substituting the above formula into $I_{221}$ in  \eqref{I22-1}, integrating by parts (in $x$-variable) and using \eqref{gradient}, \eqref{eq:laplace1} yield 
\begin{align*}
I_{221}&=2\varepsilon^{-\frac{3}{2}}\int_{\G_t(\delta)}\Big(\Delta Z(s)h(s)+2\nabla Z(s)\cdot \nabla h(s)+ Z(s)\Delta h(s)\Big)\theta''(\tfrac r\e)\zeta(r)\phi^\bot dx
\nonumber\\&\quad+2\varepsilon^{-\frac{7}{2}}
\int_{\G_t(\delta)}Z(s)h(s)
\theta''''(\tfrac r\e)\zeta(r)\phi^\bot dx\nonumber\\&\quad+2\varepsilon^{-\frac{5}{2}}\int_{\G_t(\delta)}
Z(s)h(s) \theta'''(\tfrac r\e)\zeta(r)\p_r (\ln \sqrt{g})\phi^\bot dx+\int_{\G_t(\delta)\backslash \G_t(\delta/2)}Z(s)\phi^\bot O(e^{-\frac C\e}) dx
\nonumber\\&=2\varepsilon^{-\frac{3}{2}}\int_{\G_t(\delta)}
\(\Delta_\G Z(s)h(s)-Z(s)\Delta_\G h(s)\)\theta''(\tfrac r\e)\zeta(r)\phi^\bot dx
\nonumber\\&\quad-4\varepsilon^{-\frac{3}{2}}\int_{\G_t(\delta)}Z(s)\nabla_\G h(s)\cdot\nabla_\G\big(\theta''(\tfrac r\e)\zeta(r)\phi^\bot\big) dx
+2\varepsilon^{-\frac{7}{2}}
\int_{\G_t(\delta)}Z(s)h(s)
\theta''''(\tfrac r\e)\zeta(r)\phi^\bot dx\nonumber\\&\quad+2\varepsilon^{-\frac{5}{2}}\int_{\G_t(\delta)}
Z(s)h(s) \theta'''(\tfrac r\e)\p_r (\ln \sqrt{g})\zeta(r)\phi^\bot dx+\int_{\G_t(\delta)\backslash \G_t(\delta/2)}Z(s)\phi^\bot O(e^{-\frac C\e}) dx.
\end{align*}
In the second step  above  we used $\nabla h(s)\cdot\nabla\(\theta''(\tfrac r\e)\zeta(r)\phi^\bot\)=\nabla_\G h(s)\cdot\nabla_\G\(\theta''(\tfrac r\e)\zeta(r)\phi^\bot\).$ Then 
it follows from  the  change of variable $x\mapsto (r,s)$ that
\begin{align*}
&I_{221}=2\int_{\Gamma_t}
\int_{-\delta}^{\delta} \e^{\frac 12} \Big(\Delta_{\Gamma} Z(s)h(s)-Z(s)\Delta_{\Gamma} h(s)\Big)\theta'' \e^{-2}(\zeta\phi^\bot) Jdr
\nonumber\\&-4\int_{\Gamma_t}\int_{-\delta}^{\delta}\e^{-\frac 12}Z(s) \nabla_{\Gamma} h(s)\cdot\e^{-1}\nabla_{\Gamma}(\zeta\phi^\bot ) \theta''  J dr
+2\varepsilon^{-\frac{7}{2}}\int_{\Gamma_t}\int_{-\delta}^{\delta}Z(s)h(s)\theta'''' (\zeta\phi^\bot) Jdr\nonumber\\&
+2\int_{\Gamma_t}\int_{-\delta}^{\delta}
\e^{-\frac 12}Z(s)h(s) \theta'''(\tfrac r\e)\p_r (\ln \sqrt{g})  \e^{-2}(\zeta\phi^\bot) Jdr+\int_{\G_t(\delta)\backslash \G_t(\delta/2)}Z(s)\phi^\bot O(e^{-\frac C\e}) dx.
\end{align*}  In this formula, we write the power of $\e$ separately for the convenience of applying the Cauchy-Schwarz inequality \eqref{cauchy-schwarz}.
Note that in the above expansion of $I_{221}$, the leading term is the one with $\e^{-7/2}$, which shall be cancelled with the leading order term in $I_{25}$.  Recall that $h(s)$ is a smooth function on $\G_t$ and $\Delta_\G$ depends on $r$. We apply  the Cauchy-Schwarz inequality \eqref{cauchy-schwarz} to the first, second, and the fourth terms above and yield
\begin{equation}
\begin{split}
I_{221}&\geq -C\e^2 \sup_{|r|\leq \delta}\int_{\G_t}  |\Delta_{\G} Z|^2
+2\varepsilon^{-\frac{7}{2}}
\int_{\Gamma_t}\int_{-\delta}^{\delta}Z(s)h(s)\theta''''(\zeta
\phi^\bot) J dr \\&\quad
-C(1+\nu^{-1})\int_{\Gamma_t} Z^2(s)-\nu\varepsilon^{-4}
\int_{\G_t(\delta)}\big(\phi^\bot\zeta\big)^2dx\\&\quad
-Ce^{-\frac{C}{\varepsilon}}
\int_{\G_t(\delta)\backslash\G_t(\delta/2)}
\big(\phi^\bot\big)^2dx
-\nu\varepsilon^{-2}\int_{\G_t(\delta)}
\big|\nabla_{\Gamma}
 (\zeta \phi^\bot) \big|^2dx.\label{I22-2}
\end{split}
\end{equation}
In a similar way, we can treat  $I_{222}$ by
\begin{align}
I_{222}&=\int_{\Gamma_t}\int_{-\delta}^{\delta} O(\varepsilon^{-\frac{3}{2}}) Z(s)h(s)\theta''(z)f'''(\theta)\big(\phi^\bot \zeta+\phi^\bot (1-\zeta)\big) dr
\nonumber\\& \geq-C\int_{\Gamma_t} Z^2(s)-C\varepsilon^{-2}\int_{\G_t(\delta)}\big(\phi^\bot\zeta\big)^2dx
-Ce^{-\frac{C}{\varepsilon}}\int_{\G_t(\delta)\backslash\G_t(\delta/2)}\big(\phi^\bot\big)^2dx.\label{I22-3}
\end{align}
By substituting  (\ref{I22-2}) and (\ref{I22-3}) into (\ref{I22-1}) and then using \eqref{laplace1} to treat the first term on the right of   \eqref{I22-2}, we arrive at
\begin{equation}\label{I22-4}
  \begin{split}
    I_{22}&\geq2\varepsilon^{-\frac{7}{2}}\int_{\Gamma_t}\int_{-\delta}^{\delta} Z(s)h(s)\big(\theta''''(z)-\theta''(z)f''(\theta)\big)
(\phi^\bot \zeta) Jdr
\\&\quad-\nu\int_{\Gamma_t} \big(\Delta_{\Gamma}|_{r=0}Z\big)^2-C(1+\nu^{-1})\int_{\Gamma_t} Z^2(s) -\nu\varepsilon^{-4}\int_{\G_t(\delta)}\big(\phi^\bot\zeta\big)^2dx\\&\quad-Ce^{-\frac{C}{\varepsilon}}\int_{\G_t(\delta)\backslash\G_t(\delta/2)}\big(\phi^\bot\big)^2dx
-\nu\varepsilon^{-2}\int_{\G_t(\delta)}\big|\nabla_{\Gamma}\big(\phi^\bot\zeta\big)\big|^2dx
-\frac{1}{8}I_{31}.
  \end{split}
\end{equation}
\underline{Estimate of $I_{25}$:}
It follows from \eqref{d} and \eqref{app distance} that $\Delta r=\Delta d^{(0)}+\e \Delta d^{(1)}+O(\e^2)$. This together with \eqref{eq:expan phia}, \eqref{mu12} and \eqref{eq:2.1} implies  the following expansion successively 
 \begin{subequations}
   \begin{align}
     \mu_a&=-\Delta d^{(0)}\theta'-\e \Delta d^{(1)}\theta'+\e D^{(0)}z\theta'+O(\e^2)\nonumber\\
     &=-h\theta'+\e (D^{(0)}-b)z\theta'+(1+z^2)O(\e^2),~\text{with}~z=r/\e.\\
     f'''(\phi_a)\mu_a &=-h\theta'f'''(\theta)+\e (D^{(0)}-b)z\theta'f'''(\theta)+(1+z^2)O(\e^2),~\text{with}~z=r/\e.
   \end{align}
 \end{subequations}
Substituting into \eqref{def:I25} yields
\begin{equation}
\begin{split}
I_{25}
&=-2\varepsilon^{-\frac{7}{2}}\int_{\Gamma_t}\int_{-\delta}^{\delta}Z(s)h(s)f'''(\theta)\theta'^2(\phi^\bot \zeta)Jdr
\\&\quad+\int_{\Gamma_t}\int_{-\delta}^{\delta}\e^{-\frac 12}Z(s)(D^{(0)}-b)zf'''(\theta)\theta'^2 \e^{-2}(\phi^\bot \zeta) J dr\\&\quad
+\e^{-\frac 32} \int_{\G_t}\int_{-\delta}^{\delta} O(1)\theta'(z)(1+z^2)Z(s)(\phi^\bot\zeta)Jdr,~\text{with}~z=r/\e.
\end{split}
\end{equation}
Finally we  apply \eqref{cauchy-schwarz} to the last two components and this yields
\begin{align}
I_{25}&\geq-2\varepsilon^{-\frac{7}{2}}\int_{\Gamma_t}\int_{-\delta}^{\delta}Z(s)h(s)f'''(\theta)\theta'^2\phi^\bot\zeta J dr
-\frac C\nu\int_{\Gamma_t} Z^2(s)-\nu\varepsilon^{-4}\int_{\G_t(\delta)}
\big(\phi^\bot\zeta\big)^2dx.\label{I25}
\end{align}
%\footnote{We should avoid  having terms like $-%C\varepsilon^{-4}\int_{\G_t(\delta)}\big(\phi^\bot\zeta\big)^2dx$
%in a lower bound estimate
%unless the constant $C$ can be made sufficiently small such %that they   can be absorbed by the energy term in \eqref{I3-13} in %the sequel.}
On the other hand,  differentiating the identity   $\theta''(z)=f'(\theta)$ twice leads to $\theta''''(z)=\theta''(z)f''(\theta)+f'''(\theta)\big(\theta'(z)\big)^2.$
This together with (\ref{I22-4}) and (\ref{I25}) eliminates the $\e^{-7/2}$ order term in $I_{22}+I_{25}$. So we have proved  \eqref{I22I25}.
\end{proof}
\begin{lemma}
 There exist $C>0$ and $\e_0>0$   such that
 \begin{equation}\label{I23I24}
  \begin{split}
I_{23}&\geq-\nu\int_{\Gamma_t} \big(\Delta_{\G}\big|_{r=0}Z(s)\big)^2-C\int_{\Gamma_t} Z^2(s)-\nu\varepsilon^{-4}\int_{\G_t(\delta)}
\big(\phi^\bot\zeta\big)^2dx
\\&\quad-\frac C{\e \nu} \int_{\G_t(\delta)}\big(\phi^\bot\zeta\big)^2dx
-\frac{C \e}\nu  \int_{\G_t(\delta)}\big(\Delta_{\G}\big|_{r=0}\phi^\bot\big)^2\zeta^2dx
  \end{split}
\end{equation}
holds for every $\e\in (0,\e_0)$ and $\nu\in (\e^{\frac 1{100}},1)$.
 \end{lemma}
 \begin{proof}
Using \eqref{eq:laplace1}, \eqref{cutoff} and integrating by parts in $r$ twice, we can write   $I_{23}$ by
\begin{align}
I_{23}&=2\varepsilon^{-\frac{3}{2}}\int_{\Gamma_t}
\int_{-\delta}^{\delta}\Delta_{\Gamma}Z(s) \theta'\zeta(r)
\(\e\Delta_{\G} \phi^\bot+ \tfrac {\e}{\sqrt{g}} \p_{r}\(\sqrt{g} \p_{r} \phi^\bot\)
-\varepsilon^{-1}f''(\phi_a)\phi^\bot\)
J\,dr\nonumber \\
&=2\varepsilon^{-\frac{1}{2}}\int_{\Gamma_t}\int_{-\delta}^{\delta}
\p_r\(\sqrt{g}\p_r\(\Delta_{\Gamma}Z(s)
 \theta'\zeta(r)\)\)\phi^\bot \tfrac{1}{\sqrt{g}(0,s)}dr,\nonumber\\
 &\quad-2\varepsilon^{-\frac{5}{2}}
\int_{\Gamma_t}\int_{-\delta}^{\delta}
\Delta_{\Gamma}Z(s) \theta'\zeta(r)
 f''(\phi_a)\phi^\bot Jdr\nonumber\\
&\quad+2\varepsilon^{-\frac{1}{2}}\int_{\Gamma_t}
\int_{-\delta}^{\delta}\Delta_{\Gamma}Z(s) \Delta_{\Gamma}\phi^\bot\zeta(r) \theta' Jdr
\triangleq J_1+J_2+J_3.
\label{I231}
\end{align}
To treat $J_1$, we notice that its leading order term corresponds to the case when $\theta'(\tfrac r\e)$ is differentiated twice, which gives rise to a factor $\e^{-2}$. For the rest terms, the differential operator $\p_r$ will apply to $\Delta_{\G} Z(s)$, whose coefficients depends on $r$ smoothly. So for any $\ell\geq 0$, we employ \eqref{laplace1} to obtain
  \begin{equation}\label{ellipticest}
    \int_{\G_t}\int_{-\delta}^{\delta}|\p_r^\ell \Delta_{\G} Z(s)|^2\leq C_\ell \|Z\|^2_{H^2(\G_t)}\leq C\inf_{|r|\leq \delta}\int_{\G_t} \(|\Delta_{\G} Z(s)|^2+| Z(s)|^2\).
  \end{equation}
   As a result, we can subtract the leading order terms in the expansion of  $J_1+J_2$ and the remaining ones can be controlled effectively:
  \begin{equation*}
    \begin{split}
      &J_1+J_2-\underbrace{2\e^{-\frac 52}\int_{\G_t}\int_{-\delta}^{\delta}\Delta_{\G} Z(s)\(\theta'''(z)-\theta'(z)f''(\phi_a)\)\zeta(r) \phi^\bot J dr}_{\triangleq \hat{J}}\\
      \geq &-\nu\int_{\Gamma_t} \big(\Delta_{\Gamma}{\big|_{r=0}}Z(s)\big)^2\int_{-\frac{2}
{\varepsilon}}^{\frac{2}{\varepsilon}}\(\theta''(z)^2+\theta'(z)^2\)dz
-C\int_{\Gamma_t} Z^2(s)-\nu\varepsilon^{-4}\int_{\G_t(\delta)}
\big(\phi^\bot\zeta\big)^2dx.
    \end{split}
  \end{equation*}
  Employing \eqref{good  term} will gain a factor  $\e^2$ for  $\hat{J}$. This together with \eqref{laplace2} implies  \begin{equation}\label{lowerboundJ1J2}
J_1+J_2
      \geq -\nu\int_{\Gamma_t} \big(\Delta_{\Gamma}{\big|_{r=0}}Z(s)\big)^2
-C\int_{\Gamma_t} Z^2(s)-\nu\varepsilon^{-4}\int_{\G_t(\delta)}
\big(\phi^\bot\zeta\big)^2dx,
  \end{equation}
by choosing a  sufficiently  small $\nu>0$.
%\footnote{If we treat the two terms in $\hat{J}$ separately, we would end up with  a term like $-C\varepsilon^{-4}\int_{\G_t(\delta)}\big(\phi^\bot\zeta\big)^2dx$, which can not be controlled.}

Concerning $J_3$,
we shall employ the relation  $\varphi J^{\frac 12}- \theta'J =O(e^{-\frac C\e})+z\theta'(z) O(\e)$ with $z=r/\e$. 
This  is a consequence of
$\varphi- \theta'=O(e^{-\frac{C}{\varepsilon}})$ and $J-J^{\frac{1}{2}}=O(\varepsilon)z$,
which follow from  \eqref{difference of kernal} and  \eqref{jacobian} respectively. So we can treat $J_3$  by the Caucy-Schwarz inequality and \eqref{laplace3},
\begin{equation}\label{J_3terms}
 \begin{split}
    J_3&\geq 2 \underbrace{\varepsilon^{-\frac{1}{2}}\int_{\Gamma_t}
\int_{-\delta}^{\delta}\Delta_{\Gamma}Z(s) \Delta_{\Gamma}\phi^\bot\zeta(r)\varphi J^{\frac 12}dr}_{\triangleq J_*}
\\&-C\e\int_{\Gamma_t} \(\big(\Delta_{\Gamma}|_{r=0}Z\big)^2+Z^2(s)\)-C\e\int_{\G_t(\delta)} \big(\Delta_{\Gamma}\phi^\bot\big)^2\zeta^2dx.
 \end{split}
\end{equation}
The estimate of $J_*$ is derived at \eqref{estJ5} in the sequel. 
So substituting  \eqref{lowerboundJ1J2}, \eqref{J_3terms} and   \eqref{estJ5} into 
\eqref{I231} and then choosing a sufficiently small $\nu$ yield
\begin{equation}
\begin{split}
I_{23}\geq &-\nu\int_{\Gamma_t} \big(\Delta_{\G}\big|_{r=0}Z(s)\big)^2-C\int_{\Gamma_t} Z^2(s)-\nu\varepsilon^{-4}\int_{\G_t(\delta)}
\big(\phi^\bot\zeta\big)^2dx
\\&\quad-\frac C\nu \int_{-\delta}^{\delta}\(\int_{\Gamma_t}
\left|\nabla_{\Gamma}\big|_{r=0} \phi^\bot \right|^2 \)\zeta^2dr
-C \e  \int_{\G_t(\delta)}\big(\Delta_{\G}\big|_{r=0}
(\phi^\bot\zeta)\big)^2dx.
\end{split}
\end{equation}
To eliminate the term with  $\nabla_\G|_{r=0}\,\phi^\bot$, we employ  \eqref{gauss} and obtain
\begin{align*}
 \int_{\Gamma_t}\left|\nabla_{\Gamma}\big|_{r=0}\phi^\bot\right|^2
 =-\int_{\Gamma_t}\phi^\bot\big(\Delta_{\G}\big|_{r=0}\phi^\bot\big)
 \leq\varepsilon^{-1}\int_{\Gamma_t}\big(\phi^\bot\big)^2+\frac{\varepsilon}{4}\int_{\Gamma_t}\big(\Delta_{\G}\big|_{r=0}\phi^\bot\big)^2.
\end{align*}
The above two inequalities together lead to \eqref{I23I24}.
\end{proof}
{
\begin{proof}[Proof of Proposition \ref{propcross}]Retaining \eqref{zeta term}, $I_{24}$ has the same lower bound as $I_{23}$. So   substituting (\ref{I21}), (\ref{I22I25}) and (\ref{I23I24}) into (\ref{decompose:I2}), and using (\ref{decompose L2norm}) lead to \eqref{I2:lower bound}.
It remains to treat $J_*$ in \eqref{J_3terms}. We shall show  there exist $\e_0>0$ and $C>0$ such that
  \begin{equation}\label{estJ5}
    \begin{split}
   J_*  & \geq-\nu\int_{\Gamma_t} \big(\Delta_{\G}\big|_{r=0}Z(s)\big)^2 -\frac C \nu \int_{\Gamma_t}\int_{-\delta}^{\delta}
   \left|\nabla_{\Gamma}\big|_{r=0}(\phi^\bot\zeta)\right|^2  dr-\frac C \nu \int_{\G_t(\delta)}\big(\phi^\bot\zeta\big)^2dx\\
   &\qquad -C\int_{\Gamma_t} Z^2(s)-\frac{C\e^2}\nu \int_{\G_t(\delta)}
 \big(\Delta_{\G}\big|_{r=0}\phi^\bot\zeta\big)^2dx
    \end{split}
  \end{equation}
  holds for every $\e\in (0,\e_0)$ and every $\nu\in (\e^{\frac 1{100}},1)$. To this end, we first  recall \eqref{roperator}    that $\Delta_{\G}=\Delta_{\G}\big|_{r=0}+r\RR_{\G}$. So 
\begin{equation*}
\begin{split}
J_*=&\varepsilon^{-\frac{1}{2}}\int_{\Gamma_t}
\int_{-\delta}^{\delta}\big(\Delta_{\G}\big|_{r=0}Z(s)\big)\big( \Delta_{\G}\big|_{r=0}\phi^\bot\big)\varphi(\tfrac r\e)\zeta(r)J^{\frac 12}dr\\
&+\varepsilon^{\frac{1}{2}}\int_{\Gamma_t}
\int_{-\delta}^{\delta}\big(\Delta_{\G}\big|_{r=0}Z(s)\big) \RR_{\G}\phi^\bot \tfrac r\e\varphi(\tfrac r\e) \zeta(r)J^{\frac 12}dr\\
&+\varepsilon^{\frac{1}{2}}\int_{\Gamma_t}
\int_{-\delta}^{\delta}\RR_{\G}Z(s)\big
(\Delta_{\G}\big|_{r=0}\phi^\bot\big) \tfrac r\e\varphi(\tfrac r\e) \zeta(r)J^{\frac 12}dr\\
&+\varepsilon^{\frac{3}{2}}\int_{\Gamma_t}
\int_{-\delta}^{\delta}\RR_{\G}Z(s) \RR_{\G}\phi^\bot \varphi(\tfrac r\e)(\tfrac r\e)^2\zeta(r)J^{\frac 12}dr\triangleq \sum_{1\leq i\leq 4}J_{5i}.
\end{split}
\end{equation*}
 To treat the highest order term $J_{51}$, we apply $\Delta_{\G}\big|_{r=0}$ to  (\ref{orthogonality}) and obtain
 \begin{equation*}
   0=\Delta_{\G}\big|_{r=0}\bigg(\int_{-\delta}^{\delta}
\varphi(\tfrac{r}{\varepsilon})\phi^\bot(r,s)\zeta(r)J^{\frac{1}{2}}(r,s)dr\bigg).
 \end{equation*}
This implies the following identity:
\begin{align*}
&\int_{-\delta}^{\delta}{\varphi}(\tfrac{r}{\varepsilon})
\zeta(r)\big(\Delta_{\G}\big|_{r=0}\phi^\bot(r,s)\big)J^{\frac{1}{2}}dr
\\&=-2\int_{-\delta}^{\delta}{\varphi}(\tfrac{r}{\varepsilon})\zeta(r)
\big(\nabla_{\Gamma}\big|_{r=0}\phi^\bot(r,s)\cdot
 \nabla_{\Gamma}\big|_{r=0}J^{\frac{1}{2}}\big)dr
-\int_{-\delta}^{\delta}{\varphi}(\tfrac{r}{\varepsilon})\zeta(r)\phi^\bot(r,s)\big(\Delta_{\G}\big|_{r=0}J^{\frac{1}{2}}\big)dr.
\end{align*}
Multiplying by $\Delta_{\G}\big|_{r=0}Z(s)$ and integrating  over $\G_t$ gives
\begin{align*}
J_{51}
%&=\varepsilon^{-\frac{1}{2}}\int_{\G_t}\int_{-\delta}^{\delta}
%\big(\Delta_{\G}\big|_{r=0}Z(s)\big)\big( \Delta_{\G}\big|_{r=0}\phi^\bot\big){\varphi}
%(\tfrac{r}{\varepsilon})\zeta(r)
% J^{\frac{1}{2}}(r,s)dr\nonumber\\
&=
-2\varepsilon^{-\frac{1}{2}} \int_{\G_t} \int_{-\delta}^{\delta}\big(\Delta_{\G}\big|_{r=0}Z(s)\big){\varphi}(\tfrac{r}{\varepsilon})\zeta(r)\big(\nabla_{\Gamma}\big|_{r=0}\phi^\bot(r,s)\cdot
 \nabla_{\Gamma}\big|_{r=0}J^{\frac{1}{2}}(r,s)\big)dr
\nonumber\\&-\varepsilon^{-\frac{1}{2}}\int_{\G_t} \int_{-\delta}^{\delta} \big(\Delta_{\G}\big|_{r=0}Z(s)\big){\varphi}
(\tfrac{r}{\varepsilon})\phi^\bot\zeta(r)
\big(\Delta_{\G}\big|_{r=0}J^{\frac{1}{2}}(r,s)\big) dr\\
&\geq-\nu\int_{\Gamma_t} \big(\Delta_{\G}\big|_{r=0}Z(s)\big)^2-\frac C \nu \int_{\Gamma_t}\int_{-\delta}^{\delta}
\left|\nabla_{\Gamma}\big|_{r=0}(\phi^\bot\zeta)\right|^2  dr-\frac C \nu \int_{\G_t(\delta)}\big(\phi^\bot\zeta\big)^2dx.
\end{align*}
%\footnote{Without using  (\ref{orthogonality}) there would be an uncontrollable term $-C\int_{\G_t(\delta)}\big( \Delta_{\G}\big|_{r=0}(\phi^\bot\zeta)\big)^2 dx$.}
Retain that $\RR_{\G}$  is a second order operator acting on $s$ while its coefficients depend  on $(r,s)\in (-2\delta,2\delta)\times \G_t$ smoothly. Using the  Cauchy-Schwarz inequality and \eqref{roperatorest} give the estimate of $J_{52}+J_{53}+J_{54}$:
\begin{align*}
 \sum_{2\leq i\leq 4}I_{5i}&\geq-\nu\int_{\Gamma_t} \(\big(\Delta_{\G}\big|_{r=0}Z(s)\big)^2+Z^2(s)\)-\frac{C\e^2}\nu \int_{\G_t(\delta)}
 \(\big(\Delta_{\G}\big|_{r=0}\phi^\bot\zeta\big)^2+(\phi^\bot \zeta)^2\)dx.\end{align*}
 The above two estimates lead to \eqref{estJ5}.
\end{proof}}

\section{Proof of Theorem \ref{spectral theorem}}\label{section:spectrumI3}
\indent
The 
proof of Theorem \ref{spectral theorem} will be done by the end of this section after establishing  two lemmas. The first one is concerned with   a lower bound of the   Allen-Cahn operator \eqref{omegadef}.
\begin{lemma}
There exists $\e_0>0$ such that for every $\e\in (0,\e_0)$,
\begin{align}
\tfrac{1}{\varepsilon^2}\int_{\Omega}|\mathscr{L}_\e[\phi^\bot]|^2 dx
 \gtrsim  \int_{\Omega}\big(\Delta\phi^\bot\big)^2dx+\tfrac{1}
{\varepsilon^2}\int_{\Omega}\big|\nabla\phi^\bot\big|^2dx+ \tfrac{1}{\varepsilon^4}\int_{\Omega}\big(\phi^\bot\big)^2dx.\label{I3-13}
\end{align}
\end{lemma}
 \begin{proof}
{
We first note that for sufficiently small $\e$, the above estimate holds trivially when $\Omega$ is replaced by $\Omega\backslash \G_t(\delta)$. This is because away from $\G_t$, there holds $f''(\phi_a)\geq \Lambda_a$ for some positive constant $\Lambda_a$. To proceed,  we recall
the following coercivity estimate
\begin{equation}\label{I3-4}
\int_\Omega\mathscr{L}_\e[\phi^\bot]\phi^\bot dx
 \gtrsim\int_{\Omega}
 \(\varepsilon\big|\nabla\phi^\bot\big|^2+
 \tfrac{1}{\varepsilon}\big(\phi^\bot\big)^2\)dx,
\end{equation}
which follows by applying  \cite[Lemma 2.4]{chen1994spectrum} with $\psi=\phi^\bot$.}
This combined with the Cauchy-Schwarz inequality  yields
\begin{equation}\label{I3-5}
\tfrac{1}{\varepsilon^2}\int_{\Omega}|\mathscr{L}_\e[\phi^\bot]|^2dx
\geq \tfrac{C}
{\varepsilon^2}\int_{\Omega}\big|\nabla\phi^\bot\big|^2dx+ \tfrac{C}{\varepsilon^4}\int_{\Omega}\big(\phi^\bot\big)^2dx.
\end{equation}
As a result of the inequality $(a+b)^2\geq a^2/2-b^2$,
we arrive at
\begin{equation}\label{Lest1}
  \tfrac{1}{\varepsilon^2}\int_{\Omega}|\mathscr{L}_\e[\phi^\bot]|^2dx
\geq \tfrac{1}
{2}\int_{\Omega}\big(\Delta\phi^\bot\big)^2dx-\tfrac{C_1}{\varepsilon^4}\int_{\Omega}\big(\phi^\bot\big)^2dx.
\end{equation}
Finally, multiplying \eqref{I3-5} by a sufficiently large constant and adding up to \eqref{Lest1} will lead to the desired inequality.
\end{proof}
The coercivity  of $\Delta \phi^\bot$ on the right of \eqref{I3-13} is given below:
\begin{lemma}There exists $\e_0>0$ such that for every $\e\in (0,\e_0)$,
\begin{align}
\int_{\Omega}\big(\Delta\phi^\bot\big)^2dx&\geq {\frac{1}{4}}\int_{\G_t(\delta)}\(\Delta_{\Gamma}\phi^\bot\)^2\zeta^2dx-C\int_\Omega\big(\big|\nabla \phi^\bot\big|^2+(\phi^\bot)^2\big)dx\nonumber
\\ &\quad+{\frac{1}{4}}\int_{\G_t(\delta)}\(\p_r^2(\phi^\bot\zeta)\)^2dx
+ {\frac{1}{4}}\int_{\G_t(\delta)}\big|\p_r \nabla_{\G}(\phi^\bot\zeta)\big|^2 dx.\label{lapalce improve}
\end{align}
\end{lemma}
\begin{proof}
It follows from \eqref{cutoff} and the inequality $(a+b)^2\geq a^2/2-b^2$ that
\begin{equation}\label{I3-7}
  \begin{split}
    \int_\Omega\big(\Delta\phi^\bot\big)^2dx& \geq \int_\Omega\big(\Delta\phi^\bot\zeta\big)^2dx =\int_\Omega\(\Delta\big(\phi^\bot\zeta\big)
    -2\nabla\phi^\bot\cdot\nabla\zeta-\phi^\bot\Delta\zeta\)^2dx
\\&\geq \frac{1}{2}\int_\Omega\big(\Delta\big(\phi^\bot\zeta\big)\big)^2dx-C\int_\Omega\big(\big|\nabla \phi^\bot\big|^2+(\phi^\bot)^2\big)dx.
  \end{split}
\end{equation}

%Thanks to \eqref{I3-5} we immediately get
%\begin{align}
%\frac{1}{\varepsilon^2}\int_\Omega\bigg(\varepsilon\Delta\phi^\bot-\frac{1}{\varepsilon}f''(\phi_a)\phi^\bot\bigg)^2dx
%&\geq  \frac{C}
%{\varepsilon^2}\int_{\Omega}\big|\nabla\phi^\bot\big|^2dx+ \frac{C}{\varepsilon^4}\int_{\Omega}\big(\phi^\bot\big)^2dx
%\nonumber\\&\quad+\frac{1}{2\varepsilon^2}\int_\Omega\bigg(\varepsilon\Delta\big(\phi^\bot\zeta\big)-\frac{1}{\varepsilon}f''(\theta)\phi^\bot\zeta\bigg)^2dx.\label{I3-8}
%\end{align}
  %\begin{equation*}
%\phi_a=\theta+\varepsilon^2\widetilde{\phi}^{(2)}+O(\varepsilon^3).
%\end{equation*}
It follows from \eqref{cutoff},   (\ref{eq:laplace1})   and the  Cauchy-Schwarz inequality that
\begin{equation}\label{lap im-1}
\begin{split}
\int_\Omega\big(\Delta\big(\phi^\bot\zeta\big)\big)^2dx
=&\int_{\G_t(\delta)} \(\Delta_{\G} (\phi^\bot\zeta)+\tfrac{\p_r g}{2g}\p_r (\phi^\bot\zeta) +\p_r^2(\phi^\bot\zeta)\)^2dx
\\=&\int_{\G_t(\delta)} \Big[\(\Delta_{\G} (\phi^\bot\zeta)\)^2+2\Delta_{\G} (\phi^\bot\zeta)\tfrac{\p_rg}{2g}\p_r(\phi^\bot\zeta)+2\Delta_{\G} (\phi^\bot\zeta)\p_r^2(\phi^\bot\zeta)
\\ &+\(\tfrac{\p_r g}{2g}\p_r (\phi^\bot\zeta)\)^2+2\tfrac{\p_r g}{2g}\p_r (\phi^\bot\zeta) \p_r^2(\phi^\bot\zeta)+\(\p_r^2(\phi^\bot\zeta)\)^2\Big]
\\ \geq &\tfrac{1}{2}\int_{\G_t(\delta)} \(\Delta_{\G} (\phi^\bot\zeta)\)^2dx-C\int_{\G_t(\delta)} \big(\p_r (\phi^\bot\zeta)\big)^2dx
\\  &+\tfrac{1}{2}\int_{\G_t(\delta)}\(\p_r^2(\phi^\bot\zeta)\)^2dx+2\int_{\G_t(\delta)}  \Delta_{\G} (\phi^\bot\zeta)\p_r^2(\phi^\bot\zeta)dx.
\end{split}
\end{equation}

To treat the last integral, using \eqref{divergence} and the coarea formula \eqref{coareaformula} yields
\begin{equation*}
\begin{split}
&\int_{\G_t(\delta)}  \Delta_{\G} (\phi^\bot\zeta)\p_r^2(\phi^\bot\zeta)dx
\\=& \int_{-\delta}^{\delta}\int_{\G_t^{r}}  \Div_{\G}\nabla_{\G} (\phi^\bot\zeta)\p_r^2 (\phi^\bot\zeta) d\mathcal{H}^{N-1}dr
\\=& \int_{-\delta}^{\delta}\int_{\G_t^{r}}  \Div_{\G}\(\nabla_{\G} (\phi^\bot \zeta)  \p_r^2(\phi^\bot\zeta) \) d\mathcal{H}^{N-1}dr-\int_{\G_t(\delta)}  \nabla_{\G} (\phi^\bot\zeta)  \cdot\nabla_{\G}\p_r^2(\phi^\bot\zeta)  dx.
\end{split}
\end{equation*}
Here $\mathcal{H}^{N-1}$ denotes the Hausdorff measure on $\G_t^{r}$ and we have $d\mathcal{H}^{N-1}=\sqrt{g}(r,s) ds$ under local coordinates.
The first term vanishes due to the divergence theorem \eqref{gauss} and   the second term can be treated using the identity $\p_r^2\nabla_{\G}-\nabla_{\G}\p_r^2
  =\p_r[\p_r,\nabla_{\G}]+[\p_r,\nabla_{\G}]\p_r$: 
\begin{equation*}
\begin{split}
&\int_{\G_t(\delta)}  \Delta_{\G} (\phi^\bot\zeta)\p_r^2(\phi^\bot\zeta)dx\\
 &=-\int_{\G_t}\int_{-\delta}^{\delta}  \nabla_{\G} (\phi^\bot\zeta)  \cdot\p_r^2 \nabla_{\G}(\phi^\bot \zeta) J(r,s) dr
 \\&\quad+\int_{\G_t}\int_{-\delta}^{\delta}  \nabla_{\G} ( \phi^\bot\zeta)  \cdot \big(\p_r[\p_r,\nabla_{\G}]+[\p_r,\nabla_{\G}]\p_r\big)
 (\phi^\bot\zeta) J(r,s) dr
\triangleq I_{31}+I_{32}.
\end{split}
\end{equation*}
As for $I_{31}$, we  can  integrate by parts twice in $r$ to get
\begin{align}
I_{31}
%&=\int_{\G_t}\int_{-\delta}^{\delta} \big|\p_r \nabla_{\G}(\phi^\bot\zeta)\big|^2 J(r,s)dr+\int_{\G_t}\int_{-\delta}^{\delta} \nabla_{\G}(\phi^\bot\zeta)\cdot\p_r \nabla_{\G}(\phi^\bot\zeta) \p_rJ(r,s)dr \nonumber\\&
 =\int_{\G_t}\int_{-\delta}^{\delta} \big|\p_r \nabla_{\G}(\phi^\bot\zeta)\big|^2 J(r,s)dr-\frac{1}{2}\int_{\G_t}\int_{-\delta}^{\delta} \big|\nabla_{\G}(\phi^\bot\zeta) \big|^2 \p_{r}^2J(r,s)dr.\label{I31}
\end{align}
As for $I_{32}$, it can be verified that    $ [\nabla_{\G},\p_r]$ is a first-order tangential differential operator. 
\footnote{
For any tangential differential operator $\sum_{1\leq i\leq N-1}a_i(r,s)\p_{s_i}$, its commutator with $\p_r$  is a first-order %{\it tangential} differential operator:
\begin{equation}\label{commutator}
\bigg[\sum_{1\leq i\leq N-1}a_i(r,s)\p_{s_i},\p_r\bigg] f=-\sum_{1\leq i\leq N-1}\p_r a_i(r,s) \p_{s_i}f.
\end{equation}
}
So the leading order terms of $I_{32}$ are the mixed derivatives of second order and can be controlled by the first term on the right of  (\ref{I31}):\begin{align*}
 I_{32}\geq-\frac 12\int_{\G_t}\int_{-\delta}^{\delta} \big|\p_r \nabla_{\G}(\phi^\bot\zeta)\big|^2 J(r,s)dr-C\int_{\G_t}\int_{-\delta}^{\delta} \big|\nabla_{\G}(\phi^\bot\zeta) \big|^2 dr.
\end{align*}
Adding up the previous two inequalities leads to 
\begin{align*}
 \int_{\G_t(\delta)}  \Delta_{\G} (\phi^\bot\zeta)\p_r^2(\phi^\bot\zeta)dx\geq&\frac 12\int_{\G_t}\int_{-\delta}^{\delta} \big|\p_r \nabla_{\G}(\phi^\bot\zeta)\big|^2 J(r,s)dr
-C\int_{\G_t} \int_{-\delta}^{\delta}\big|\nabla_{\G}(\phi^\bot\zeta) \big|^2 dr.
\end{align*}
Substituting this inequality into (\ref{lap im-1})  and combining with  (\ref{I3-7}) imply \eqref{lapalce improve}.
\end{proof}

\begin{proof}[{\bf Proof of Theorem \ref{spectral theorem}}]
{We first recall \eqref{decompose energy-3} for the definitions of $I_{31}$ and   of $I_3$. The term $-\frac 12 I_{31}$ on the right hand side of \eqref{I2:lower bound} can be absorbed by  $I_3$ by choosing a small $\e$.
 It follows from (\ref{decompose energy}), (\ref{I1:lower bound}), (\ref{I2:lower bound}) and  \eqref{laplace1}  that
\begin{equation}
\begin{split}
&\frac{1}{\varepsilon^2}\int_\Omega|\mathscr{L}_\e[\phi]|^2dx+\frac{1}{\varepsilon^3}\int_\Omega f'''(\phi_a)\mu_a\phi^2dx
\\&\geq\frac{1}{4\varepsilon^2}\int_\Omega
 |\mathscr{L}_\e[\phi^\bot]|^2dx+\frac{1}{\varepsilon^3}\int_\Omega f'''(\phi_a)\mu_a\big(\phi^\bot\big)^2dx-\nu\varepsilon^{-4}\int_{\G_t(\delta)}\big(\phi^\bot\zeta\big)^2dx\\&\quad-C\varepsilon^{-\frac{C}{\varepsilon}}\int_{\G_t(\delta)\backslash\G_t(\delta/2)}\big(\phi^\bot\big)^2dx
-\nu\varepsilon^{-2}\int_{\G_t(\delta)}\big|\nabla_{\Gamma}\big(\phi^\bot\zeta\big)\big|^2dx
\\&\quad +O(\e) \int_{\G_t(\delta)}\big(\Delta_
{\Gamma}\big|_{r=0}\phi^\bot\big)^2\zeta^2dx + \Lambda_4\|Z\|_{H^2(\G_t)}^2
-C\int_{\Omega}\phi^2dx,\label{I1I2I3:lower bound}
\end{split}
\end{equation}}
where  $\nu$ is a small positive constant to be determined later.
Combining (\ref{I3-13}) and (\ref{lapalce improve}) with (\ref{I1I2I3:lower bound}), and choosing   $\nu$ and $\varepsilon$ sufficiently small, we  deduce
\begin{equation}\label{spectrum}
\begin{split}
&\frac{1}{\varepsilon^2}\int_\Omega|\mathscr{L}_\e[\phi]|^2dx+\frac{1}{\varepsilon^3}\int_\Omega f'''(\phi_a)\mu_a\phi^2dx
\\&\geq \Lambda_4 \|Z\|^2_{H^2(\G_t)}
-\hat{C}\int_{\Omega}\phi^2dx+C_2\int_{\Omega}\(|\Delta \phi^\bot|^2+\e^{-2}\big|\nabla\phi^\bot\big|^2+ \e^{-4}\big(\phi^\bot\big)^2\)dx
\\&\geq 
\tilde{C} \mathcal{K}(t)-\hat{C}\int_{\Omega}\phi^2dx.
\end{split}
\end{equation}
Note that in the last step we apply the Sobolev embedding and then choose  a smaller $\e$ to recover \eqref{def K(t)}.  So  we obtain the desired estimate \eqref{spectral condition}. 
%\begin{align*}
%\frac{1}{\varepsilon^2}\int_{\Omega}|\mathscr{L}_\e[\phi]|^2dx+\frac{1}{\varepsilon^3}\int_{\Omega} f'''(\phi_a)\mu_a\phi^2dx\geq-C\int_{\Omega}\phi^2dx.
%\end{align*}
%It remains to recover the term $\e^4|\Delta\phi|^2$ on the right hand side. To this end,
%we apply the inequality $(a+b)^2\geq\frac{1}{2}a^2-b^2$ to the above estimate and multiply by $\e^4$:
%\begin{align*}
%\e^4\(\frac{1}{\varepsilon^2}\int_{\Omega}|\mathscr{L}_\e[\phi]|^2dx+\frac{1}{\varepsilon^3}\int_{\Omega} f'''(\phi_a)\mu_a\phi^2dx\)\geq \e^4\(\frac{1}{2}\int_{\Omega}(\Delta\phi)^2dx-\frac{C}{\varepsilon^4}\int_{\Omega}\phi^2dx\),
%\end{align*}
%Adding up the above two inequalities and choosing a small $\e$ leads to the desired inequality.
\end{proof}

\section{Proof of Theorem \ref{main theorem}}\label{section:estimate}
With  the spectral condition \eqref{spectrum}, we can prove  Theorem \ref{main theorem}. Recalling \eqref{omegadef} and \eqref{error equation}, one can verify that the nonlinear terms $\mathcal{H}$ can be written as
$\mathcal{H}=\mathcal{H}_1+\mathcal{H}_2$ where
\begin{subequations}\label{nonlinearterms}
  \begin{align}
\mathcal{H}_1&=\e^{-2}\Delta \mathcal{T}-\e^{-4}f''(\phi_{\e})\mathcal{T}-\e^{-3}(f''(\phi_{\e})-f''(\phi_a))\mathscr{L}_\e[\phi]-3\e^{-3}\phi^2\mu_a,
\\ \mathcal{H}_2&=-\Delta \mathfrak{R}_2\e^{k-3}
+ f''(\phi_{\e})\mathfrak{R}_2\e^{k-5}-\mathfrak{R}_1\e^{k-4}.
\end{align}
\end{subequations}
Here $\mu= \mathscr{L}_\e[\phi]+\e^{-1}\mathcal{T}-\mathfrak{R}_2\e^{k-2},$ and $\mathcal{T}$ is defined by
\begin{equation}\label{defT}
  \mathcal{T}\triangleq f'(\phi_{\e})-f'(\phi_a)-f''(\phi_a)\phi=3\phi^2\phi_a+\phi^3.
\end{equation}
It follows from \eqref{omegadef} and   the inequality $(a+b)^2\geq\frac{1}{2}a^2-b^2$  that
\begin{align*}
\e^4\(\frac{1}{\varepsilon^2}\int_{\Omega}|\mathscr{L}_\e[\phi]|^2dx+\frac{1}{\varepsilon^3}\int_{\Omega} f'''(\phi_a)\mu_a\phi^2dx\)\geq \e^4\(\frac{1}{2}\int_{\Omega}(\Delta\phi)^2dx-\frac{C}{\varepsilon^4}\int_{\Omega}\phi^2dx\).
\end{align*}
Adding   the above  inequality to \eqref{spectral condition} and then  choosing a small $\e$ leads to
{\begin{equation}
\frac{1}{\varepsilon^2}\int_\Omega|\mathscr{L}_\e[\phi]|^2dx+
\frac{1}{\varepsilon^3}\int_\Omega f'''(\phi_a)\mu_a\phi^2dx\geq \frac{\e^4}{4}\int_{\Omega}(\Delta\phi)^2dx -C\int_{\Omega}\phi^2dx.
\end{equation}
This together with  \eqref{error energy-1}  and  \eqref{def K(t)}  leads to
\begin{equation}\label{spectrum-modi}
\begin{split}
\frac 12\frac d{dt}\|\phi\|^2_{L^2(\Omega)}+ \frac{\e^4}{8}\int_{\Omega}(\Delta\phi)^2dx+\frac{\tilde{C}}2\,\mathcal{K}(t)
-C\|\phi\|^2_{L^2(\Omega)}\leq  \int_{\Omega}\mathcal{H}_1\phi+\int_{\Omega}\mathcal{H}_2\phi.
\end{split}
\end{equation}}
It remains to estimate the integrals on the right hand side. Using \eqref{omegadef} and formula 
\begin{equation}\label{diff f''}
f''(\phi_{\e})-f''(\phi_a)=3\phi^2+6\phi\phi_a,
\end{equation}
 we can rewrite $\mathcal{H}_1$ by
\begin{align}
\mathcal{H}_1
%&=-\e^{-3}\mathscr{L}_\e[\mathcal{T}]-\e^{-3}(f''(\phi_{\e})-f''(\phi_a))(\mathscr{L}_\e[\phi]+\e^{-1}\mathcal{T}) -3\e^{-3}\phi^2\mu_a\nonumber
%\\&=-\e^{-3}\mathscr{L}_\e[\mathcal{T}]-\e^{-3}(3\phi^2+6\phi\phi_a)(\mathscr{L}_\e[\phi]+\e^{-1}\mathcal{T}) -3\e^{-3}\phi^2\mu_a.
 =-\e^{-3}\mathscr{L}_\e[\mathcal{T}]-\e^{-3}(3\phi^2+6\phi\phi_a)\mathscr{L}_\e[\phi]-\e^{-4}(3\phi^2+6\phi\phi_a)\mathcal{T}-3\e^{-3}\phi^2\mu_a.
\end{align}
Using \eqref{defT}, we can write 
\begin{equation}\label{H1H2sum}
\begin{split}
\int_{\Omega} \mathcal{H}_1\phi\, dx&=-\e^{-3}\int_{\Omega} (4\phi^3+9\phi^2\phi_a) \mathscr{L}_\e[\phi]\\
&-\e^{-4}\int_{\Omega}(3\phi^3+6\phi^2\phi_a)(3\phi^2\phi_a+\phi^3)-3\e^{-3}\int_{\Omega}\phi^3\mu_a.
\end{split}
\end{equation}
Regarding  the first  term on the right of \eqref{H1H2sum},
 it follows from \eqref{decompose phi}, an integration by parts  and $\phi\mid_{\p\Omega}=0$  that
\begin{equation}
\begin{split}
-\e^{-3}&\int_{\Omega}(4\phi^3+9\phi^2\phi_a) \mathscr{L}_\e[\phi]=-12\e^{-2}\int_{\Omega} |\nabla \phi|^2\phi^2-4\e^{-4}\int_{\Omega} f''(\phi_a) \phi^4
\\&\underbrace{-9\e^{-3}\int_{\Omega}\phi^2\phi_a\mathscr{L}_\e[\phi^\top]}_{\triangleq \mathcal{N}_1}\underbrace{-9\e^{-3}\int_{\Omega}\phi^2\phi_a \mathscr{L}_\e[\phi^\bot]}_{\triangleq \mathcal{N}_2}\underbrace{-9\e^{-3}\int_{\Omega}\phi^2\phi_a \mathscr{L}_\e[\phi_e]}_{\triangleq \mathcal{N}_*}
.\label{H1H2sum1}
%\\&=-\e^{-3}\int_{\Omega}3\phi^2\phi_a \varepsilon^{-\frac{1}{2}}\p_r (\ln \sqrt{g}) Z(s)\theta''(z)\zeta(r)+\cdots\\
%&\lesssim\e^{-\frac 72}\int_{\G_t}Z^3+\red{\e^{-\frac 72}\int_{\Omega}Z(s)(\phi^\bot)^2\theta''(z)\zeta(r)}+\cdots
\end{split}
\end{equation}
 
%\begin{align*}
% \mathcal{N}_1&=-9\e^{-\frac 72}\int_{\Omega}\phi^2\phi_a\p_r (\ln \sqrt{g}) Z(s)\theta''(z)\zeta(r)-9\e^{-\frac 52}\int_{\Omega}\phi^2\phi_a\Delta_{\Gamma}Z(s) \theta'(z)\zeta(r)+\cdots
% \\&\thickapprox-9\e^{-\frac 92}\int_{\Omega}Z^2(s)(\theta'(z))^2\zeta^2(r)\cdot\theta(z)\cdot\p_r (\ln \sqrt{g}) Z(s)\theta''(z)\zeta(r)+
% \\&\thickapprox-9\e^{-\frac 92}\int_{\Omega}Z^3(s)h(s)(\theta'(z))^2\theta(z)\theta''(z)\zeta^3(r)+.
% \end{align*}
% Noting that
% \begin{align*}
%-9\e^{-\frac 92}\int_{-\infty}^\infty(\theta'(z))^2\theta(z)\theta''(z)=9\e^{-\frac 92}\int_{-\infty}^\infty(\theta'(z))^4+18\e^{-\frac 92}\int_{-\infty}^\infty(\theta'(z))^2\theta(z)\theta''(z)
% \end{align*}
% which implies
% \begin{align*}
% \e^{-\frac 92}\int_{-\infty}^\infty(\theta'(z))^2\theta(z)\theta''(z)=-\frac{1}{3}\e^{-\frac 92}\int_{-\infty}^\infty(\theta'(z))^4.
% \end{align*}
%
%Moreover,one has
%\begin{align*}
%-3\e^{-3}\int_{\Omega}\phi^3\mu_a=-3\e^{-\frac 92}\int_{\Omega}Z^3(s)(\theta'(z))^3\zeta^3(r)\cdot(-\Delta d^{(0)})\theta'(z)=3\e^{-\frac 92}\int_{\Omega}Z^3(s)h(s)(\theta'(z))^4\zeta^3(r)+\cdots.
%\end{align*}
%\begin{equation}\label{H1H2suma}
%\begin{split}
%\int_{\Omega}  \mathcal{H}_1 \phi\, dx &= \mathcal{N}_1+\mathcal{N}_2-12\e^{-2}\int_{\Omega} |\nabla \phi|^2\phi^2-4\e^{-4}\int_{\Omega} f''(\phi_a) \phi^4\\
%&-\e^{-4}\int_{\Omega}(3\phi^3+6\phi^2\phi_a)(3\phi^2\phi_a+\phi^3)-3\e^{-3}\int_{\Omega}\phi^3\mu_a.
%\end{split}
%\end{equation}
Observe that the second term on the right of \eqref{H1H2sum}  involves polynomials of $\phi$, and the leading order one is $-3\e^{-4}\int_{\Omega}\phi^6$. Moreover $\phi_a$ and $\mu_a$ are uniformly bounded according to Proposition \ref{construction of app solution}. So we substitute \eqref{H1H2sum1} into \eqref{H1H2sum}, move the good terms to the left and then apply Cauchy-Schwarz inequality to absorb the term involving  $\phi^5$. This yields 
\begin{equation}\label{H1H2sumb}
\begin{split}
\int_{\Omega} \mathcal{H}_1\phi  +12\e^{-2}\int_{\Omega} |\nabla \phi|^2\phi^2+2\e^{-4}\int_\Omega\phi^6 \leq  \mathcal{N}_1+\mathcal{N}_2+\mathcal{N}_*+C\underbrace{\int_{\Omega} \(\e^{-4} \phi^4+\e^{-3}|\phi|^3 \)}_{\triangleq \mathcal{N}_3}.
\end{split}
\end{equation}
In view of \eqref{decompose difference} and \eqref{decompose phi}, the structures of $\mathcal{N}_1$ and $\mathcal{N}_*$ are the same except that the later  is of scale  $e^{-C/\e}$. So we shall omit its estimate for the sake of simplicity, and focus on the estimates of $\mathcal{N}_1,\mathcal{N}_2$ and $\mathcal{N}_3$.

In the sequel, we shall frequently employ the following Sobolev embedding theorems:
\begin{subequations}
\begin{align}
H^{1/3}(\G_t)\hookrightarrow L^3(\G_t),\quad H^{1/2}(\G_t)\hookrightarrow L^4(\G_t),~\text{where}~\G_t~\text{is a 2D surface}.\label{sobolev1}\\
H^{1/2}(\Omega)\hookrightarrow L^3(\Omega),\quad H^{3/4}(\Omega)\hookrightarrow L^4(\Omega),~\text{where}~\Omega~\text{is a 3D domain}. \label{sobolev2}
\end{align}
\end{subequations}

 %To simplify the presentation, we shall NOT estimate the terms related to $A(r,s)$ in \eqref{expansion phi top1} because of \eqref{zeta term}. Similarly, we shall omit the treatment of  $\phi_e$ in \eqref{decompose phi} because of \eqref{decompose difference}.
\begin{lemma}
There exists $\e_0>0$ such that for every $\e\in (0,\e_0)$ and $\nu \in (\e^{\frac 1{100}},1)$,
\begin{equation}\label{estimateN1c}
\begin{split}
{\mathcal{N}_1  \lesssim C(\nu)\(\e^{-14/3}\|\phi\|_{L^2(\Omega)}^{ 10/3}+\varepsilon^{-8}\|\phi\|_{L^2(\Omega)}^6+ \varepsilon^{-16}\|\phi\|_{L^2(\Omega)}^{10}\)+\nu\mathcal{K}(t).}
\end{split}
\end{equation}
\end{lemma}

\begin{proof} 
 We first  recall \eqref{expansion phi top} and \eqref{good term}, which yields
 \begin{equation}\label{expansion phi top1}
\begin{split}
\mathscr{L}_\e[\phi^\top]=&\varepsilon^{-\frac{1}{2}}\p_r (\ln \sqrt{g}) Z(s)\theta''(z)\zeta(r)-\e^{\frac 12}Zf'''(\theta)\widetilde{\phi}^{(2)}\theta'\zeta(r)\\
&+\varepsilon^{\frac{1}{2}}\Delta_{\Gamma}Z(s) \theta'(z)\zeta(r)+\varepsilon Z(s)A(r,s)+O(\e^{\frac 32})Z\theta'\zeta(r).
\end{split}
 \end{equation}
In view of \eqref{def K(t)}, among the five terms on the right of \eqref{expansion phi top1}, the first and the third are of leading order. The second and fifth terms have   similar structures as the first but are of order  $\e^1$ better. The term involving  $A(r,s)$ is much more regular because of \eqref{zeta term}. To simplify the presentation, we shall focus on the leading order terms:
 \begin{equation}\label{simplifyN1}
 \mathcal{N}_1=-9\e^{-\frac 72}\int_{\Omega}\phi^2\phi_a\p_r (\ln \sqrt{g}) Z(s)\theta''(z)\zeta(r)-9\e^{-\frac 52}\int_{\Omega}\phi^2\phi_a\Delta_{\Gamma}Z(s) \theta'(z)\zeta(r)+\cdots.
 \end{equation} 
Using \eqref{def K(t)}, we apply Sobolev embedding \eqref{sobolev1},  the interpolation inequality and   Young's inequality successively to obtain
\begin{equation*}
\begin{split}
\e^{-\frac 72}\|Z\|^3_{L^3(\G_t)}\lesssim \e^{-\frac 72}\|Z\|^3_{H^{\frac 13}(\G_t)}\lesssim \e^{-\frac 72}\|Z\|^{\frac 52}_{L^2(\G_t)}\|Z\|^{\frac 12}_{H^2(\G_t)}\lesssim  C(\nu) \e^{-14/3}\|Z\|^{\frac {10}3}_{L^2(\G_t)}+\nu \mathcal{K}(t).
\end{split}
\end{equation*}
{Observing that $z=r/\e$,  we apply  H\"{o}lder's inequality, a change of variable and Sobolev embedding \eqref{sobolev2} and yield \begin{equation*}
\begin{split}
\e^{-\frac 72}\left|\int_{\Omega}Z(s)(\phi^\bot)^2\theta''(z)\zeta(r)\right|&\lesssim\e^{-3}\|Z\|_{L^2(\G_t)}\|\phi^\bot\|^2_{L^4(\Omega)}
\\&\lesssim\e^{-3/2}\|Z\|_{L^2(\G_t)}\|\phi^\bot\|_{L^2(\Omega)}^{\frac{1}{2}}\, \|\phi^\bot\|_{H^1(\Omega)}^{\frac{3}{2}}\e^{-3/2}
\\&\lesssim C(\nu)\e^{-6}\|Z\|_{L^2(\G_t)}^4\|\phi^\bot\|_{L^2(\Omega)}^2+\nu\varepsilon^{-2}\|\phi^\bot\|_{H^1(\Omega)}^2
\\&\lesssim C(\nu)\e^{-6}\|\phi\|_{L^2(\Omega)}^{6}+\nu\mathcal{K}(t).
\end{split}
\end{equation*}
Note that in the last step we employed \eqref{decompose L2norm} and  \eqref{def K(t)}.
With the above two estimates we can treat the first term on the right hand side of \eqref{simplifyN1}. Combining \eqref{decompose phi}\footnote{Here the contributions due to $\phi_e$ is omitted as they  are of much lower order.}  with the above two inequalities lead to
\begin{equation}\label{estimateN1a}
\begin{split}
\varepsilon^{-\frac{7}{2}}\left|\int_{\Omega}\phi^2\phi_a \p_r (\ln \sqrt{g}) Z(s)\theta''(z)\zeta(r)\right|&\lesssim\e^{-\frac 72}\int_{\G_t}|Z|^3+\e^{-\frac 72}\int_{\Omega}\left|Z(s)(\phi^\bot)^2\theta''(z)\zeta(r)\right|
%\\\lesssim \nu^{-1}\e^{-\frac {14}3}&\|Z\|^{\frac {10}3}_{L^2(\G_t)}+\nu^{-1}\e^{-6}\|Z\|_{L^2(\G_t)}^4\|\phi^\bot\|_{L^2(\Omega)}^2+\nu \mathcal{K}(t)
\\\lesssim C(\nu)\e^{-\frac {14}3}&\|\phi\|_{L^2(\Omega)}^{\frac {10}3}+C(\nu)\e^{-6}\|\phi\|_{L^2(\Omega)}^{6}+\nu \mathcal{K}(t).
\end{split}
\end{equation}}
For  the second term on the right hand side of \eqref{simplifyN1}, we employ \eqref{sobolev1} and estimate 
\begin{equation*}
\begin{split}
\e^{-\frac 52}\left|\int_{\G_t}Z^2 \Delta_{\Gamma}Z(s)\theta'(z)\zeta(r)\right|&\lesssim\varepsilon^{-2}\|Z\|_{L^4(\G_t)}^2\| Z\|_{H^2(\G_t)}
%\\&\lesssim\varepsilon^{-2}\bigg(\|Z\|_{L^2(\G_t)}^{\frac{3}{4}}\|Z\|_{H^2(\G_t)}^{\frac{1}{4}}\bigg)^2\|Z\|_{H^2(\G_t)}
\\&\lesssim\varepsilon^{-2}\|Z\|_{L^2(\G_t)}^{\frac{3}{2}}\|Z\|_{H^2(\G_t)}^{\frac{3}{2}}
\lesssim C(\nu) \varepsilon^{-8}\|Z\|_{L^2(\G_t)}^6+\nu\mathcal{K}(t).
\end{split}
\end{equation*}
Moreover, we need to estimate the following integral using \eqref{sobolev2}, the interpolation inequality and Young's inequality:
\begin{equation*}
\begin{split}
\e^{-\frac 52}\left|\int_{\Omega}\Delta_{\Gamma}Z(s)(\phi^\bot)^2\theta''(z)\zeta(r)\right|&\lesssim\varepsilon^{-2}\|Z\|_{H^2(\G_t)}\|\phi^\bot\|^2_{L^4(\Omega)}
\\ \lesssim  \varepsilon^{-2} & \|Z\|_{H^2(\G_t)}\|\phi^\bot\|_{H^2(\Omega)}^{\frac{3}{4}}\|\phi^\bot\|_{L^2(\Omega)}^{\frac{5}{4}}
\\ \lesssim   \nu & \((\|Z\|_{H^2(\G_t)}\|\phi^\bot\|_{H^2(\Omega)}^{\frac{3}{4}}\)^{8/7}+C(\nu)\varepsilon^{-16}\|\phi^\bot\|_{L^2(\Omega)}^{10}
\\ \lesssim   \nu & \mathcal{K}(t)+C(\nu)\varepsilon^{-16}\|\phi^\bot\|_{L^2(\Omega)}^{10},
%\\&\lesssim\varepsilon^{-2}\|\Delta_{\Gamma} Z\|_{L^2(\G_t)}\|\phi^\bot\|_{L^2(\Omega)}^{\frac{1}{2}}\|\nabla\phi^\bot\|_{L^2(\Omega)}^{\frac{3}{2}}
\end{split}
\end{equation*}
Combining the above two results gives the estimate  of the second term on the right hand side of \eqref{simplifyN1}:
\begin{equation}\label{estimateN1b}
\begin{split}
&\varepsilon^{-\frac{5}{2}}\left|\int_{\Omega}9\phi^2\phi_a \Delta_{\Gamma}Z(s) \theta'(z)\zeta(r)\right|\\
&\lesssim\e^{-\frac 52}\left|\int_{\G_t}Z^2 \Delta_{\Gamma}Z(s)\theta'(z)\zeta(r)\right|+\e^{-\frac 52}\left|\int_{\Omega}\Delta_{\Gamma}Z(s)(\phi^\bot)^2\theta'(z)\zeta(r)\right|
%\\&\lesssim\varepsilon^{-2}\|Z\|_{L^4(\G_t)}^2\|Z\|_{H^2(\G_t)}+\varepsilon^{-2}\|Z\|_{H^2(\G_t)}\|\phi^\bot\|^2_{L^4(\Omega)}
%\\&\lesssim\frac{1}{\nu\varepsilon^{8}}\|Z\|_{L^2(\G_t)}^6+\frac{1}{\nu\varepsilon^{16}}\|\phi^\bot\|_{L^2(\Omega)}^{10}+\nu\mathcal{K}(t)
\\&\lesssim C(\nu)\(\varepsilon^{-8}\|\phi\|_{L^2(\Omega)}^6+ \varepsilon^{-16}\|\phi\|_{L^2(\Omega)}^{10}\)+\nu\mathcal{K}(t).
%\\&\lesssim\varepsilon^{-2}\|Z\|_{L^4(\G_t)}^2\|\Delta_{\Gamma} Z\|_{L^2(\G_t)}+\varepsilon^{-2}\|\Delta_{\Gamma} Z\|_{L^2(\G_t)}\|\phi^\bot\|_{L^2(\Omega)}^{\frac{1}{2}}\|\nabla\phi^\bot\|_{L^2(\Omega)}^{\frac{3}{2}}
\end{split}
\end{equation}
Finally substituting \eqref{estimateN1a} and \eqref{estimateN1b} into \eqref{simplifyN1} yields \eqref{estimateN1c}.
\end{proof}

\begin{lemma}
There exists $\e_0>0$ such that for every $\e\in (0,\e_0)$ and $\nu\in (\e^{\frac 1{100}},1)$ there holds
\begin{equation}\label{estimateN2}
\begin{split}
\mathcal{N}_2+\mathcal{N}_3  &\lesssim C(\nu) \big( \e^{-\frac{14}3}\|\phi\|_{L^2(\Omega)}^{\frac {10}3}+\e^{-\frac{24}5}\|\phi\|_{L^2(\Omega)}^{\frac{18}{5}}+ \varepsilon^{-10}\|\phi\|_{L^2(\Omega)}^6+ \varepsilon^{-16}\|\phi\|_{L^2(\Omega)}^{10}\big)+\nu\mathcal{K}(t).
\end{split}
\end{equation}
\end{lemma}
\begin{proof}
Applying the Cauchy-Schwartz inequality and \eqref{def K(t)} yields
\begin{equation*}
\begin{split}
\mathcal{N}_2\lesssim \e^{-2}\int_{\Omega} \phi^2|\Delta \phi^\bot|+\e^{-4}\int_{\Omega} \phi^2\left|\phi^\bot\right|
\lesssim   C(\nu)\e^{-4}\int_{\Omega}|\phi|^4+\nu\mathcal{K}(t).
\end{split}
\end{equation*}
To estimate $\mathcal{N}_3$, 
owing to \eqref{decompose phi} and \eqref{decompose L2norm}, we   apply \eqref{sobolev1} and \eqref{sobolev2} and obtain 
\begin{equation}
\begin{split}
\e^{-4}\int_{\Omega}|\phi|^4&\lesssim\e^{-5}\int_{\Gamma_t}Z^4+\e^{-4}\int_{\Omega}|\phi^\perp|^4
\\&\lesssim\e^{-5}\|Z\|_{L^2(\G_t)}^{3}\|Z\|_{H^2(\G_t)}+\e^{-4}\|\phi^\bot\|_{L^2(\Omega)}^{\frac{5}{2}}\|\phi^\bot\|_{H^2(\Omega)}^{\frac{3}{2}}
\\&\lesssim C(\nu)\(\e^{-10}\|Z\|_{L^2(\G_t)}^{6}+ \e^{-16}\|\phi^\perp\|_{L^2(\Omega)}^{10}\)+\nu \mathcal{K}(t)
\\&\lesssim C(\nu)\(\varepsilon^{-10}\|\phi\|_{L^2(\Omega)}^6+ \varepsilon^{-16}\|\phi\|_{L^2(\Omega)}^{10}\)+\nu\mathcal{K}(t).\label{H1H2sumL4}
\end{split}
\end{equation} 
In a similar way, we have
\begin{equation}\label{H1H2sumL5}
\begin{split}
\e^{-3}\int_{\Omega}|\phi|^3&\lesssim\e^{-\frac{7}{2}}\int_{\Gamma_t}|Z|^3+\e^{-3}\int_{\Omega}|\phi^\perp|^3
\\&\lesssim\e^{-\frac{7}{2}}\|Z\|^{\frac 52}_{L^2(\G_t)}\|Z\|^{\frac 12}_{H^2(\G_t)}+\e^{-3}\|\phi^\bot\|_{L^2(\Omega)}^{\frac{9}{4}}\|\phi^\bot\|_{H^2(\Omega)}^{\frac{3}{4}}
\\&\lesssim C(\nu)\(\e^{-\frac{14}{3}}\|Z\|_{L^2(\G_t)}^{\frac{10}{3}}+ \e^{-\frac{24}{5}}\|\phi^\perp\|_{L^2(\Omega)}^{\frac{18}{5}}\)+\nu \mathcal{K}(t).
%\\&\lesssim C(\nu)\(\e^{-14/3}\|\phi\|_{L^2(\Omega)}^{\frac{10}{3}}+ \e^{-24/5}\|\phi\|_{L^2(\Omega)}^{\frac{18}{5}}\)+\nu \mathcal{K}(t).
\end{split}
\end{equation}
So the desired estimate follows from the above inequalities.
\end{proof}
%It is noted that we only need to estimate
%
%
%
%\begin{align}
%\e^{-\frac 72}\int_{\Omega}Z(s)(\phi^\bot)^2\theta''(z)\zeta(r)&\leq\frac{1}{\nu}\int_{\G_t}Z^2+C\nu\varepsilon^{-6}\|\phi^\bot\|_{L^4}^4
%\nonumber\\&\leq\frac{1}{\nu}\int_{\G_t}Z^2
%+C\nu\varepsilon^{-6}\|\phi^\bot\|_{L^2}\|\nabla\phi^\bot\|_{L^2}^3
%\nonumber\\&\leq\frac{1}{\nu}\int_{\G_t}Z^2
%+C\nu\varepsilon^{-6}\cdot\varepsilon^2\mathcal{K}^{\frac{1}{2}}(t)\cdot\varepsilon^3\mathcal{K}^{\frac{3}{2}}(t)
%\end{align}
\begin{proof}[{\bf Proof of Theorem \ref{main theorem}}]
{The construction of the approximate solution fulfilling \eqref{app model} is already given in 
Proposition \ref{construction of app solution}. We focus on  the estimate of the difference \eqref{time error}.}
Submitting \eqref{estimateN1c}, \eqref{estimateN2}  into \eqref{H1H2sumb} leads to 
\begin{equation}\label{estimateH1phi}
\begin{split}
\int_{\Omega}\mathcal{H}_1\phi\lesssim C(\nu)\(\frac 1{\e^{\frac {14}3}}\|\phi\|_{L^2(\Omega)}^{\frac {10}3}+\frac{1}{\e^{\frac{24}{5}}}\|\phi\|_{L^2(\Omega)}^{\frac{18}{5}}+\frac{1}{\varepsilon^{10}}\|\phi\|_{L^2(\Omega)}^6+\frac{1}{\varepsilon^{16}}\|\phi\|_{L^2(\Omega)}^{10}\)+\nu\mathcal{K}(t),
\end{split}
\end{equation} 
for any $\e\in (0,\e_0)$ and $\nu\in (\e^{\frac 1{100}},1)$.
 Now we treat $\mathcal{H}_2$ in \eqref{nonlinearterms}. It follows from Theorem \ref{spectral theorem} that $\mathfrak{R}_1,\mathfrak{R}_2$ are uniformly bounded in $\e$ and $(x,t)$.
This together with \eqref{diff f''}, the Cauchy-Schwarz inequality and \eqref{H1H2sumL5}  leads to
\begin{align}
\int_{\Omega}\mathcal{H}_2\phi
&\lesssim \e^{k-5}\int_\Omega \(|\Delta \phi|+|\phi|\)+\e^{k-5}\int_\Omega |\phi|^3\nonumber\\
&\lesssim \e^{k-5}+ \e^{k-5}\int_\Omega \(\phi^2+|\Delta \phi|^2\)+\e^{k-2}\e^{-3}\int_\Omega |\phi|^3.
\end{align}
We can  use \eqref{H1H2sumL5} to treat the last term and add it  to \eqref{estimateH1phi}. This brings 
\begin{equation}
\begin{split}
\int_{\Omega}(\mathcal{H}_1+\mathcal{H}_2)\phi & \lesssim C(\nu)\(\frac 1{\e^{\frac{14}3}}\|\phi\|_{L^2(\Omega)}^{\frac {10}3}+\frac{1}{\e^{\frac{24}{5}}}\|\phi\|_{L^2(\Omega)}^{\frac{18}{5}}+\frac{1}{\varepsilon^{10}}\|\phi\|_{L^2(\Omega)}^6+\frac{1}{\varepsilon^{16}}\|\phi\|_{L^2(\Omega)}^{10}\)\\
&+\nu\mathcal{K}(t)+{{\e^{k-5}}}+\e^{k-5}\int_\Omega \(\phi^2+|\Delta \phi|^2\).
\end{split}
\end{equation}
We can  substitute the above estimate into \eqref{spectrum-modi} and choose $\nu$ sufficiently small (but independent of $\e$) and $k\geq 10$  to obtain    
\begin{equation}\label{gronwall}
\begin{split}
&\frac d{dt}\|\phi\|^2_{L^2(\Omega)}
-C\|\phi\|^2_{L^2(\Omega)}\\
\leq& C_6\(\frac 1{\e^{14/3}}\|\phi\|_{L^2(\Omega)}^{\frac {10}3}+\frac{1}{\e^{\frac{24}{5}}}\|\phi\|_{L^2(\Omega)}^{\frac{18}{5}}+\frac{1}{\varepsilon^{10}}\|\phi\|_{L^2(\Omega)}^6+\frac{1}{\varepsilon^{16}}\|\phi\|_{L^2(\Omega)}^{10}+{{\e^{k-5}}}\),~\forall \e\in (0,\e_0),
\end{split}
\end{equation}
where
$C,C_6$   are  positive constants   independent of $\e$.

We shall use the Gr\"{o}nwall inequality and the continuity method to close the energy estimate under the  assumption \eqref{initial error}. We set  
\begin{equation}\label{continuity}
T_\e=\sup \{\tau \in (0,T_{max}]\mid \|\phi\|_{C([0,\tau),L^2(\Omega))}\leq  \Lambda\e^{7/2}\}
\end{equation}
with $\Lambda>0$ and $T_{max}\in (0,T]$ being determined later on.
 It  can be verified from \eqref{gronwall} that, on the time interval $(0,T_\e)$ there holds
 \begin{equation}
 \frac d{dt}\|\phi\|^2_{L^2(\Omega)} 
-C\|\phi\|^2_{L^2(\Omega)}
\leq C_6\(\Lambda^{4/3}+\e^{4/5}\Lambda^{8/5}+\e^4\Lambda^4+\e^{12}\Lambda^8\)\|\phi\|^2_{L^2(\Omega)}+C_6 \e^5.
 \end{equation}
{We shall work with   $\e\in (0,\e_1)$ where 
\begin{equation}\label{eps1 choice}
\e_1=\min (\e_0,1,\Lambda^{-2},\Lambda^{-1}, T^{-1}C_6^{-1}, e^{-T(C+19 C_6)}),
\end{equation}
 then  $\frac d{dt}\|\phi\|^2_{L^2(\Omega)}
\leq \(C+C_6\Lambda^{4/3}+3C_6\) \|\phi\|^2_{L^2(\Omega)}+C_6 \e^5.$
Applying the Gr\"{o}nwall inequality and \eqref{initial error} yields
 \begin{equation}\label{continuity1}
\begin{split}
\|\phi\|_{C([0,T_{\e}];L^2(\Omega))}\leq &\exp ( T_\e  (C+C_6\Lambda^{4/3}+3C_6))\(\|\phi\mid_{t=0}\|_{L^2(\Omega)}+T_\e C_6 \e^5\)\\\leq & \exp ( T_{max} (C+C_6\Lambda^{4/3}+3C_6))(C_{in}\e^{7/2}+T_{max} C_6 \e^5).
\end{split}
 \end{equation}
\noindent $\bullet$ For an   arbitrary $C_{in}>0$,
  we shall choose $\Lambda=10 C_{in}$ in \eqref{continuity} and 
\begin{equation}
T_{max}=\min\left\{ C_{in}/C_6, (C+C_6\Lambda^{4/3}+3C_6)^{-1},T\right\}.
\end{equation}
Then it follows from \eqref{continuity1} and \eqref{continuity} that $T_\e\geq T_{\max}$, and thus \eqref{time error} holds. 

\noindent $\bullet$
If $C_{in}\leq \exp \(-T(C+19 C_6)\)$,  we shall choose  $\Lambda=8$ and $T=T_{max}$. Then it follows from \eqref{eps1 choice} and \eqref{continuity1} that 
 $\|\phi\|_{C([0,T_{\e}];L^2(\Omega))}\leq 2\e^{7/2}$. This combined with \eqref{continuity} leads to   
 \eqref{time error}.}
 \end{proof}
%here we have used
%\begin{equation*}
%\begin{split}
%\e^{-\frac 52}\|Z\|^3_{L^3(\G_t)}\leq \e^{-\frac 52}\|Z\|^3_{H^{\frac 13}(\G_t)}\leq \e^{-\frac 52}\|Z\|^{\frac 52}_{L^2(\G_t)}\|Z\|^{\frac 12}_{H^2(\G_t)}\leq \frac 1{\nu\e^{\frac {10}3} }\|Z\|^{\frac {10}3}_{L^2(\G_t)}+\nu \mathcal{K}(t)
%\end{split}
%\end{equation*}
%Since $\G_t$ is two dimensional, there holds Sobolev embedding $L^3(\G_t)\hookrightarrow H^{1/3}(\G_t)$.
%As a result, the leading order term above can be estimated by
%\begin{equation}
%\begin{split}
%\left|\e^{-3}\int_{\Omega} \mathscr{L}_\e[\mathcal{T}]\phi\right|\lesssim\e^{-3-\frac 12}\|Z\|^3_{L^3(\G_t)}\leq  &\e^{-\frac 72}\|Z\|^3_{H^{\frac 13}(\G_t)}\leq \e^{-\frac 72}\|Z\|^{\frac 52}_{L^2(\G_t)}\|Z\|^{\frac 12}_{H^2(\G_t)}+\cdots\\ \leq & \frac 1{\e^{\frac {14}3} \nu}\|Z\|^{\frac {10}3}_{L^2(\G_t)}+\nu \mathcal{K}(t)
%\end{split}
%\end{equation}
%In the last step, we employed Young's inequality $ab\leq   \frac{a^p}p+\frac { b^q} q$ with $p=4, q=4/3$.
%{\color{red}So in order to apply Gronwall,
%this requires $\e^{-\frac 72}\|Z\|_{L^2(\G_t)}$ to be bounded. This  amounts to requiring that $\|\phi_\e-\phi_a\|_{L^2(\Omega)}\approx \|Z\|_{L^2(\G_t)}\leq \e^{7/2}$.  This is much better than our previous result which requires $\|\phi_\e-\phi_a\|_{L^2(\Omega)}\leq \e^{24}$.}

\begin{rmk}\label{verification}We discuss the admissible initial data satisfying \eqref{initial error}. 
Because of \eqref{eq:expan phia} and \eqref{inn-01},  the approximate solution satisfies
 \begin{equation}
 \phi_a(x,t)=\(\theta(z)+\varepsilon^2\widetilde{\phi}^{(2)}
\(z,x,t\)+\varepsilon^3\widetilde{\phi}^{(3)}
\(z,x,t\)\)\Big|_{z=r(x,t)/\e}+O(\varepsilon^4)\quad\text{in}~\G^0(\delta).
 \end{equation}
Here $\widetilde{\phi}^{(2)}$  has an explicit form \eqref{expression of phi2}, and $\widetilde{\phi}^{(2)},\widetilde{\phi}^{(3)}$  both decay exponentially to $0$ as $z\to \pm\infty$. 
  On the other hand, due to \eqref{app distance}, \eqref{d} and the construction made in Appendix \ref{innerexpan}, we have $r(x,0)=d^{(0)}(x,0)$ (the signed-distance to the initial interface   $\G^0_0$). So 
   \begin{equation}\label{verify2}
 \phi_a(x,0)= \theta(\tfrac{d^{(0)}(x,0)}\e)+\varepsilon^2\widetilde{\phi}^{(2)}
\(\tfrac{d^{(0)}(x,0)}\e,x,0\)+\varepsilon^3\widetilde{\phi}^{(3)}
\(\tfrac{d^{(0)}(x,0)}\e,x,0\)+O(\varepsilon^4).
 \end{equation}
 If we choose $\phi_\e(x,0)=\theta(\frac{d^{(0)}(x,0)}\e)+\varepsilon^2\widetilde{\phi}^{(2)}
\(\frac{d^{(0)}(x,0)}\e,x,0\)$, then by a change of variable in $\G_0^0(\delta)$ (the $\delta$-tubular neighborhood of $\G_0^0$), we arrive at 
\begin{equation}
\|\phi_\e(x,0)-\phi_a(x,0)\|_{L^2(\G_0^0(\delta))}=O(\e^{7/2}).
\end{equation}
  This
  together with an appropriate choice of $\phi_\e(x,0)$ outside $\G_0^0(\delta)$ leads to \eqref{initial error}.
   \end{rmk}

\appendix
%%%%%%%%%%%%%%%%%%%%%%%%%%%%%%%%%%%%%%%%%%%%%%%%%%%%%%%%%%%%%%%%%%%%%%
%%%%%%%%%%%%%%%%%%%%%%%%%%%%%%%%%%%%%%%%%%%%%%%%%%%%%%%%%%%%%%%%%%%%%%
\section{Formally matched expansions}\label{innerexpan}
\indent
We use the matched asymptotic expansions
to obtain   \eqref{phi12} and \eqref{mu12}, and thus give a complete proof of Proposition \ref{construction of app solution}, and of  \eqref{app model}. The construction will  employ the solvability  of an ODE, see   \cite[Lemma 4.1]{Alikakos1994}.
\begin{lemma}\label{ODEsolver-special case}
Let $\ell,m, n\in \{0, 1, 2\}$ and  $\widetilde{A}(z,x,t):\mathbb{R}\times\G^0(\delta)\to \mathbb{R}$ 
be a function   satisfying 
 \begin{equation}\label{decay-0}
\p_z^\ell\p_x^m\p_t^n  \widetilde{A}(z,x,t)=O(e^{-C|z|})~\text{as}~z\to \pm\infty,~\text{uniformly in}~(x,t)\in \G^0(\delta),
\end{equation}
 and the following compatibility condition holds,
\begin{equation}\label{compatiblity-1}
  \int_{\mathbb{R}}\widetilde{A}(z,x,t)\theta'(z)dz=0,~(x,t)\in \G^0(\delta).
\end{equation}
 Then  the equation $\mathscr{L} \tilde{U}=\widetilde{A}$ has a bounded solution so  that
 \begin{equation}\label{decay-1}
\p_z^\ell\p_x^m\p_t^n  \widetilde{U}(z,x,t)=O(e^{-C|z|})~\text{as}~z\to \pm\infty,~\text{uniformly in}~(x,t)\in \G^0(\delta).
\end{equation}
 Moreover, there exists a smooth function $U(x,t)$ such that
\begin{equation}\label{odesolver}
{\widetilde{U}(z,x,t)=U(x,t)\theta'(z)+\left[\int_{0}^{z}(\theta'(\zeta))^{-2}\(\int_{\zeta}^{+\infty}\tilde{A}(\tau,x,t)\theta'(\tau)d\tau \) d\zeta\,\right]\theta'(z).}
\end{equation}
The  unique solution satisfying $\widetilde{U}(0,x,t)=0$ corresponds to $U\equiv 0$.
\end{lemma}
Recall from Section \ref{levelset} that $d^{(0)}$ is the signed-distance to $\G^0$.   We need the following lemma whose proof can be found   in  \cite{MR3383330}.
\begin{lemma}\label{equivalentdistance}
The interface $\G^0$ evolves under the Willmore flow \eqref{interface law-geometric1} if and only if  $d^{(0)}$ fulfills 
\begin{align}
\partial_td^{(0)}+\Delta^2 d^{(0)}=\Delta d^{(0)}{D^{(0)}}
+\nabla d^{(0)}\cdot\nabla{D^{(0)}}\ \ \ \text{on}\ \ \Gamma^0.
\label{distance law}
\end{align}
\end{lemma}

Following \cite{Alikakos1994}, we employ the stretched variable $z=\frac{d_\e}{\varepsilon}\in\mathbb{R}$ (see \eqref{d} for the definition of $d_\e$) and set the following Ansatz in $\Gamma^0(3\delta)$ for the inner expansion
\begin{align}
\widetilde{\phi}^\varepsilon(z,x,t)=\sum_{\ell\geq 0}\varepsilon^\ell \widetilde{\phi}^{(\ell)}(z,x,t),\qquad 
\widetilde{\mu}^\varepsilon(z,x,t)=\sum_{\ell\geq 0}\varepsilon^\ell \widetilde{\mu}^{(\ell)}(z,x,t),\label{inex-1}
\end{align}
which should fulfill the following   matching conditions in $(x,t)\in\Gamma^0(3\delta)$:
\begin{align}
&D_z^\gamma D_x^\alpha D_t^\beta \Big(\widetilde{\phi}^{(i)}(z,x,t)-\phi^{(i)}_{\pm}(x,t)\Big)=O(e^{-\nu|z|}),\label{mc-1}\\
&D_z^\gamma D_x^\alpha D_t^\beta \Big(\widetilde{\mu}^{(i)}(z,x,t) -\mu^{(i)}_{\pm}(x,t)\Big)=O(e^{-\nu|z|}).\label{mc-2}
\end{align}
Here $\nu$ is a positive fixed constant  and $0\leq\alpha,\beta,\gamma\leq 2$. It follows from the Taylor expansion and (\ref{inex-1}) that,
\begin{subequations}\label{f expansion}
\begin{align}
f'(\widetilde{\phi}^\varepsilon)&=f'(\widetilde{\phi}^{(0)})+ f''(\widetilde{\phi}^{(0)})\sum_{i\geq 1}\varepsilon^i\widetilde{\phi}^{(i)}+\sum_{i\geq 1} \e^{i} g_{i-1}\big(\widetilde{\phi}^{(0)},\cdots,
\widetilde{\phi}^{(i-1)}\big),\\
f''(\widetilde{\phi}^\varepsilon)&=f''(\widetilde{\phi}^{(0)})+ f'''(\widetilde{\phi}^{(0)})\sum_{i\geq 1}\e^i\widetilde{\phi}^{(i)}+ \sum_{i\geq 1} \e^i g_{i-1}^*\big(\widetilde{\phi}^{(0)},\cdots,
\widetilde{\phi}^{(i-1)}\big),
\end{align}
\end{subequations}
where $g_i,g_i^*$ enjoy the following property:
\begin{lemma}\label{polynomials}
For $i\geq 1$, $g_i(x_0,\cdots,x_i)$ and $g^*_i(x_0,\cdots,x_i)$ are  polynomials of $i+1$ variables. Moreover, they vanish when $x_1=\cdots =x_{i-1}=0$, and $g_0=g^*_0=0$.
\end{lemma}
Using $|\nabla d_\e|=1$ and the chain-rule, we have the following expansions
\begin{align*}
\p_t \widetilde{\phi}^\varepsilon(\tfrac{d_\e}{\varepsilon},x,t)&=\partial_t\widetilde{\phi}^\varepsilon+\varepsilon^{-1}\partial_z\widetilde{\phi}^\varepsilon\partial_td_\e,\\
\Delta\widetilde{\mu}^\varepsilon(\tfrac{d_\e}\e,x,t)&=\varepsilon^{-2}\p_z^2\widetilde{\mu}^\varepsilon+2\varepsilon^{-1}\nabla_x\partial_z\widetilde{\mu}^\varepsilon\cdot\nabla d_\e+\varepsilon^{-1}\partial_z\widetilde{\mu}^\varepsilon\Delta d_\e+\Delta_x\widetilde{\mu}^\varepsilon,\\
\Delta\widetilde{\phi}^\varepsilon(\tfrac{d_\e}\e,x,t)&=\varepsilon^{-2}\p_z^2\widetilde{\phi}^\varepsilon+2\varepsilon^{-1}\nabla_x\partial_z\widetilde{\phi}^\varepsilon\cdot\nabla d_\e+\varepsilon^{-1}\partial_z\widetilde{\phi}^\varepsilon\Delta d_\e+\Delta_x\widetilde{\phi}^\varepsilon.
\end{align*}
We expect   $(\widetilde{\phi}^\varepsilon(z,x,t),\widetilde{\mu}^\varepsilon(z,x,t))|_{z= d_\e/\varepsilon}$ to satisfy \eqref{model} up to a high order term in $\e$ and thus determine the terms in \eqref{inex-1}:
\begin{align}
\p_z^2\widetilde{\mu}^\varepsilon-f''(\widetilde{\phi}^\varepsilon)\widetilde{\mu}^\varepsilon+2\varepsilon\nabla_x\partial_z\widetilde{\mu}^\varepsilon\cdot\nabla d_\e+\varepsilon\partial_z\widetilde{\mu}^\varepsilon\Delta d_\e
-\varepsilon^2\partial_z\widetilde{\phi}^\varepsilon\partial_td_\e\nonumber\\+\varepsilon^2\Delta_x\widetilde{\mu}^\varepsilon-\varepsilon^3\partial_t\widetilde{\phi}^\varepsilon
= O(\e^{k}),\label{newequ-1}\\
-\p_z^2\widetilde{\phi}^\varepsilon+f'(\widetilde{\phi}^\varepsilon)-2\varepsilon\nabla_x\partial_z\widetilde{\phi}^\varepsilon\cdot\nabla d_\e-\varepsilon\partial_z\widetilde{\phi}^\varepsilon\Delta d_\e-\varepsilon\widetilde{\mu}^\varepsilon-\varepsilon^2\Delta_x\widetilde{\phi}^\varepsilon= O(\e^{k}).\label{newequ-2}
\end{align}
 Since $z=\frac{d_\e}{\varepsilon}$, we need the above two equations to hold merely on $$S^\e\triangleq\{(z,x,t)\in\mathbb{R}\times \G^0(3\delta):z= d_\e/\e\}.$$ So we can add  in \eqref{newequ-1} terms which are multiplied by  $d_\e-\varepsilon z$. These terms will give more degrees of  freedom to construct and to  solve the  equations for $d^{(\ell)}$. See Remark \ref{deterd_0-re}  below  and \cite{Alikakos1994}.
So we modify \eqref{newequ-1}  as follows
\begin{align}
\p_z^2\widetilde{\mu}^\varepsilon-f''(\widetilde{\phi}^\varepsilon)\widetilde{\mu}^\varepsilon&+2\varepsilon\nabla_x\partial_z\widetilde{\mu}^\varepsilon\cdot\nabla d_\e+\varepsilon\partial_z\widetilde{\mu}^\varepsilon\Delta d_\e-\varepsilon^2\partial_z\widetilde{\phi}^\varepsilon\partial_td_\e\nonumber\\&+\varepsilon^2\Delta_x\widetilde{\mu}^\varepsilon-\varepsilon^3\partial_t\widetilde{\phi}^\varepsilon
+\varepsilon^2\chi^\varepsilon(d_\e-\varepsilon z)\eta'= O(\e^{k}),\label{newequ-1-m}
\end{align}
where $\eta(z)$ is a smooth non-decreasing function satisfying
\begin{align}
\eta(z)=0\ \text{if}\ z\leq -1;\ \eta(z)=1\ \text{if}\ z\geq 1;\ \ \eta'(z)\ \text{is\ even},
\end{align}
and $\chi^{\varepsilon}(x, t)=\sum_{i=0}^{\infty}\varepsilon^{i}\chi^{(i)}(x, t)$ with $\chi^{(i)}$ being  determined later on.
\begin{definition}\label{convention}
We shall  use {\bf   $\e^\ell$-scale} to denote  the terms of form $\e^\ell g(z,x,t)$. The  {\bf{$\ell$-order}} will  refer to  those    indexed by  $\ell$ if  $\ell\geq 0$, and by $0$ if $\ell <0$.
Moreover, $\tilde{\Psi}^{(\ell)}(z,x,t)$ and $\Xi^{(\ell)}(x,t)$ will denote   {\bf generic terms which might change from line to line}, and will  depend on terms of order at most $\ell$. 
\end{definition}

 %%%%%%%%%%%%%%%%%%%%%%%%%%%%%%%%%%%%%%%%%%%%%%%%%%%%%%%%%%%%%%%%%%%%%%%%
\subsection{$\varepsilon^1$-scale}\label{subsection matching 01}
Collecting all terms of  $\varepsilon^0$-scale in (\ref{newequ-2})-(\ref{newequ-1-m}),  we have $\p_z^2\widetilde{\phi}^{(0)}=f'(\widetilde{\phi})$ and $\p_z^2\widetilde{\mu}^{(0)}=f''(\widetilde{\phi}^{(0)})\widetilde{\mu}^{(0)}$.
Together with  matching condition \eqref{mc-1} and \eqref{mc-2} yields 
\begin{equation}
\widetilde{\phi}^{(0)}=\theta(z),\quad \widetilde{\mu}^{(0)}(z,x,t)=\mu_0(x,t)\theta'(z),
\end{equation}
for some function  $\mu_0$ which will be determined later on.
To proceed we recall the operator $\mathscr{L}$ defined at \eqref{1doperator}, which enjoys
\begin{equation}\label{commutator2}
\mathscr{L}~\text{is self-adjoint},~
  \p_z\mathscr{L}-\mathscr{L}\p_z=f'''(\theta)\theta'\mathcal{I}.\end{equation}
   Collecting all terms of  $\varepsilon^1$-scale in  (\ref{newequ-2})-(\ref{newequ-1-m}), and using $\widetilde{\phi}^{(0)}=\theta(z)$, we have 
\begin{subequations}\label{inequ-1-1 new}
\begin{align}
\mathscr{L}\widetilde{\mu}^{(1)}&=-f'''(\theta)\widetilde{\phi}^{(1)}\widetilde{\mu}^{(0)}
+2\nabla_x\partial_z\widetilde{\mu}^{(0)}\cdot\nabla d^{(0)}+\partial_z\widetilde{\mu}^{(0)}\Delta d^{(0)},\label{inequ-1-1}\\
\mathscr{L}\widetilde{\phi}^{(1)}&=\widetilde{\mu}^{(0)}+\partial_z\theta\Delta d^{(0)}.\label{inequ-1-2}
\end{align}
\end{subequations}
Here $\mu_0$ is chosen such that \eqref{inequ-1-2} fulfills \eqref{compatiblity-1}, i.e. $\mu_0=-\Delta d^{(0)}$. Thus
 \begin{equation}\label{eq:muphi1}
 \widetilde{\mu}^{(0)}(z,x,t)=-\Delta d^{(0)}\theta',~\widetilde{\phi}^{(1)}(z,x,t)=0.
 \end{equation}
  This together with Lemma \ref{polynomials} implies
\begin{align}
g_{1}=g_{2}=g_{1}^*=g_{2}^*=0.\label{partial zero}
\end{align}
Substituting this  into (\ref{inequ-1-1}) yields 
${\mathscr{L}\widetilde{\mu}^{(1)}}=-{2D^{(0)}}\theta''$
where  $D^{(0)}$ is given by
\begin{equation}\label{def:D0D1app}
D^{(0)}(x,t)=\nabla\Delta d^{(0)}\cdot\nabla d^{(0)}+\tfrac{1}{2}(\Delta d^{(0)})^2.
\end{equation}
Using \eqref{odesolver}, for some $\mu_1(x,t)$ which will  be determined later on, there holds
\begin{align}\label{eq:1.2}
\widetilde{\mu}^{(1)}(z,x,t)&=D^{(0)}z\theta'(z)+\mu_1(x,t)\theta'(z).
\end{align}

%%%%%%%%%%%%%%%%%%%%%%%%%%%%%%%%%%%%%%%%%%%%%%%%%%%%%%%%%%%%%%%%%%%%%%%%
\subsection{$\varepsilon^2$-scale}\label{subsection matching 2}
Substituting (\ref{inex-1}) into   (\ref{newequ-1})-(\ref{newequ-2}),  and collecting all terms of  $\varepsilon^2$-scale, we obtain 
\begin{align*}
0&=\p_z^2\widetilde{\mu}^{(2)}-f''(\theta)\widetilde{\mu}^{(2)}-f'''(\theta)\widetilde{\phi}^{(1)}\widetilde{\mu}^{(1)}
-\big(f'''(\theta)\widetilde{\phi}^{(2)}+g_1^*\big(\widetilde{\phi}^{(0)},\widetilde{\phi}^{(1)}\big)\big)\widetilde{\mu}^{(0)}
\nonumber\\&\quad+2\nabla_x\partial_z\widetilde{\mu}^{(1)}\cdot\nabla d^{(0)}
+2\nabla_x\partial_z\widetilde{\mu}^{(0)}\cdot\nabla d^{(1)}
+\partial_z\widetilde{\mu}^{(1)}\Delta d^{(0)}+\partial_z\widetilde{\mu}^{(0)}\Delta d^{(1)}
\nonumber\\&\quad-\partial_z\theta\partial_td^{(0)}+
\Delta_x\widetilde{\mu}^{(0)}+\chi^{(0)} d^{(0)}\eta',\\
0&=-\p_z^2\widetilde{\phi}^{(2)}+f''(\theta)\widetilde{\phi}^{(2)}
+g_1\big(\widetilde{\phi}^{(0)},\widetilde{\phi}^{(1)}\big)
-2\nabla_x\partial_z\widetilde{\phi}^{(1)}\cdot\nabla d^{(0)}-2\nabla_x\partial_z\theta\cdot\nabla d^{(1)}
\nonumber\\&\quad-\partial_z\widetilde{\phi}^{(1)}\Delta d^{(0)}-\partial_z\theta\Delta d^{(1)}-\widetilde{\mu}^{(1)}-\Delta_x\widetilde{\phi}^{(0)}.
\end{align*}
In view of \eqref{partial zero},   the above two equations can be simplified as
\begin{subequations}
  \begin{align}
{\mathscr{L}\widetilde{\mu}^{(2)}}&=
-f'''(\theta)\widetilde{\phi}^{(2)}\widetilde{\mu}^{(0)}
+2\nabla_x\partial_z\widetilde{\mu}^{(1)}\cdot\nabla d^{(0)}
+2\nabla_x\partial_z\widetilde{\mu}^{(0)}\cdot\nabla d^{(1)}
\nonumber\\&\quad+\partial_z\widetilde{\mu}^{(1)}\Delta d^{(0)}+\partial_z\widetilde{\mu}^{(0)}\Delta d^{(1)}-\partial_z\theta\partial_td^{(0)}+
\Delta_x\widetilde{\mu}^{(0)}+\chi^{(0)} d^{(0)}\eta',\label{inequ-2-1-s}\\
{\mathscr{L}\widetilde{\phi}^{(2)}}&=\partial_z\theta\Delta d^{(1)}+\widetilde{\mu}^{(1)}
=\big(\Delta d^{(1)}+\mu_1\big)\theta'+{D^{(0)}}z\theta'(z).\label{inequ-2-2-s}
\end{align}
\end{subequations}
Recall \eqref{eq:1.2} that   $\mu_1$ shall be determined such that \eqref{inequ-2-2-s} fulfills \eqref{compatiblity-1}, i.e.
$\big(\Delta d^{(1)}+\mu_1\big)\sigma=0$,
where $\sigma\triangleq \int_{\mathbb{R}}(\theta')^2dz$.
This leads to the formula for $\mu_1$ and completes formula \eqref{eq:1.2}:
\begin{align}
\mu_1(x,t)=-\Delta d^{(1)},\qquad \widetilde{\mu}^{(1)}(z,x,t)&=D^{(0)}z\theta'(z)-\Delta d^{(1)}\theta'(z).\label{eq:muphi2}
\end{align} 
As a result, \eqref{inequ-2-2-s} is  simplified to
\begin{align}
{\mathscr{L}\widetilde{\phi}^{(2)}}={D^{(0)}}z\theta'(z).\label{inequ-2-2-s0}
\end{align}
Using $\int_{\mathbb{R}}z(\theta')^2dz=0$ and formula \eqref{odesolver}, we can solve (\ref{inequ-2-2-s0}) and yield
\begin{align}
\widetilde{\phi}^{(2)}(z,x,t)={D^{(0)}}\theta'(z) \alpha(z),~\text{with}~\alpha(z)\triangleq\int_{0}^{z}(\theta'(\zeta))^{-2}\int_{\zeta}^{+\infty}\tau(\theta'(\tau))^2d\tau d\zeta\label{expression of phi2}
\end{align}
being an odd function. This  implies   that $\widetilde{\phi}^{(2)}$ is odd with respect to $z$. On the other hand, $\chi^{(0)}$ is determined so that the right hand side of (\ref{inequ-2-1-s}) fulfills \eqref{compatiblity-1}, e.g.
\begin{align}
\chi^{(0)}d^{(0)}\sigma^{-1}\overline{\sigma}=\mathscr{G}_0d^{(0)},~\text{with}~\overline{\sigma}=\int_{\mathbb{R}}\eta'\theta' dz,\label{distance law-0}\\
\mathscr{G}_0d^{(0)}\triangleq \partial_td^{(0)}+\Delta^2 d^{(0)}-\Delta d^{(0)}{D^{(0)}}
-\nabla d^{(0)}\cdot\nabla{D^{(0)}}.
\end{align}
Note that we used the following formula which is due to   \eqref{commutator2} and \eqref{inequ-2-2-s0}:
\begin{align}
\int_{\mathbb{R}}f'''(\theta)\widetilde{\phi}^{(2)}(\theta')^2dz=\int_{\mathbb{R}}\partial_z\big({\mathscr{L}\widetilde{\phi}^{(2)}}\big)\theta'dz-\int_{\mathbb{R}} \mathscr{L}\(\partial_z\widetilde{\phi}^{(2)}\)\theta' dz=\tfrac{\sigma}{2}  D^{(0)}.\label{zero condition}
\end{align}
Combining   \eqref{distance law-0} and  \eqref{distance law} leads to the choice of $\chi^{(0)}$: 
\begin{eqnarray}\label{formula:chi0}
\chi^{(0)}\triangleq\left \{
\begin {array}{llr}
\sigma(\overline{\sigma})^{-1}\(\mathscr{G}_0d^{(0)}\)/d^{(0)}, \  &\forall (x,t)\in\Gamma^0(3\delta)\backslash\Gamma^0,\\
\sigma(\overline{\sigma})^{-1}\nabla\(\mathscr{G}_0d^{(0)}\)\cdot\nabla d^{(0)}, \  &\forall (x,t)\in\Gamma^0.
\end{array}
\right.
\end{eqnarray}
\begin{rmk}\label{deterd_0-re}
If we do not modify the equation  \eqref{newequ-1} into \eqref{newequ-1-m}, then we would require  the equation \eqref{distance law} to hold in $\Gamma^0(3\delta)$, which is not compatible with $|\nabla d^{(0)}|=1$ in general. 
\end{rmk}
The formula \eqref{distance law-0} reduces (\ref{inequ-2-1-s}) to
\begin{align}
\mathscr{L}\widetilde{\mu}^{(2)}&=
\big(f'''(\theta)(\theta')^2\alpha+z\theta''\big)\Delta d^{(0)}{D^{(0)}}
+\big(\theta'+2z\theta''\big)\nabla d^{(0)}\cdot\nabla{D^{(0)}}\nonumber\\&
-2\theta''{D^{(1)}}
+\chi^{(0)}d^{(0)}\eta'-
\sigma^{-1}\overline{\sigma}\chi^{(0)}d^{(0)}\theta',\label{inequ-2-1-ss}\\
D^{(1)}&=\nabla\Delta d^{(1)}\cdot\nabla d^{(0)}+\nabla\Delta d^{(0)}\cdot\nabla d^{(1)}+\Delta d^{(0)}\Delta d^{(1)}.\label{def:D0D1app2}
\end{align}
Note that  $D^{(1)}$ is consistent with   \eqref{def:D0D1}.
We can solve \eqref{inequ-2-1-ss} by employing  \eqref{odesolver},
\begin{align}
\widetilde{\mu}^{(2)}(z,x,t)=&\Delta d^{(0)}{D^{(0)}}\theta'(z)\gamma_1(z)
+\nabla d^{(0)}\cdot\nabla{D^{(0)}}\theta'(z)\gamma_2(z)
\nonumber\\&+{D^{(1)}}z\theta'
+\mu_2(x,t)\theta'
+\chi^{(0)}d^{(0)}\theta'(z)\gamma_3(z),\label{expression of mu2-0}
\end{align}
where  $\mu_2(x,t)$ is a smooth function which  will be determined by the $\varepsilon^3$-scale below, and $\gamma_1(z)$ and $\gamma_2(z)$ and  $\gamma_3(z)$ are three  even functions defined by
%\begin{equation}\label{mu12}
%\begin{split}
%\widetilde{\mu}^{(0)}(z,x,t)&=-\Delta d^{(0)}(x,t)\theta'(z),\\
%\widetilde{\mu}^{(1)}(z,x,t)&=-\Delta d^{(1)}(x,t)\theta'(z)+D^{(0)}(x,t)z\theta'(z),\\
%\widetilde{\mu}^{(2)}(z,x,t)&=\Delta d^{(0)} {D^{(0)}}\theta'(z)\gamma_1(z)
%+\nabla d^{(0)}\cdot\nabla {D^{(0)}}\theta'(z)\gamma_2(z)
%+ {D^{(1)}}z\theta'
%+\mu_2(x,t)\theta'\\&\qquad{+\big(1-\overline{\sigma}\big)\chi^{(0)}d^{(0)}\theta'(z)\gamma_3(z)},
%\end{split}
%\end{equation}
\begin{equation}\label{eq:2.7}
\begin{split}
&\gamma_1(z)=\int_{0}^{z}(\theta'(\zeta))^{-2}\int_{\zeta}^{+\infty}\theta'(\tau)\big(f'''(\theta)(\theta')^2\alpha+\tau\theta''\big)(\tau)d\tau d\zeta,
\quad\gamma_2(z)=-z^2/2,
\\&\gamma_3(z)=\int_{0}^{z}(\theta'(\zeta))^{-2}\int_{\zeta}^{+\infty}\theta'(\tau)\big(\eta'(\tau)-\sigma^{-1}\overline{\sigma}\theta'(\tau)\big)d\tau d\zeta.
\end{split}
\end{equation}
%%%%%%%%%%%%%%%%%%%%%%%%%%%%%%%%%%%%%%%%%%%%%%%%%%%%%%%%%%%%%%%%%%%%%%%%
\subsection{$\varepsilon^3$-scale}
Substituting (\ref{inex-1}) into   (\ref{newequ-2})-(\ref{newequ-1-m}), then  using (\ref{partial zero}) and collecting all terms of  $\varepsilon^3$-scale lead to 
\begin{subequations}
  \begin{align}
\mathscr{L}\widetilde{\mu}^{(3)}&=-f'''(\theta)\widetilde{\phi}^{(2)}\widetilde{\mu}^{(1)}
-f'''(\theta)\widetilde{\phi}^{(3)}\widetilde{\mu}^{(0)}+\big(\chi^{(0)}d^{(1)}+\chi^{(1)}d^{(0)}\big)
\eta'-\chi^{(0)}z\eta'
\nonumber\\&\quad+2\nabla_x\partial_z\widetilde{\mu}^{(2)}\cdot\nabla d^{(0)}
+2\nabla_x\partial_z\widetilde{\mu}^{(0)}\cdot\nabla d^{(2)}+2\nabla_x\partial_z\widetilde{\mu}^{(1)}\cdot\nabla d^{(1)}
\nonumber\\&\quad+\partial_z\widetilde{\mu}^{(2)}\Delta d^{(0)}+\partial_z\widetilde{\mu}^{(0)}\Delta d^{(2)}+\partial_z\widetilde{\mu}^{(1)}\Delta d^{(1)}
-\partial_z\theta\partial_td^{(1)}+\Delta_x\widetilde{\mu}^{(1)}
,\label{inequ-3-1}\\
{\mathscr{L}\widetilde{\phi}^{(3)}}&=
2\nabla_x\partial_z\widetilde{\phi}^{(2)}\cdot\nabla d^{(0)}
+\partial_z\widetilde{\phi}^{(2)}\Delta d^{(0)}+\partial_z\theta\Delta d^{(2)}+\widetilde{\mu}^{(2)}.\label{inequ-3-2}
\end{align}
\end{subequations}
We  determine    $\mu_2(x,t)$ in \eqref{expression of mu2-0} so that    (\ref{inequ-3-2}) satisfies (\ref{compatiblity-1}), i.e.
\begin{align}
\big(\Delta d^{(2)}+\mu_2\big)\sigma
&=-\(\nabla d^{(0)}\cdot\nabla{D^{(0)}}
+\tfrac{1}{2}\Delta d^{(0)}{D^{(0)}}\)\int_{\mathbb{R}}\int_{z}^{+\infty}\tau(\theta'(\tau))^2d\tau dz
\nonumber\\&\quad-\Delta d^{(0)}{D^{(0)}}\int_{\mathbb{R}}(\theta')^2\gamma_1(z)dz
-\nabla d^{(0)}\cdot\nabla{D^{(0)}}\int_{\mathbb{R}}(\theta')^2\gamma_2(z)dz
\nonumber\\&\quad-\chi^{(0)}d^{(0)}\int_{\mathbb{R}}
(\theta'(z))^2\gamma_3(z)dz.\label{expression of mu2-1}
\end{align}
To prove \eqref{expression of mu2-1}, it follows from \eqref{expression of phi2} and integration by parts that
  \begin{align}
&\int_{\mathbb{R}}\(2\nabla_x\partial_z\widetilde{\phi}^{(2)}\cdot\nabla d^{(0)}
+\partial_z\widetilde{\phi}^{(2)}\Delta d^{(0)}+\partial_z\theta\Delta d^{(2)}\)\theta'dz
\nonumber\\&=-\big(2\nabla  d^{(0)}\cdot\nabla D^{(0)}+D^{(0)}\Delta d^{(0)}\big)\int_{\mathbb{R}}\alpha\theta'\theta''dz+\Delta d^{(2)}\sigma
\nonumber\\&=\(\nabla d^{(0)}\cdot\nabla{D^{(0)}}
+\tfrac{1}{2}\Delta d^{(0)}{D^{(0)}}\)\int_{\mathbb{R}}\int_{z}^{+\infty}\tau(\theta'(\tau))^2d\tau dz+\Delta d^{(2)}\sigma.\label{nasty3}
\end{align}
This together with \eqref{expression of mu2-0} leads to \eqref{expression of mu2-1}. So we can use \eqref{expression of mu2-1} to rewrite \eqref{expression of mu2-0} as
\begin{align}
\widetilde{\mu}^{(2)}(z,x,t)&=-\Delta d^{(2)}(x,t)\theta'(z)+D^{(1)}(x,t)z\theta'(z)+\tilde{\Psi}^{(0)}(z,x,t),\label{expression of mu2}
\end{align}
where $\tilde{\Psi}^{(0)}$ only depends on $0$-order terms:
\begin{align}
\tilde{\Psi}^{(0)}&=\Delta d^{(0)}D^{(0)}\theta'\gamma_1+\nabla  d^{(0)}\cdot\nabla D^{(0)}\theta'\gamma_2
\nonumber\\&\quad- (2\sigma)^{-1}\(\int_{\mathbb{R}}\int_{z}^{+\infty}\tau(\theta'(\tau))^2d\tau dz\)\bigg(\Delta d^{(0)}D^{(0)}+2\nabla  d^{(0)}\cdot\nabla D^{(0)}\bigg)\theta'
\nonumber\\&\quad-\sigma^{-1}\bigg(\Delta d^{(0)}{D^{(0)}}\int_{\mathbb{R}}(\theta')^2\gamma_1(z)dz
+\nabla d^{(0)}\cdot\nabla{D^{(0)}}\int_{\mathbb{R}}(\theta')^2\gamma_2(z)dz\bigg)\theta'
\nonumber\\&\quad-\sigma^{-1}\bigg(\int_{\mathbb{R}}(\theta'(z))^2
\gamma_3(z)dz\bigg)\chi^{(0)}d^{(0)}\theta'.\label{psi0term}
\end{align}
Finally, applying Lemma \ref{ODEsolver-special case} to  (\ref{inequ-3-2}) yields  a solution $\widetilde{\phi}^{(3)}$: 
 \begin{lemma}
$\tilde{\Psi}^{(0)}$ satisfies  \eqref{decay-0} and the equation  (\ref{inequ-3-2}) has 
 a unique smooth solution $\widetilde{\phi}^{(3)}$ depending up to $1$-order terms,  satisfying  $\widetilde{\phi}^{(3)}|_{z=0}=0$ and  (\ref{decay-1}).
 \end{lemma}
$d^{(1)}$ is determined so that the right hand side of  (\ref{inequ-3-1}) fulfills \eqref{compatiblity-1}:
\begin{lemma}\label{deterd_1}
There exists $\Xi^{(0)}(x,t)$  depending on $0$-order terms such that
\begin{align}\label{d1 equation}
&\mathscr{G}_1d^{(1)}=\tfrac{\overline{\sigma}}{\sigma}\big(\chi^{(0)}d^{(1)}
+\chi^{(1)}d^{(0)}\big)+\Xi^{(0)}~\text{in}~\G^0(3\delta),\\
&\qquad \text{where}\qquad \mathscr{G}_1d^{(1)}\triangleq \partial_td^{(1)}+\Delta^2d^{(1)}-\sum\limits_{i=0,1}\(\nabla D^{(i)}\cdot\nabla d^{(1-i)}+  D^{(i)}\Delta d^{(1-i)}\).\label{d1 equation new}
\end{align}
\end{lemma}
\begin{proof}
We note that $f'''(\theta)(\theta'(z))^3\alpha(z) z$ is odd, so it follows from 
  \eqref{eq:muphi2}, \eqref{expression of phi2} and (\ref{zero condition}) that 
\begin{align}\label{similar-1}
-\int_{\mathbb{R}}f'''(\theta)\widetilde{\phi}^{(2)}\widetilde{\mu}^{(1)}\theta'dz&
=\Delta d^{(1)}\int_{\mathbb{R}}f'''(\theta)\widetilde{\phi}^{(2)}(\theta')^2 dz=\frac{\sigma}{2}D^{(0)}\Delta d^{(1)}.
\end{align}
Using \eqref{eq:muphi1}, \eqref{inequ-3-2} and \eqref{expression of phi2}, we can proceed in the same way as we obtain (\ref{zero condition}) and yield 
\begin{align*}
&-\int_{\mathbb{R}}f'''(\theta)\widetilde{\phi}^{(3)}\widetilde{\mu}^{(0)}\theta'dz=\Delta d^{(0)}\int_{\mathbb{R}}f'''(\theta)\widetilde{\phi}^{(3)}(\theta')^2dz
\nonumber\\&=\Delta d^{(0)}\int_{\mathbb{R}}\partial_z\(2\nabla_x\partial_z\widetilde{\phi}^{(2)}\cdot\nabla d^{(0)}
+\partial_z\widetilde{\phi}^{(2)}\Delta d^{(0)}+\partial_z\theta\Delta d^{(2)}+\widetilde{\mu}^{(2)}\)\theta'dz
\nonumber\\&=-\Delta d^{(0)}\int_{\mathbb{R}}\widetilde{\mu}^{(2)}\theta''dz.
\end{align*}
This combined with   \eqref{expression of mu2} leads to 
\begin{align}\label{similar-2}
&-\int_{\mathbb{R}}f'''(\theta)\widetilde{\phi}^{(3)}\widetilde{\mu}^{(0)}\theta'dz
\nonumber\\&=-\Delta d^{(0)}\int_{\mathbb{R}}\big(-\Delta d^{(2)}\theta'(z)+D^{(1)}z\theta'(z)+\tilde{\Psi}^{(0)}(z,x,t)\big)\theta''dz
\nonumber\\&=\frac{\sigma}{2}D^{(1)}\Delta d^{(0)}-\Delta d^{(0)}\int_{\mathbb{R}}\tilde{\Psi}^{(0)}(z,x,t)\theta''dz.
\end{align}

We continue treating the terms on the right hand side of  (\ref{inequ-3-1}). It follows from \eqref{eq:muphi1}, \eqref{eq:muphi2}, and \eqref{expression of mu2}  that 
\begin{align*}
&\int_{\mathbb{R}}\(2\nabla_x\partial_z\widetilde{\mu}^{(2)}\cdot\nabla d^{(0)}
+2\nabla_x\partial_z\widetilde{\mu}^{(0)}\cdot\nabla d^{(2)}+2\nabla_x\partial_z\widetilde{\mu}^{(1)}\cdot\nabla d^{(1)}\)\theta'dz
\nonumber\\&=\sigma\big(\nabla D^{(1)}\cdot\nabla d^{(0)}+\nabla D^{(0)}\cdot\nabla d^{(1)}\big)
-2\nabla d^{(0)}\cdot\int_{\mathbb{R}}\nabla_x\tilde{\Psi}^{(0)}(z,x,t)\theta''dz.
\end{align*}
Moreover, we have the following two identities:
 \begin{align}
&\int_{\mathbb{R}}\(\partial_z\widetilde{\mu}^{(2)}\Delta d^{(0)}+\partial_z\widetilde{\mu}^{(0)}\Delta d^{(2)}+\partial_z\widetilde{\mu}^{(1)}\Delta d^{(1)}\)\theta'dz
\nonumber\\&\qquad =\frac{\sigma}{2}\big(D^{(1)}\Delta d^{(0)}+D^{(0)}\Delta d^{(1)}\big)
-\Delta d^{(0)}\int_{\mathbb{R}}\tilde{\Psi}^{(0)}(z,x,t)\theta''dz,\\
&\int_{\mathbb{R}}
\big(-\partial_z\theta\partial_td^{(1)}+\Delta_x\widetilde{\mu}^{(1)}\big)\theta'dz=-\sigma\big(\partial_td^{(1)}+\Delta^2d^{(1)}\big).
\end{align}
Therefore, using the notation \eqref{d1 equation new}
, we deduce that $d^{(1)}$ satisfies \eqref{d1 equation} and   \begin{align*}
 \Xi^{(0)}=-\frac2{\sigma}\nabla d^{(0)}\cdot\int_{\mathbb{R}}\nabla_x\tilde{\Psi}^{(0)}(z,x,t)\theta''dz-\frac 2{\sigma}\Delta d^{(0)}\int_{\mathbb{R}}\tilde{\Psi}^{(0)}(z,x,t)\theta''dz.
\end{align*}
\end{proof}
To determine $d^{(1)}$ and $\chi^{(1)}$ so that   \eqref{d1 equation} holds, we need  the following result:
\begin{corol}\label{deterd_1-c}
The following equation about    $d^{(1)}$ has a local in time classical solution:
\begin{align}\label{equation of d1}
\mathscr{G}_1d^{(1)}=\sigma^{-1}\overline{\sigma}\chi^{(0)}d^{(1)}+\Xi^{(0)}~\text{on}~\Gamma^0.
\end{align}
Moreover,   to have \eqref{d1 equation} holds in $\G^0(3\delta)$,  $\chi^{(1)}$ can be    defined by 
{\begin{eqnarray}\label{formula:chi1}
\chi^{(1)}\triangleq\left \{
\begin {array}{ll}
 \sigma (\overline{\sigma})^{-1}\(\mathscr{G}_1d^{(1)}-\sigma^{-1}\overline{\sigma}\chi^{(0)}d^{(1)}-\Xi^{(0)}\)/d^{(0)}& \ \text{in}~\Gamma^0(3\delta)\backslash\Gamma^0,\\
\sigma\big(\overline{\sigma}\big)^{-1}\nabla\big(
\mathscr{G}_1d^{(1)}-\sigma^{-1}\overline{\sigma}\chi^{(0)}d^{(1)}-\Xi^{(0)}\big)\cdot\nabla d^{(0)}& \ \text{on}~\Gamma^0.
\end{array}
\right.
\end{eqnarray}}
\end{corol}
\begin{proof}
 Note that $d^{(1)}$ might not fulfill \eqref{equation of d1} in $\Gamma^0(3\delta)$.
Since  $\p_r d^{(1)}=0$ (see \eqref{equ:gradient d}), it suffices to determine $d^{(1)}$ on $\G^0$ and then extends  constantly in the normal direction. Using \eqref{commutator} we can convert  mixed  derivatives of $d^{(1)}$ into tangential ones.
This combined with  \eqref{def:D0D1app2} and   \eqref{equation of d1} yields \begin{equation}\label{solvability d0}
  \partial_td^{(1)}+\Delta^2d^{(1)}-(\nabla d^{(0)}\otimes \nabla d^{(0)}):\nabla^2 \Delta  d^{(1)} =\mathfrak{T}(d^{(1)}),
\end{equation}
where $\mathfrak{T}$ is a generic term that includes  at most third-order (tangential) derivatives of  $d^{(1)}$. Using   \eqref{commutator}, \eqref{gradient} and \eqref{eq:laplace1} yields  
\begin{equation}
 \begin{split}
    &\Delta^2d^{(1)}-(\nabla d^{(0)}\otimes \nabla d^{(0)}):\nabla^2 \Delta d^{(1)}\\
  =& \Div_{\G_0} (\nabla\Delta_{\G_0}  d^{(1)})=\Delta^2_{\G_0} d^{(1)}+\(\Div_{\G_0}\nn\)\p_r (\Delta_{\G_0}d^{(1)}).
 \end{split}
\end{equation}
So we can write \eqref{solvability d0} as (see \cite{Alikakos1994} for similar arguments)
\begin{equation}\label{solvability d0new}
  \partial_td^{(1)}+  \Delta_{\G_0}^2 d^{(1)} =\mathfrak{T}(d^{(1)}).
\end{equation}
 This is a surface evolutionary equation and has  a local in time smooth solution.
\end{proof}

\subsection{$\varepsilon^{K}$-scale:}
With Definition \ref{convention}, we set the following statements indexed by $K$:
\begin{subequations}\label{induction K}
\begin{align}
A_K&:\widetilde{\phi}^{(i)}~\text{depends on terms of order up to}~(i-2);~2\leq i\leq K+1,\label{nasty4}\\
B_K&:\widetilde{\mu}^{(i)}=-\Delta d^{(i)}\theta'+{D^{(i-1)}z\theta'}+\widetilde{\Psi}^{(i-2)}~\text{for}~2\leq i\leq K,\label{expression of mui}\\
C_K&:\widetilde{\mu}^{(K+1)}=\mu_{K+1}(x,t)\theta'+{D^{(K)}z\theta'}
+\tilde{\Psi}^{(K-1)},\label{expression of muk1-0}\\
D_K&:  (d^{(K)}, \chi^{(K)})~\text{depend on terms up to order $(K-1)$ through}\nonumber\\
&\qquad \mathscr{G}_K d^{(K)}=\sigma^{-1}\overline{\sigma}\big(\chi^{(0)}d^{(K)}+\chi^{(K)}d^{(0)}\big)+\Xi^{(K-1)},\label{equation of dk}
\end{align}
\end{subequations}
where    $D^{(i)}$ is defined by \eqref{def:D0D1},   $\tilde{\Psi}^{(K)}$ satisfy the decay property (\ref{decay-1}), and 
\begin{align}\label{def d_k}
&\mathscr{G}_K d^{(K)}\triangleq \partial_td^{(K)}+\Delta^2d^{(K)}-\sum_{\ell=0,K}\(\nabla D^{(\ell)}\cdot\nabla d^{(K-\ell)}+D^{(\ell)}\Delta d^{(K-\ell)}\),\\
&\chi^{(K)}\triangleq \left \{
\begin {array}{ll}
\sigma\overline{\sigma}^{-1}\(\mathscr{G}_K d^{(K)}-\sigma^{-1}\overline{\sigma}\chi^{(0)}d^{(K)}-\Xi^{(K-1)}\)/d^{(0)}& \ \text{in}~\Gamma^0(3\delta)\backslash\Gamma^0,\\
\sigma\overline{\sigma}^{-1}\nabla\big(\mathscr{G}_K d^{(K)}-\sigma^{-1}\overline{\sigma}\chi^{(0)}d^{(K)}-\Xi^{(K-1)}\big)\cdot\nabla d^{(0)}& \ \text{in}~\Gamma^0.
\end{array}
\right.\label{formula:chik}
\end{align}
\begin{lemma}
The statements $(A_1,B_1,C_1)$ and $(A_2,B_2,C_2,D_1)$ are valid.
\end{lemma}
\begin{proof}
Recall the results in previous subsections.
Using  $d^{(1)}$ we can determine   $\widetilde{\mu}^{(1)}$ through   (\ref{eq:muphi2}). Using $d^{(2)}$ determined by \eqref{def d_k} with $K=2$,  we obtain $\widetilde{\mu}^{(2)}$ by \eqref{expression of mu2} and   $\widetilde{\phi}^{(3)}$  by solving (\ref{inequ-3-2}). Finally we can rewrite  \eqref{inequ-3-1} as 
\begin{align}\label{nasty5}
\mathscr{L}\widetilde{\mu}^{(3)}=-2\theta''D^{(2)}+\tilde{\Psi}^{(1)},~\text{where}~D^{(2)}=\sum\limits_{0\leq \ell\leq 2}\(\nabla\Delta d^{(\ell)}\cdot\nabla d^{(2-\ell)}+\tfrac 12\Delta d^{(\ell)}\Delta d^{(2-\ell)}\),
\end{align}
and $\tilde{\Psi}^{(1)}$  satisfies \eqref{decay-0}. Applying \eqref{odesolver}   yields 
\begin{align}
\widetilde{\mu}^{(3)}(z,x,t)=\mu_3(x,t)\theta'(z)+D^{(2)}(x,t)z\theta'(z)+\tilde{\Psi}^{(1)}(z,x,t).  \label{expression of mu3}
\end{align}
where  $\tilde{\Psi}^{(1)}$ satisfies \eqref{decay-0}, and $\mu_3(x,t)$ shall   be determined by the $\varepsilon^4$-scale.
\end{proof}
{\bf We argue by   induction on $K$}. Assuming  $\(A_K,B_K,C_K,D_{K-1}\)$.
We substitute (\ref{inex-1})  into   (\ref{newequ-2})-(\ref{newequ-1-m})  and use \eqref{f expansion} and  \eqref{partial zero} to sort  all terms of  $\varepsilon^{K+2}$-scale:
\begin{subequations}
  \begin{align}
\mathscr{L}\widetilde{\mu}^{(K+2)}
=&-\sum\limits_{2\leq i\leq K+2}\(f'''(\theta)\widetilde{\phi}^{(i)}+g_{i-1}^*\big(\widetilde{\phi}^{(0)},\cdots,
\widetilde{\phi}^{(i-1)}\big)\)\widetilde{\mu}^{(K+2-i)}
\nonumber\\&+2\sum\limits_{0\leq i\leq K+1}\nabla_x\partial_z\widetilde{\mu}^{(i)}\cdot\nabla d^{(K+1-i)}
+\sum\limits_{0\leq i\leq K+1}\partial_z\widetilde{\mu}^{(i)}\Delta d^{(K+1-i)}
\nonumber\\&-\sum\limits_{1\leq i\leq K}\partial_z\widetilde{\phi}^{(i)}\partial_td^{(K-i)}-\partial_z\theta\partial_td^{(K)}
+\Delta_x\widetilde{\mu}^{(K)}-\partial_t\widetilde{\phi}^{(K-1)}
\nonumber\\&+\Big(\chi^{(0)}d^{(K)}+\sum\limits_{1\leq i\leq K-1}\chi^{(i)}
d^{(K-i)}+\chi^{(K)}d^{(0)}\Big)\eta'-\chi^{(K-1)}z\eta',\label{inequ-k2-1}\\
\mathscr{L}\widetilde{\phi}^{(K+2)}=&-g_{K+1}^*\(\widetilde{\phi}^{(0)},\cdots,
\widetilde{\phi}^{(K+1)}\)+2\sum\limits_{2\leq i\leq K+1}\nabla_x\partial_z\widetilde{\phi}^{(i)}\cdot\nabla d^{(K+1-i)}
\nonumber\\&\quad+\theta'\Delta d^{(K+1)}+\sum\limits_{2\leq i\leq K+1} \partial_z\widetilde{\phi}^{(i)}\cdot\Delta d^{(K+1-i)}+\widetilde{\mu}^{(K+1)}+\Delta_x\widetilde{\phi}^{(K)}.
\label{inequ-k2-2}
\end{align}
\end{subequations}
Using $A_K$ and $C_K$, 
 we can write \eqref{inequ-k2-2} as
\begin{equation}\label{solving phi_K}
{\mathscr{L}\widetilde{\phi}^{(K+2)}=\theta'\Delta  d^{(K+1)}+\mu_{K+1}(x,t)\theta'+D^{(K)} z\theta'+\widetilde{\Psi}^{(K-1)}.}\end{equation}
To fulfill  the compatibility condition  \eqref{compatiblity-1}, we   choose   
$\mu_{K+1}=-\Delta d^{(K+1)}+{\Xi^{(K-1)}}$.
 This together with $C_K$ implies $B_{K+1}$,
 and reduce  \eqref{solving phi_K} to the following equation, which leads to $A_{K+1}$:
\begin{equation}\label{phiK+2}
{\mathscr{L}\widetilde{\phi}^{(K+2)}=D^{(K)} z\theta'+\widetilde{\Psi}^{(K-1)}.}
\end{equation}
 
\begin{proposition}
The equation \eqref{inequ-k2-1} can be written as
\begin{equation}\label{equ mu K+2}
\mathscr{L}\widetilde{\mu}^{(K+2)}={-2\sum_{\ell=0,K+1} \nabla\Delta d^{(\ell)}\cdot\nabla d^{(K+1-\ell)}\theta''}-2\Delta d^{(0)}\Delta d^{(K+1)}\theta''+\widetilde{\Psi}^{(K)},
\end{equation}
and its compatibility condition
 is guaranteed by $D_K$. 
 \end{proposition}
\begin{proof}
We consider  the right hand side of \eqref{inequ-k2-1}. Using \eqref{partial zero}, \eqref{nasty4} and \eqref{expression of mui},
\begin{align*}
&-\sum\limits_{2\leq i\leq K+2}\(f'''(\theta)\widetilde{\phi}^{(i)}+
g_{i-1}^*\big(\widetilde{\phi}^{(0)},\cdots,
\widetilde{\phi}^{(i-1)}\big)\)\widetilde{\mu}^{(K+2-i)}
\nonumber\\&=-f'''(\theta)\widetilde{\phi}^{(2)}\widetilde{\mu}
^{(K)}-f'''(\theta)\widetilde{\phi}^{(K+2)}\widetilde{\mu}^{(0)}
-g_{K+1}^*\big(\widetilde{\phi}^{(0)},\cdots,
\widetilde{\phi}^{(K+1)}\big)
\nonumber\\&\qquad-\sum\limits_{3\leq i\leq K+1}
\(f'''(\theta)\widetilde{\phi}^{(i)}+g_{i-1}^*\big(\widetilde{\phi}^{(0)},\cdots,
\widetilde{\phi}^{(i-1)}\big)\)\widetilde{\mu}^{(K+2-i)}
\nonumber\\&=\Delta d^{(K)}f'''(\theta)\widetilde{\phi}^{(2)}\theta'
-f'''(\theta)
\widetilde{\phi}^{(K+2)}\widetilde{\mu}^{(0)}+\widetilde{\Psi}^{(K-1)}
\end{align*}
In a similar way,
\begin{align*}
&2\sum\limits_{0\leq i\leq K+1}\nabla_x\partial_z\widetilde{\mu}^{(i)}\cdot\nabla d^{(K+1-i)}\nonumber\\&=2\nabla_x\partial_z\widetilde{\mu}^{(0)}\cdot\nabla d^{(K+1)}+2\nabla_x\partial_z\widetilde{\mu}^{(1)}\cdot\nabla d^{(K)}+2\nabla_x\partial_z\widetilde{\mu}^{(K)}\cdot\nabla d^{(1)}
\nonumber\\&\quad+2\nabla_x\partial_z\widetilde{\mu}^{(K+1)}\cdot\nabla d^{(0)}+2\sum\limits_{2\leq i\leq K-1}\nabla_x\partial_z\widetilde{\mu}^{(i)}\cdot\nabla d^{(K+1-i)}
\nonumber\\&{=-2\sum_{\ell=0,1,K,K+1} \nabla\Delta d^{(\ell)}\cdot\nabla d^{(K+1-\ell)}\theta''}
\nonumber\\&\quad+2\(\nabla  D^{(K)}\cdot\nabla d^{(0)}+\nabla D^{(0)}\cdot\nabla d^{(K)}\)(z\theta')'+\widetilde{\Psi}^{(K-1)},
\end{align*}
\begin{align*}
&\quad \sum\limits_{0\leq i\leq K+1}\partial_z\widetilde{\mu}^{(i)}\Delta d^{(K+1-i)}
\nonumber\\&=\partial_z\widetilde{\mu}^{(0)}\Delta d^{(K+1)}+\partial_z\widetilde{\mu}^{(1)}\Delta d^{(K)}+\sum\limits_{2\leq i\leq K-1}\partial_z\widetilde{\mu}^{(i)}\Delta d^{(K+1-i)}
+\partial_z\widetilde{\mu}^{(K)}\Delta d^{(1)}+\partial_z\widetilde{\mu}^{(K+1)}\Delta d^{(0)}
\nonumber\\&=-2(\Delta d^{(0)}\Delta d^{(K+1)}+\Delta d^{(1)}\Delta d^{(K)})\theta''+(D^{(0)}\Delta d^{(K)}+D^{(K)}\Delta d^{(0)})(z\theta')'+\widetilde{\Psi}^{(K-1)}.
\end{align*}
Finally,
\begin{align*}
&-\sum\limits_{1\leq i\leq K}\partial_z\widetilde{\phi}^{(i)}\partial_td^{(K-i)}-\partial_z\theta\partial_td^{(K)}
+\Delta_x\widetilde{\mu}^{(K)}-\partial_t\widetilde{\phi}^{(K-1)}\\&\qquad
 +\big(\chi^{(0)}d^{(K)}+\sum\limits_{1\leq i\leq K-1}\chi^{(i)}d^{(K-i)}
+\chi^{(K)}d^{(0)}\big)\eta'-\chi^{(K-1)}z\eta'
\\&=-\big(\partial_td^{(K)}+\Delta^2d^{(K)}\big)\theta'+\eta' \big(\chi^{(0)}d^{(K)}+\chi^{(K)}d^{(0)}\big)+\widetilde{\Psi}^{(K-1)}.
\end{align*}
The above four results imply \eqref{equ mu K+2}. They also imply the  compatibility condition
\begin{align}
&\sigma\big(\partial_td^{(K)}+\Delta^2d^{(K)}\big)\nonumber\\
&=\Delta d^{(K)}\int_{\mathbb{R}}f'''(\theta)\widetilde{\phi}^{(2)}
(\theta')^2dz-\int_{\mathbb{R}}f'''(\theta)\widetilde{\phi}
^{(K+2)}\widetilde{\mu}^{(0)}\theta'dz+\int_{\mathbb{R}}\widetilde{\Psi}^{(K-1)}\theta' dz
\nonumber\\&\quad+\sigma\(\nabla  D^{(K)}\cdot\nabla d^{(0)}+\nabla D^{(0)}\cdot\nabla d^{(K)}+\tfrac{1}{2}\big(D^{(0)}\Delta d^{(K)}+D^{(K)}\Delta d^{(0)}\big)\)
.\label{compatibility-muk2}
\end{align}
It remains to calculate the first two terms on the right hand side of \eqref{compatibility-muk2}. Using \eqref{zero condition} yields
\begin{align}
\Delta d^{(K)}\int_{\mathbb{R}}f'''(\theta)\widetilde{\phi}^{(2)}(\theta')^2dz
=\frac{\sigma}{2}D^{(0)}\Delta d^{(K)}.\label{undetermined-k2}
\end{align}
With the aid  of \eqref{eq:muphi1} and \eqref{commutator2} we have
\begin{align*}
-\int_{\mathbb{R}}f'''(\theta)\widetilde{\phi}^{(K+2)}\widetilde{\mu}^{(0)}\theta'dz&=\Delta d^{(0)}\int_{\mathbb{R}}f'''(\theta)\widetilde{\phi}^{(K+2)}(\theta')^2dz
 \nonumber\\&=\Delta d^{(0)}\int_{\mathbb{R}}\(\partial_z\big(\mathscr{L}\widetilde{\phi}^{(K+2)})\theta'-\big(\mathscr{L}\partial_z\widetilde{\phi}^{(K+2)}\big)\theta'\)dz
 \nonumber\\&=-\Delta d^{(0)}\int_{\mathbb{R}}\mathscr{L}\widetilde{\phi}^{(K+2)}\theta''dz.
\end{align*}
In view of  \eqref{phiK+2},
  the above two formulas together lead to
\begin{align*}
-\int_{\mathbb{R}}f'''(\theta)\widetilde{\phi}^{(K+2)}
\widetilde{\mu}^{(0)}\theta'dz&=-\Delta d^{(0)}\int_{\mathbb{R}}\(D^{(K)}z\theta'+\tilde{\Psi}^{(K-1)}\)\theta''dz
=\frac{\sigma}{2}D^{(K)}\Delta d^{(0)}+\Xi^{(K-1)}.
\end{align*}
Substituting  \eqref{undetermined-k2} and the above formula into \eqref{compatibility-muk2} leads to $D_K$.
\end{proof}
Using \eqref{def:D0D1}, we can write \eqref{equ mu K+2}
by
\begin{equation}
\mathscr{L}\widetilde{\mu}^{(K+2)}=-2D^{(K+1)}\theta''+\tilde{\Psi}^{(K)}(z,x,t). 
\end{equation}
Applying Lemma \ref{ODEsolver-special case} to the above equation implies $C_{K+1}$.  To conclude $D_K$, we need:
\begin{corol}\label{deterd_k-c}
The following equation about    $d^{(K)}$ has a local in time classical solution
%\footnote{If we do not modify the equation  \eqref{newequ-1} into \eqref{newequ-1-m}, then we would require the equation 
%$\mathscr{G}_k  d^{(k)}=\Xi^{(k-1)}$ to hold in  $\Gamma^0(3\delta)$, which is not necessarily compatible with  \eqref{equ:gradient d} in $\Gamma^0(3\delta)$.}
\begin{align}\label{equation of dk new}
\mathscr{G}_K d^{(K)}=\sigma^{-1} \overline{\sigma}\chi^{(0)}d^{(K)}+\Xi^{(K-1)}~\text{on}~\G^0.
\end{align}
Moreover, to have $D_K$ holds   in $\G^0(3\delta)$,  $\chi^{(K)}$ can be    defined by \eqref{formula:chik}
\end{corol}
Note that the local in time solution follows from the same argument for Corollary \ref{deterd_1-c}. {\bf So we have shown $(A_{K+1}, B_{K+1},C_{K+1}, D_K)$ and  the induction for  \eqref{induction K} is completed}. 

\begin{proposition}\label{endprop1}
Assume (\ref{interface law-geometric1}) has a smooth solution $\Gamma^0$ within $[0,T]$, starting from a smooth closed hypersurface $\Gamma^0_0\subset\mathbb{R}^N$, and let $d^{(0)}$  be the signed-distance, defined in  $\Gamma^0(3\delta)$. Then we can construct the inner expansion with Ansatz \eqref{inex-1}   so that for $i\geq 1$
\begin{align}\label{matching condition}
 D_x^\alpha D_t^\beta D_z^\gamma\widetilde{\phi}^{(i)}(z,x,t)=O(e^{-C|z|}),D_x^\alpha D_t^\beta D_z^\gamma\widetilde{\mu}^{(i)}(z,x,t)=O(e^{-C|z|}),
\end{align}
 as $z\rightarrow\pm\infty$
for  $(x,t)\in\Gamma^0(3\delta)$ and $0\leq\alpha,\beta,\gamma\leq 2$. Moreover,
  \begin{align}
\hat{\phi}^I_a(x,t)=\sum_{0\leq i\leq k}\varepsilon^i\tilde{\phi}^{(i)}(z,x,t)\big|_{z=
\frac{ d^{[k]}(x,t) }{\varepsilon}},~\hat{\mu}^I_a(x,t)=\sum_{0\leq i\leq k}\varepsilon^i\tilde{\mu}^{(i)}(z,x,t)\big|_{z=\frac{d^{[k]}(x,t)}{\varepsilon}},\label{inn-01}
\end{align}
satisfies for  $(x,t)\in\Gamma^0(3\delta)$
\begin{align}
\e^3\p_t\hat{\phi}_a^I&= \e^2\Delta \hat{\mu}_a^I-f''(\hat{\phi}_a^I)\hat{\mu}_a^I
+O(\varepsilon^{k}),\label{newapprequ-5}
\\ \e\hat{\mu}_a^I&=-\varepsilon^2\Delta \hat{\phi}_a^I+f'(\hat{\phi}_a^I) +O(\varepsilon^{k}).\label{newapprequ-6}
\end{align}
\end{proposition}
\begin{proof}
 It follows from \eqref{inex-1}  and  chain-rule that
\begin{align*}
&-\(\e^3\p_t\hat{\phi}^I_a-\e^2\Delta \hat{\mu}^I_a+f''(\hat{\phi}^I_a)\hat{\mu}^I_a\)
\nonumber\\ =&\p_z^2\widetilde{\mu}^\e\big|\nabla  d^{[k]}\big|^2-f''(\widetilde{\phi}^\e)\widetilde{\mu}^\e+2\varepsilon\nabla
\partial_z\widetilde{\mu}^\e\cdot\nabla d^{[k]}\nonumber\\&+\varepsilon\partial_z\widetilde{\mu}^\e\Delta d^{[k]}-
\varepsilon^2\partial_z\widetilde{\phi}^\e\partial_td^{[k]}+\varepsilon^2
\Delta_x\widetilde{\mu}^\e-\e^3 \p_t \widetilde{\phi}^\e,\qquad \text{with}~z=d^{[k]}(x,t)/\e.\end{align*}
If we replace $d_\e$ by $d^{[k]}$ in  \eqref{newequ-1}, and compare it with the above  formula, then we arrive at  \eqref{newapprequ-5}. In a similar way we can show \eqref{newapprequ-6} and the details are omitted.
  \end{proof}

\noindent{\it Acknowledgements}. M. Fei is partially supported by NSF of China under Grant 11871075. Y. Liu is partially supported by NSF of China under Grant  11971314.

\noindent {\bf Conflict of interest} The authors declare that they have no conflict of interest.

%Thanks to the outer expansions, inner expansions, inner-outer matching conditions (\ref{matching condition}) and (\ref{app distance}),
% we easily find that  $(\phi_a,\mu_a)$ defined in (\ref{definition app}) satisfies (\ref{app model}) and then   Theorem \ref{spectral theorem} and Proposition \ref{construction of app solution} are proved.

% \bibliographystyle{abbrv}
%\bibliography{cfs}

\begin{thebibliography}{10}

\bibitem{abels2016sharp}
H.~Abels and Y.~Liu.
\newblock Sharp interface limit for a {S}tokes/{A}llen--{C}ahn system.
\newblock {\em Arch. Ration. Mech. Anal.}, 229(1):417--502, 2018.

\bibitem{Alikakos1994}
N.~D. Alikakos, P.~W. Bates, and X.~Chen.
\newblock Convergence of the {C}ahn-{H}illiard equation to the {H}ele-{S}haw
  model.
\newblock {\em Arch. Ration. Mech. Anal.}, 128(2):165--205, 1994.

\bibitem{MR3155251}
G.~Bellettini.
\newblock {\em Lecture notes on mean curvature flow, barriers and singular
  perturbations}, volume~12 of {\em Appunti. Scuola Normale Superiore di Pisa
  (Nuova Serie) [Lecture Notes. Scuola Normale Superiore di Pisa (New
  Series)]}.
\newblock Edizioni della Normale, Pisa, 2013.

\bibitem{bellettini1993approssimazione}
G.~Bellettini and M.~Paolini.
\newblock Approssimazione variazionale di funzionali con curvatura.
\newblock {\em Seminario di Analisi Matematica, Dipartimento di Matematica
  dell�Universita di Bologna}, pages 87--97, 1993.

\bibitem{MR3383330}
E.~Bretin, S.~Masnou, and E.~Oudet.
\newblock Phase-field approximations of the {W}illmore functional and flow.
\newblock {\em Numer. Math.}, 131(1):115--171, 2015.

\bibitem{MR2188465}
E.~A. Carlen, M.~C. Carvalho, and E.~Orlandi.
\newblock Approximate solutions of the {C}ahn-{H}illiard equation via
  corrections to the {M}ullins-{S}ekerka motion.
\newblock {\em Arch. Ration. Mech. Anal.}, 178(1):1--55, 2005.

\bibitem{chen1994spectrum}
X.~Chen.
\newblock Spectrum for the {Allen}-{C}ahn, {C}ahn-{H}illard, and phase-field
  equations for generic interfaces.
\newblock {\em Comm. Partial Differential Equations}, 19(7-8):1371--1395, 1994.

\bibitem{MR1425577}
X.~Chen.
\newblock Global asymptotic limit of solutions of the {C}ahn-{H}illiard
  equation.
\newblock {\em J. Differential Geom.}, 44(2):262--311, 1996.

\bibitem{chen2009mass}
X.~Chen, D.~Hilhorst, and E.~Logak.
\newblock Mass conserved {A}llen-{C}ahn equation and volume preserving mean
  curvature flow.
\newblock {\em Interfaces Free Bound.}, 12(4):527--549, 2010.

\bibitem{MR2846014}
P.~Colli and P.~Lauren\c{c}ot.
\newblock A phase-field approximation of the {W}illmore flow with volume
  constraint.
\newblock {\em Interfaces Free Bound.}, 13(3):341--351, 2011.

\bibitem{MR3023428}
P.~Colli and P.~Lauren\c{c}ot.
\newblock A phase-field approximation of the {W}illmore flow with volume and
  area constraints.
\newblock {\em SIAM J. Math. Anal.}, 44(6):3734--3754, 2012.

\bibitem{de1995geometrical}
P.~De~Mottoni and M.~Schatzman.
\newblock Geometrical evolution of developed interfaces.
\newblock {\em Trans. Amer. Math. Soc.}, 347(5):1533--1589, 1995.

\bibitem{MR3618120}
M.~del Pino and J.~Wei.
\newblock An introduction to the finite and infinite dimensional reduction
  methods.
\newblock In {\em Geometric analysis around scalar curvatures}, volume~31 of
  {\em Lect. Notes Ser. Inst. Math. Sci. Natl. Univ. Singap.}, pages 35--118.
  World Sci. Publ., Hackensack, NJ, 2016.

\bibitem{MR1301070}
M.~P. {D}o Carmo.
\newblock {\em Differential forms and applications}.
\newblock Universitext. Springer-Verlag, Berlin, 1994.
\newblock Translated from the 1971 Portuguese original.

\bibitem{MR2328722}
Q.~Du, M.~Li, and C.~Liu.
\newblock Analysis of a phase field {N}avier-{S}tokes vesicle-fluid interaction
  model.
\newblock {\em Discrete Contin. Dyn. Syst. Ser. B}, 8(3):539--556, 2007.

\bibitem{MR2531189}
Q.~Du, C.~Liu, R.~Ryham, and X.~Wang.
\newblock Energetic variational approaches in modeling vesicle and fluid
  interactions.
\newblock {\em Phys. D}, 238(9-10):923--930, 2009.

\bibitem{MR2062909}
Q.~Du, C.~Liu, and X.~Wang.
\newblock A phase field approach in the numerical study of the elastic bending
  energy for vesicle membranes.
\newblock {\em J. Comput. Phys.}, 198(2):450--468, 2004.

\bibitem{MR1897710}
G.~Dziuk, E.~Kuwert, and R.~Sch\"{a}tzle.
\newblock Evolution of elastic curves in {$\Bbb R^n$}: existence and
  computation.
\newblock {\em SIAM J. Math. Anal.}, 33(5):1228--1245, 2002.

\bibitem{Evans2010PDE}
L.~C. Evans.
\newblock {\em Partial differential equations}, volume~19 of {\em Graduate
  Studies in Mathematics}.
\newblock American Mathematical Society, Providence, RI, second edition, 2010.

\bibitem{MR3409135}
L.~C. Evans and R.~F. Gariepy.
\newblock {\em Measure theory and fine properties of functions}.
\newblock Textbooks in Mathematics. CRC Press, Boca Raton, FL, revised edition,
  2015.

\bibitem{fei2018isotropic}
M.~Fei, W.~Wang, P.~Zhang, and Z.~Zhang.
\newblock On the isotropic--nematic phase transition for the liquid crystal.
\newblock {\em Peking Mathematical Journal}, pages 1(2):141--219, 2018.

\bibitem{helfrich1973elastic}
W.~Helfrich.
\newblock Elastic properties of lipid bilayers: theory and possible
  experiments.
\newblock {\em Zeitschrift f{\"u}r Naturforschung C}, 28(11-12):693--703, 1973.

\bibitem{MR2119722}
E.~Kuwert and R.~Sch\"{a}tzle.
\newblock Removability of point singularities of {W}illmore surfaces.
\newblock {\em Ann. of Math. (2)}, 160(1):315--357, 2004.

\bibitem{MR2882586}
E.~Kuwert and R.~Sch\"{a}tzle.
\newblock The {W}illmore functional.
\newblock In {\em Topics in modern regularity theory}, volume~13 of {\em CRM
  Series}, pages 1--115. Ed. Norm., Pisa, 2012.

\bibitem{MR2402921}
N.~Q. Le.
\newblock A gamma-convergence approach to the {C}ahn-{H}illiard equation.
\newblock {\em Calc. Var. Partial Differential Equations}, 32(4):499--522,
  2008.

\bibitem{MR1757511}
P.~Loreti and R.~March.
\newblock Propagation of fronts in a nonlinear fourth order equation.
\newblock {\em European J. Appl. Math.}, 11(2):203--213, 2000.

\bibitem{MR3590375}
N.~Masmoudi and F.~Rousset.
\newblock Uniform regularity and vanishing viscosity limit for the free surface
  {N}avier-{S}tokes equations.
\newblock {\em Arch. Ration. Mech. Anal.}, 223(1):301--417, 2017.

\bibitem{MR2151414}
G.~M\'{e}tivier.
\newblock {\em Small viscosity and boundary layer methods}.
\newblock Modeling and Simulation in Science, Engineering and Technology.
  Birkh\"{a}user Boston, Inc., Boston, MA, 2004.
\newblock Theory, stability analysis, and applications.

\bibitem{ou1999geometric}
Z.-C. Ou-Yang, J.-X. Liu, Y.-Z. Xie, and X.~Yu-Zhang.
\newblock {\em Geometric methods in the elastic theory of membranes in liquid
  crystal phases}, volume~2.
\newblock World Scientific, 1999.

\bibitem{MR2430975}
T.~Rivi\`ere.
\newblock Analysis aspects of {W}illmore surfaces.
\newblock {\em Invent. Math.}, 174(1):1--45, 2008.

\bibitem{MR2253464}
M.~R\"{o}ger and R.~Sch\"{a}tzle.
\newblock On a modified conjecture of {D}e {G}iorgi.
\newblock {\em Math. Z.}, 254(4):675--714, 2006.

\bibitem{seifert1997configurations}
U.~Seifert.
\newblock Configurations of fluid membranes and vesicles.
\newblock {\em Advances in physics}, 46(1):13--137, 1997.

\bibitem{wang2008asymptotic}
X.~Wang.
\newblock Asymptotic analysis of phase field formulations of bending elasticity
  models.
\newblock {\em SIAM J. Math. Anal.}, 39(5):1367--1401, 2008.

\bibitem{MR3032974}
H.~Wu and X.~Xu.
\newblock Strong solutions, global regularity, and stability of a hydrodynamic
  system modeling vesicle and fluid interactions.
\newblock {\em SIAM J. Math. Anal.}, 45(1):181--214, 2013.

\end{thebibliography}

\end{document}